\documentclass{article}
\newskip\stdskip
\usepackage{amscd}
\usepackage{amsfonts}
\usepackage{amsmath}
\usepackage{amssymb}
\usepackage{amstext}
\usepackage{amsthm}
\usepackage{graphicx}
\usepackage{psfrag}
\usepackage[all]{xy}
\usepackage{hyperref}
\usepackage{url}
\usepackage{color}
\usepackage[toc,page]{appendix}
\usepackage{overpic}
\usepackage{authblk}

\newcommand{\C}{\mathbb C}
\newcommand{\R}{\mathbb R}
\newcommand{\N}{\mathbb N}

\newcommand{\Z}{\mathbb Z}

\newcommand{\op}{\operatorname}
\newcommand{\bs}{\boldsymbol}

\newtheorem{thm}{Theorem}[section]
\newtheorem{lemma}[thm]{Lemma}
\newtheorem{cor}[thm]{Corollary}
\newtheorem{prop}[thm]{Proposition}
\newtheorem{conj}[thm]{Conjecture}

\newtheorem{assumption}[thm]{Assumption}
\newtheorem{conv}[thm]{Convention}

\theoremstyle{definition}
\newtheorem{dfn}[thm]{Definition}
\newtheorem{rem}[thm]{Remark}

\DeclareMathAlphabet{\mathpzc}{OT1}{pzc}{m}{it}
\newcommand{\hh}{\mathpzc{h}}
\begin{document}
\title{An exposition of the equivalence of Heegaard Floer homology and embedded contact homology}
\author[1]{Vincent Colin}
\author[2]{Paolo Ghiggini}
\author[3]{Ko Honda}
\affil[1]{Laboratoire de Math\'ematiques Jean Leray, Universit\'e de Nantes, \texttt{vincent.colin@univ-nantes.fr}}
\affil[2]{Laboratoire de Math\'ematiques Jean Leray, Universit\'e de Nantes and CNRS, \texttt{paolo.ghiggini@univ-nantes.fr}}
\affil[3]{University of California, Los Angeles, Los Angeles, CA 90095, \texttt{honda@math.ucla.edu}}

\maketitle
\begin{abstract}\noindent
This article is a survey on the authors' proof of the isomorphism between Heegaard Floer homology and embedded contact homology appeared in \cite{binding, I, II,  III}. 
\end{abstract}
\section{Introduction}
In the search for gluing formulas for gauge theoretic invariants of four-manifolds (Donaldson's invariants and Seiberg-Witten invariants) it was soon realised that the object which should be associated to a three-manifold is a Floer homology. The first to be constructed was {\em instanton Floer homology} in \cite{floer} by Floer himself, although only for integer homology spheres. On the other hand, the development of a Seiberg-Witten Floer homology languished for several years, until Kronheimer and Mrowka defined {\em monopole Floer homology} in \cite{KMbook}. 

In the meantime the deep connection between gauge theory and symplectic geometry became clear. For example, Atiyah proposed a conjectural reinterpretation of instanton Floer homology as a Lagrangian intersection Floer homology, now known as the ``Atiyah-Floer conjecture'', in \cite{A} and Taubes proved in \cite{SW=GrBook} that the Seiberg-Witten invariants of a symplectic four-manifold are equivalent to a count of $J$-holomorphic curves which, in most cases, are embedded. This result is usually called ``SW=Gr''.

Ozsv\'ath and Szab\'o defined {\em Heegaard Floer homology} in \cite{OSz1, OSz2} motivated by the Atiyah-Floer conjecture. Instead of the original definition, in this article we will use an equivalent one given by Lipshitz in \cite{Li}. The starting point of its construction is a {\em pointed Heegaard diagram} $(\Sigma, \boldsymbol{\alpha}, \boldsymbol{\beta}, z)$ which describes a three-manifold $M$. Here $\Sigma$ is a Heegaard surface of genus $h>0$ associated to some self-indexing Morse function with a unique maximum and a unique minimum, $\boldsymbol{\alpha}$ is the collection of the attaching circles for the index one critical points, $\boldsymbol{\beta}$  is the collection of the attaching circles for the index two critical points, and $z$ is a basepoint in the complement of $\boldsymbol{\alpha}$ and $\boldsymbol{\beta}$. The Heegaard Floer complexes are  generated by $h$-tuples of intersection points between the $\boldsymbol{\alpha}$-curves and the $\boldsymbol{\beta}$-curves and the differential counts certain $J$-holomorphic curves in $\R \times [0,1] \times \Sigma$ with boundary on $\R \times \{ 0 \} \times \boldsymbol{\beta}$ and $\R \times \{ 1 \} \times \boldsymbol{\alpha}$. We refer to Section \ref{sec: HF} for an overview of Heegaard Floer homology in its cylindrical reformulation.

Hutchings, partially in collaboration with Taubes, defined {\em embedded contact homology} in \cite{Hu1, HT1, HT2} motivated by SW=Gr and by Eliashberg, Givental and Hofer's symplectic field theory \cite{SFT}. The starting point for embedded contact homology is a contact form $\alpha$ on $M$. The contact form determines the {\em Reeb vector field} $R$ by
$$\iota_R d \alpha =0 \quad \text{and} \quad \alpha(R)=1.$$
The embedded contact homology complex is generated by finite sets of simple Reeb orbits with finite multiplicities (called orbit sets) and its differential counts certain $J$-holomorphic
curves in the symplectisation $(\R \times M, d(e^s \alpha))$. We refer to Section \ref{sec: ECH} for an overview of embedded contact homology.

Thus by 2009 we had three Floer homology theories for three-manifolds, each with its own strengths and weaknesses, which share many formal properties: among others a $U$-map, a contact invariant and a decomposition as direct sum of subgroups indexed by $\op{Spin}^c$ structures in the case of monopole and Heegaard Floer homology and by first homology classes in the case of embeded contact homology. The relative advantage of monopole Floer homology is its link with geometry, that of embedded contact homology is its link with Reeb dynamics and that of Heegaard Floer homology is that it needs less sophisticated analytical tools than the previous two, which makes it relatively computable and easier to develop. 

These three Floer homology theories were expected to be isomorphic according to some big picture, and the proofs of the isomorphisms, which appeared in the last ten years, span several hundred pages of very technical arguments. The isomorphism between monopole Floer homology and embedded contact homology is due to Taubes and appeared in \cite{T2.1, T2.2, T2.3, T2.4, T2.5}. The isomorphism between monopole Floer homology and Heegaard Floer homology is due to Kutluhan, Lee and Taubes and appeared in \cite{KLT1, KLT2, KLT3, KLT4, KLT5}. The isomorphism between Heegaard Floer homology and embedded contact homology, which appeared in \cite{binding, I, II, III}, is the topic of this survey. On the other hand, the relation of these three theories with instanton Floer homology is still to be clarified.

These isomorphisms have already had various applications. The most striking ones are the proofs, for contact three manifolds, of the Weinstein conjecture by Taubes \cite{weinsein} and of the Arnold's chord conjecture by Hutchings and Taubes \cite{HT3, HT4}. 

The picture, however, is not yet complete for several reasons. First of all, the definitions of the isomorphisms require some choices, but it is not clear at what extent the result depends on them. Second, we do not know if composing two of the isomorphisms we obtain the third one. Of course this question will make sense only after the previous one has been answered. Finally, monopole Floer homology and Heegaard Floer homology are ``almost'' $3+1$ topological quantum field theories, but it is not known if the isomorphisms commute with maps induced by cobordisms.

The isomorphism between Heegaard Floer homology and embedded contact homology is expressed in a more precise way by the following theorem.
\begin{thm}\label{main}
Let $M$ be a closed three manifold and $\xi$ a contact structure on $M$. Then there are isomorphisms
\begin{align*}
\widehat{\Phi}_* & \colon \widehat{HF}(-M) \to \widehat{ECH}(M), \\
\Phi^+_* & \colon HF^+(-M) \to ECH(M)
\end{align*}
such that the diagram
\begin{equation}\label{commutativity of isomorphisms}
\xymatrix{
\ldots \ar[r] & \widehat{HF}(-M) \ar[r] \ar[d]_{\widehat{\Phi}_*} & HF^+(-M) \ar[r]^U \ar[d]_{\Phi^+_*} & HF^+(-M) \ar[r] \ar[d]_{\Phi^+_*} & \ldots \\
\ldots \ar[r] & \widehat{ECH}(M) \ar[r] & ECH(M) \ar[r]^U & ECH(M) \ar[r] & \ldots
}
\end{equation}
commutes. Moreover the isomorphisms map the contact class to the contact class and match the splitting according to Spin$^c$-structures in Heegaard Floer homology to the splitting according to first homology classes in embedded contact homology.
\end{thm}
Vinicius Ramos in \cite{Vini} proved that $\widehat{\Phi}_*$ and $\Phi^+_*$ respect also the grading by homotopy classes of plane fields which exists in both theories. For simplicity we will not talk about Spin$^c$-structures or gradings in this survey.

A natural setting for relating Heegaard Floer homology and embedded contact homology is that of {\em open book decompositions} (see Definition~\ref{dfn: open book}) because an open book decomposition determines both a Heegaard splitting and a contact structure. The Heegaard splitting is obtained by taking as Heegaard surface the union of two opposite pages. The contact structure is provided by a construction of Thurston and Winkelnkemper from \cite{TW}.

Fix an open book decomposition $(S, \hh)$ for $M$ where $S$ has genus $g>0$ and connected boundary. The first step in the definition of the isomorphisms is to adapt the definitions of $\widehat{HF}(-M)$ and $\widehat{ECH}(M)$ to the open book decomposition $(S,\hh)$. For Heegaard Floer homology this is achieved by pushing all interesting intersection points between the $\boldsymbol{\alpha}$- and $\boldsymbol{\beta}$-curves to one side of the Heegaard surface obtained from $(S, \hh)$; see Subsection \ref{HFhat on a page}.

For embedded contact homology, this is achieved in \cite{binding} and reviewed in Subsections \ref{subsec: ECH with boundary} and \ref{subsec: proof of binding}. There we introduce the group $\widehat{ECH}(N,\partial N)$ for the mapping torus $N$ of $(S, \hh)$ and prove that $\widehat{ECH}(M)\cong \widehat{ECH}(N,\partial N)$. This group is defined as a direct limit
$$\widehat{ECH}(N,\partial N)=\varinjlim ECH_i(N),$$
where $i$ is the number of intersections of an orbit set in $N$ with a fibre and the direct limit is taken with respect to maps $ECH_i(N)\to ECH_{j+1}(N)$
defined by increasing the multiplicity of an elliptic orbit on $\partial N$. This orbit
can be regarded intuitively as a receptacle for the $J$-holomorphic curves in
$\R \times M$ which intersect the cylinder over the binding.

Chain maps between Floer complexes are often defined by counting holomorphic curves in some symplectic cobordisms. Here, we use the open book $(S, \hh)$ to build a symplectic cobordism $W_+$ from $[0,1] \times S$ to $N$ and a symplectic cobordism $X_+$ from $[0,1] \times \Sigma$ to $M$. Here $\Sigma$ is the union of two opposite pages of the open book decomposition, and therefore is a closed surface of genus $2g$. Our convention is that symplectic cobordisms have cylindrical ends and we say that they go {\em from} the convex (positive) end {\em to} the concave (negative) end.

Counting holomorphic curves in $W_+$ and $X_+$ with suitable asymptotics and Lagrangian boundary conditions gives maps 
$\Phi_* \colon \widehat{HF}(-M) \to ECH_{2g}(N)$ and $ \Phi^+_* \colon HF^+(-M) \to ECH(M)$. 
After composing $\Phi_*$ with the natural map $ECH_{2g}(N) \to \widehat{ECH}(N, \partial N)$, we obtain the map $\widehat{\Phi}_*$. The maps $\Phi_*$ and $\widehat{\Phi}_*$ are defined in \cite{I} (with a slightly different notation), while $\Phi^+_*$ is defined in \cite{III}. Both are reviewed in Section \ref{sec: OC}. Strictly speaking, in the construction of $\Phi$ we replace the embedded Floer homology groups of $N$ with isomorphic periodic Floer homology groups; see Subsection \ref{subsec: PFH}.

In \cite{I} we also define a map $\Psi_* \colon ECH_{2g}(N) \to \widehat{HF}(-M)$ by counting holomorphic curves in a symplectic cobordism $\overline{W}_-$ passing through a base point. The cobordism $\overline{W}_-$ is defined as a compactification of $W_+$ turned upside-down. The Lagrangian boundary condition on $\overline{W}_-$ is singular, and this leads to many more potential degenerations of holomorphic curves. For this reason the analysis of the $J$-holomorphic curves defining $\Psi_*$ is longer and more difficult than the analysis of the $J$-holomorphic curves defining $\Phi_*$. The construction of $\Psi_*$ is reviewed in Section \ref{sec: CO}.

Then, in \cite{II} we prove that $\Phi_*$ and $\Psi_*$ are inverses of each other by composing the two cobordisms and degenerating them in a different way.
The proof that $\Phi_* \circ \Psi_*$ and $\Psi_* \circ \Phi_*$ are isomorphisms is thus
reduced to a computation of some relative Gromov-Taubes invariants. This step is briefly described in Section \ref{sec: homotopies}.

Finally we prove that the natural map $ECH_{2g}(N) \to \widehat{ECH}(N, \partial N)$ is an isomorphism by an argument based on stabilising the open book decomposition. This last step is described in Section \ref{sec: stabilisation}. This finishes the proof that $\widehat{\Phi}_*$ is an isomorphism. A simple algebraic argument based on the commutativity of the diagram \eqref{commutativity of isomorphisms} and the properties of the map $U$ in both theories shows that $\Phi^+_*$ is also an isomorphism.

\begin{conv}
Throughout the paper we will follow the convention that all three-manifolds are connected, oriented and compact. If $X$ is a manifold with boundary, we will denote $\op{int}(X)=X \setminus \partial X$.
\end{conv}

We finish this introduction with a warning to the reader: the articles \cite{binding,I,II,III} are still under review at the time of writing this survey, and therefore the numbering of the references to those articles is likely to become obsolete.
\subsection*{Acknowledgements}
We are indebted to Michael Hutchings for many helpful conversations. We also thank Denis Auroux, Thomas Brown, Tobias Ekholm, Dusa McDuff, Ivan Smith, Jean-Yves Welschinger and Chris Wendl for illuminating exchanges. Finally we are grateful to the organisers of the Boyerfest for giving us the opportunity of writing this survey. Since the beginning of the project VC and PG have been partially supported by  ANR grant ANR-08-BLAN-0291-01 ``Floer Power'', ERC grant 278246 ``GEODYCON'' and ANR grant ANR- 16-CE40-017 ``Quantact''.  KH was supported by NSF Grants DMS-0805352, DMS-1105432, DMS-1406564, and DMS-1549147. An important part of the work leading to \cite{binding, I, II, III} was carried out when PG and KH were visiting MSRI in Spring 2010; we are grateful to that institution for the hospitality.

\section{Moduli spaces of $J$-holomorphic curves}
In this section we  review some generalities about moduli spaces of holomorphic curves with the goal of fixing notations and conventions. The reader should keep in mind that we are not trying to give a one-size-fits-all definition of the various moduli space we will use: each one will be defined in due course; here we introduce the terminology that we will use in their definitions.

 Let ${\mathcal Y}$ be a compact, connected three-manifold. A {\em stable Hamiltonian structure} on ${\mathcal Y}$ is a pair $(\alpha, \omega)$ where $\alpha$ is a $1$-form, and $\omega$ is a closed $2$-form which satisfy $\alpha \wedge \omega >0$ and $\ker d \alpha \subset \ker \omega$. Stable Hamiltonian structures in this article will satisfy also one of the following conditions:
\begin{itemize}
\item[(i)] $d \alpha = \omega$, or
\item[(ii)] $d \alpha =0$.
\end{itemize}
Clearly when (i) is satisfied $\alpha$ is a contact form on ${\mathcal Y}$. Fibrations over $S^1$ are the main source of stable Hamiltonian structures satisfying (ii). In fact, if $\pi \colon {\mathcal Y} \to S^1$ is a locally trivial fibration with fibre ${\mathcal S}$, it is always possible to find a representative of the monodromy $\phi \colon {\mathcal S} \to {\mathcal S}$ preserving an area form $\omega$, and therefore to regard $\omega$ as a $2$-form on ${\mathcal Y}$. If $dt$ is a length form on $S^1$ and we define $\alpha= \pi^* dt$, then $(\alpha, \omega)$ will be called the stable Hamiltonian structure {\em induced by the fibration $\pi \colon {\mathcal Y} \to S^1$}. Stable Hamiltonian structures satisfying (ii) will arise, in the same way, also from fibrations $\pi \colon {\mathcal Y} \to I$ over a closed interval.

A stable Hamiltonian structure determines a Reeb vector field\footnote{Also called {\em Hamiltonian vector field}, especially when $\alpha$ is not a contact form.} $R$ on ${\mathcal Y}$, which is the vector field defined by the equations
$$\begin{cases}
\alpha(R)=1, \\
\iota_R \omega =0.
\end{cases}$$
An almost complex structure $J$ on $\R \times {\mathcal Y}$ is {\em compatible} with the stable Hamiltonian structure $(\alpha, \omega)$, or with the contact form $\alpha$ when (i) is satisfied, if
\begin{itemize}
\item $J$ is invariant under translations in the $\R$ direction,
\item $J(\partial_s)=R$, where $s$ denotes the coordinate on $\R$,
\item $J(\xi)=\xi$, where $\xi= \ker \alpha$, and
\item $\omega( \cdot, J \cdot)$ is an Euclidean metric on $\xi$.
\end{itemize}

Let $({\mathcal X}, \Omega_{\mathcal X})$ be a symplectic four-manifold, possibly with boundary and cylindrical ends (in the latter case we also call  $({\mathcal X}, \Omega_{\mathcal X})$ a {\em symplectic cobordism}). 
The positive ends are identified with $(0, +\infty) \times {\mathcal Y}_+$ and the negative ends are identified with $(-\infty, 0) \times {\mathcal Y}_-$ where ${\mathcal Y}_+$ and ${\mathcal Y}_-$ are three-manifolds which are  endowed with a stable Hamiltonian structure. On each end $\Omega_{\mathcal X}= d(e^s \alpha)$ if the stable Hamiltonian structure satisfies (i), or $\Omega_{\mathcal X}= ds \wedge \alpha + \omega$ if it satisfies (ii).
The boundary of ${\mathcal X}$ has a  vertical part which is foliated by symplectic submanifolds and a horizontal part which is convex. 
 The vertical boundary of ${\mathcal X}$, if nonempty, is equipped with a Lagrangian submanifold ${\mathcal L}$ such that the intersection of ${\mathcal L}$ with an end of ${\mathcal X}$ is a half cylinder over a collection of properly embedded curves in $\partial {\mathcal Y}$ which are tangent to $\xi$. Those curves
are called the boundary at infinity\footnote{This term will be used only in this section.} of ${\mathcal L}$. The prototype is ${\mathcal X}= \R \times [0,1] \times S$ with a product symplectic form $\Omega_{\mathcal X}= ds \wedge dt + \omega$, where $S$ is a surface with boundary, endowed with the Lagrangian submanifold ${\mathcal L}=  (\R \times \{ 0 \} \times  \mathbf{a}_0) \cup  (\R \times \{ 1 \} \times  \mathbf{a}_1)$, where $\mathbf{a}_0$ and $\mathbf{a}_1$ are properly embedded one-dimensional submanifolds with boundary of $S$.

On ${\mathcal X}$ we choose an almost complex structure $J_{\mathcal X}$ which is compatible with $\Omega_{\mathcal X}$ and with the stable Hamiltonian structures
on the ends. In order to have SFT compactness for $J_{\mathcal X}$-holomorphic curves, we assume that the vertical part of $\partial {\mathcal X}$ is foliated by $J_{\mathcal X}$-holomorphic submanifolds which form a barrier to $J_{\mathcal X}$-holomorphic curves and that the horizontal boundary is $J_{\mathcal X}$-convex, so that a maximum principle for $J_{\mathcal X}$-holomorphic curves holds. Note that in most cases ${\mathcal X}$ is the total space of a symplectic fibration ${\mathcal X} \xrightarrow{\varpi} {\mathcal B}$ and the condition on $J_{\mathcal X}$ at the vertical boundary is obtained by asking that $\varpi$ should be holomorphic, at least near $\partial {\mathcal X}$. We often require more properties from the almost complex structures, but in the text we will recall only those which are relevant for a first reading.

\begin{dfn}
A {\em $J_{\mathcal X}$-holomorphic curve} in ${\mathcal X}$ with boundary on ${\mathcal L}$ is a triple $(F, j, u)$ such that $(F, j)$ is a smooth Riemann surface, possibly with boundary and punctures, and $u \colon F \to {\mathcal X}$ is a proper map satisfying $du \circ j = J_{\mathcal X} \circ du$ and mapping each connected component of $\partial F$ to a distinct connected component of ${\mathcal L}$. If $F$ is connected, then $(F, j, u)$ is called {\em irreducible}. If $F_0 \subset F$ is a connected component, the restriction $(F_0, j|_{F_0}, u|_{F_0})$ will be called an {\em irreducible component} of $(F, j, u)$.
\end{dfn}
 We will simply say a {\em holomorphic curve} when the almost complex structure is clear from the context. In Section \ref{sec: CO} we will need to consider holomorphic curves with boundary on a singular Lagrangian ${\mathcal L}$. In that case we will assume that each connected component of $\partial F$ is sent to a distinct connected component of the smooth part of ${\mathcal L}$.

When ${\mathcal X}$ is  the total space of a symplectic fibration $\varpi \colon {\mathcal X} \to {\mathcal B}$ and $J_{\mathcal X}$ is compatible with the fibration, which means that $\varpi$ is holomorphic, we say that a holomorphic curve $(F, j, u)$ is a {\em degree $d$ multisection} of ${\mathcal X}$ if $\varpi \circ u \colon F \to {\mathcal B}$ is a branched cover of degree $d$.

 The punctures of $F$ are divided into positive and negative punctures so that $u$ maps a punctured neighbourhood of a positive puncture to a positive end of ${\mathcal X}$ and a punctured neighbourhood of a negative puncture to a negative end. Around each positive puncture we choose coordinates $(s,t) \in (0, + \infty) \times [0,1]$ for a boundary puncture and $(s,t) \in (0, + \infty) \times S^1$ for an interior puncture. Around each negative puncture we choose similar coordinates with $(0, + \infty)$ replaced by $(- \infty, 0)$. If $e$ is either a Reeb orbit in ${\mathcal Y}$ reparametrised so that it has period one, or a Reeb chord connecting two points of the boundary at infinity of ${\mathcal L}$, we say that $u$ is {\em positively} (resp. {\em negatively}) {\em asymptotic} to $e$ if, in the neighbourhood of a positive (resp. negative) puncture, $\lim \limits_{s \to \pm\infty} u(s,t)=e(t)$ (with positive sign for positive punctures and negative sign for negative punctures). We call an {\em end} of a holomorphic curve the restriction of $u$ which is mapped to an end of ${\mathcal X}$.
\begin{dfn}
If ${\mathcal X}= \R \times {\mathcal Y}$ and $J_{\mathcal X}$ is adapted to a stable Hamiltonian structure on ${\mathcal Y}$, a $J_{\mathcal X}$-holomorphic curve in ${\mathcal X}$ is either a {\em trivial cylinder} or {\em strip} if it parametrises $\R \times e$ where $e$ is, respectively, a closed Reeb orbit or a Reeb chord. In a general symplectic manifold ${\mathcal X}$ we say that an end of a $J_{\mathcal X}$-holomorphic curve is {\em trivial} if it coincides with a portion of a trivial cylinder or strip.
\end{dfn}

We say that two holomorphic curves $(F, j, u)$ and $(F', j', u')$ are equivalent, and write $(F, j, u) \sim (F', j', u')$, if there exists an orientation-preserving diffeomorphism $\phi \colon F \to F'$ extending smoothly over the punctures such that $\phi_*j=j'$ and $u= u' \circ \phi$. 
Given a set $\mathbf{e}$ of chords and orbits in the various ends, the moduli space\footnote{In the next sections we will distinguish chords and orbits belonging to different ends of ${\mathcal X}$. Here, for simplicity, we don't.}  ${\mathcal M}_{\mathcal X}(\mathbf{e})$ is the quotient by the equivalence relation $\sim$ of the space of $J_{\mathcal X}$-holomorphic curves in ${\mathcal X}$ with boundary on ${\mathcal L}$ which are asymptotic to the chords and orbits in $\mathbf{e}$. In the definition of ${\mathcal M}_{\mathcal X}(\mathbf{e})$ the topology of $F$ is {\em not} fixed, and therefore different $J_{\mathcal X}$-holomorphic curves can have different genera, number of punctures and number of connected components. Moreover, multiple orbits are treated in the embedded contact homology way, which means that only their total multiplicity in $\mathbf{e}$ counts. For example, if  $\mathbf{e}$ contains an orbit $\gamma$ with multiplicity two, then holomorphic curves\footnote{We will always call the elements of the moduli spaces ``holomorphic curves'' even if, strictly speaking, they are equivalence classes of holomorphic curves.} in ${\mathcal M}_{\mathcal X}(\mathbf{e})$ can have either one end at a double cover of $\gamma$ or two ends at $\gamma$. See \cite{Hu1} or Subsection \ref{subsec: ECH} for more details.

Let ${\mathcal T}(F)$ be the Teichm\"{u}ller space of complex structures on $F$ and $\operatorname{Aut}(F, j)$ the group of automorphisms of the Riemann surface $(F, j)$. To a $J_{\mathcal X}$-holomorphic curve $(F, j, u)$ we associate a formal deformation complex
$$0 \to T_{id}\operatorname{Aut}(F, j) \xrightarrow{\mathfrak{L}} T_{[j]}{\mathcal T}(F) \oplus \Omega^0(F; u^*T{\mathcal X}) \xrightarrow{D_{j,u}} \Omega^{0,1}(S; u^*T{\mathcal X}) \to 0$$
where $D_{j,u}$ is the formal linearised Cauchy-Riemann operator where $j$ is also considered as a variable, and $\mathfrak{L}$ is the formal linearisation of the action of $\operatorname{Aut}(F, j)$. If the ends of $u$ are nondegenerate, a suitable Sobolev completion of the formal deformation complex is an elliptic complex. Let ${\mathcal H}^i$
(for $i=0,1,2$) be its homology groups, which are finite dimensional because of the Fredholm property. We define the {\em Fredholm index} of a $J_{\mathcal X}$-holomorphic curve $(F, j, u)$ as
$$\op{ind}(u)= - \dim {\mathcal H}^0 + \dim {\mathcal H}^1 - \dim {\mathcal H}^2$$
If $u$ is not constant (and it will never be constant in this article), then ${\mathcal H}^0=0$. We say that $(F, j, u)$ is {\em regular} if the linearised operator $D_{j,u}$ is surjective in the appropriated Sobolev completion, i.e.\ if ${\mathcal H}^2=0$. If  $(F, j, u)$ is nonconstant and regular, then its equivalence class has a neighbourhood in the moduli spaces ${\mathcal M}_{\mathcal X}(\mathbf{e})$ that is diffeomorphic to a ball of dimension $\op{ind}(u)$. This is all standard holomorphic curves theory, and the reader can find the details in the many expositions of the topic, e.g.\ \cite[Lecture 7]{We-sft} and \cite[Section 13a]{Seidel-book}.

In the article we will more often use another index, which is specific to holomorphic curves in four-dimensional symplectic manifolds: the ECH-type index $I(u)$. For its definition we refer to the original articles (\cite{Hu1, Hu2}, where it was first introduced for embedded contact homology, and \cite{I, II, III} for its extension to the various symplectic cobordisms used in the proof of Theorem \ref{main}). Here we will only state its main properties:
\begin{enumerate}
\item {\em homology invariance}: $I(u)$ depends only on the relative homology class defined by $u$ (after a suitable compactification);
\item  {\em concatenation}: if two holomorphic curves $(F_+, j_+, u_+)$ and $(F_-, j_-, u_-)$ such that the positive ends of $u_-$ match the negative ends of $u_+$ are glued along the matching ends to form a holomorphic curve $(F, j, u)$, then $I(u)=I(u_-) + I(u_+)$; and
\item {\em index inequality}: if $\delta(u)$ is an algebraic count of singularities of $u$ as in \cite[Appendix E]{MS}, then
\begin{equation}\label{eqn: index inequality}
\op{ind}(u) + 2\delta(u) \le I(u).
\end{equation}
Moreover there is equality if $u$ has no end which is asymptotic to a closed Reeb orbit.
\end{enumerate}
The index inequality was first proved by McDuff for closed holomorphic curves as an adjunction inequality (see \cite[Appendix E]{MS} for a more modern exposition) and reinterpreted by Taubes as an index inequality (see \cite{T4}). For punctured holomorphic curves it was proved by Hutchings in \cite{Hu1} (see also \cite{Hu2}) in order to define embedded contact homology. The extension to other settings is straightforward and is done in \cite{I}.

The index inequality has the following important consequence: if $I(u)<2$ (and, in some cases, also if $I(u)=2$) then $u$ is embedded and $I(u)=\op{ind}(u)$. Embeddedness in particular implies somewhere injectivity, and therefore for a generic almost complex structure $J_{\mathcal X}$ all holomorphic curves of low ECH-type index are regular by standard techniques; see in \cite[Section 3.2]{MS}. A similar result holds for holomorphic curves of higher index satisfying constraints of sufficiently high codimension. For this reason in \cite{I, II, III} we do not need to worry about the regularity issues which trouble many parts of symplectic topology.

Given a property $\bigstar$ (e.g.\ $I=0$), we denote the subset of $J_{\mathcal X}$-holomorphic curves in ${\mathcal M}_{\mathcal X}(\mathbf{e})$ satisfying $\bigstar$ by ${\mathcal M}_{\mathcal X}^\star(\mathbf{e})$. This convention will be used throughout the paper.

For sake of brevity, in the next sections we will always denote a holomorphic curve $(F, j, u)$ simply by $u$. In view of the embeddedness of the holomorphic curves of low ECH-type index, we will also often identify a holomorphic curve with its image.
\section{Heegaard Floer homology}\label{sec: HF}

\subsection{A review of Heegaard Floer homology}
In this section we briefly review the Heegaard Floer homology groups associated to a closed $3$-manifold $M$. We will work with Lipshitz's ``cylindrical reformulation'' from \cite{Li} instead of with the original definition from \cite{OSz1}.

Every closed three-manifold can be encoded by a {\em pointed Heegaard diagram} $(\Sigma, \bs{\alpha}, \bs{\beta}, z)$ associated to a Heegaard decomposition of $M$. Here $\Sigma$ is a closed, oriented, connected surface of genus $h$ which divides $M$ into two handlebodies $H_\alpha$ and $H_\beta$, and $\bs{\alpha} =\{\alpha_1,\dots,\alpha_{h}\}$ and   $\bs{\beta} =\{\beta_1,\dots,\beta_{h}\}$ are  collections of $h$ pairwise disjoint simple closed curves in $\Sigma$ which bound discs in $H_\alpha$ and $H_\beta$ respectively and are linearly independent in $H_1(\Sigma;\Z)$.  Finally $z\in\Sigma-\bs{\alpha}-\bs{\beta}$ is a base point, which is not needed to describe $M$, but is a crucial ingredient in the definition of the Heegaard Floer chain complexes.

After fixing an area form $\omega$ on $\Sigma$, we obtain a stable Hamiltonian structure $(dt, \omega)$ on $[0,1] \times \Sigma$ with Reeb vector field $R= \partial_t$. The submanifolds $L_{\bs{\alpha}}= \R \times \{ 1 \} \times \bs{\alpha}$ and $L_{\bs{\beta}} = \R \times \{ 0 \} \times \bs{\beta}$ of $\R \times [0,1] \times \Sigma$ are Lagrangian for the symplectic form $\Omega_X= ds \wedge dt + \omega$. We choose an almost complex structure $J$ on $X$ which is compatible with the stable Hamiltonian structure $(dt, \omega)$ and such that the surface $\R \times [0,1] \times \{ z \}$ is holomorphic.

We denote by ${\mathcal S}_{\bs{\alpha}, \bs{\beta}}$ the set of unordered $h$-tuples of intersection points $\mathbf{y}= \{ y_1, \ldots, y_h \}$ between $\alpha$- and $\beta$-curves for which there is a permutation $\sigma \in \mathfrak{S}_h$ such that $y_i \in \alpha_i \cap \beta_{\sigma(i)}$.
Given $\mathbf{y}_+, \mathbf{y}_- \in {\mathcal S}_{\alpha, \beta}$, we consider the moduli
space\footnote{The same moduli spaces are denoted by $\mathcal{M}^X_J(\mathbf{y}, \mathbf{y}')$ in \cite{I}.} ${\mathcal M}_X(\mathbf{y}_+, \mathbf{y}_-)$ of  $J$-holomorphic curves in $X$ with boundary on $L_{\bs{\alpha}} \cup L_{\bs{\beta}}$ and asymptotic to chords $[0,1] \times \mathbf{y}_\pm$ for $s \to \pm \infty$ ({\em $HF$-curves} in the following). Translations in the $\R$ direction act on ${\mathcal M}_X(\mathbf{y}_+, \mathbf{y}_-)$  and we denote the quotient by $\widehat{\mathcal M}_X(\mathbf{y}_+, \mathbf{y}_-)$.

To a $J$-holomorphic curve $u \in {\mathcal M}_X(\mathbf{y}_+, \mathbf{y}_-)$ we associate two topological quantities: an ``intersection number'' $n_z(u)$, defined as the algebraic intersection of the image of $u$ with $[-1,1] \times [0,1] \times \{ z \}$ (see \cite[Appendix E]{MS} for the intersection number between two $J$-holomorphic curves), and the ECH-type index\footnote{denoted by $I_{HF}$ in \cite{I}.} $I(u)$ (\cite[Equation 4.5.4]{I}), which in this context is a reformulation of Lipshitz's index formula \cite[Corollary 4.3]{Li} (see also \cite{Li2}). The index inequality is proved in \cite[Theorem 4.5.13]{I}. 
By positivity of intersection for $J$-holomorphic curves in dimension four, $n_z(u) \ge 0$. Moreover, for a generic almost complex structure, $I(u) \ge 0$ for each HF-curve $u$. 

The chain complex $\widehat{CF}(\Sigma, \bs{\alpha}, \bs{\beta}, z)$ is freely generated by ${\mathcal S}_{\bs{\alpha}, \bs{\beta}}$ as a vector space over $\Z/2\Z$.  The differential is
$$\widehat{\partial} \mathbf{y}_+ = \sum_{\mathbf{y}_- \in {\mathcal S}_{\bs{\alpha}, \bs{\beta}}} \# \widehat{\mathcal M}_X^{I=1, n_z=0}(\mathbf{y}_+, \mathbf{y}_-) \mathbf{y}_-.$$
The chain complex $CF^+(\Sigma, \bs{\alpha}, \bs{\beta}, z)$ is freely generated, as a vector space over $\Z/2\Z$, by the pairs $[\mathbf{y}, i]$ where $\mathbf{y} \in {\mathcal S}_{\bs{\alpha}, \bs{\beta}}$ and $i \in \N$ (our convention is that $0 \in \N$). The differential is
$$\partial^+([\mathbf{y}_+, i])= \sum_{\mathbf{y}_- \in {\mathcal S}_{\bs{\alpha}, \bs{\beta}}} \sum_{0 \le j \le i} \# \widehat{\mathcal M}_X^{I=1, n_z=j}(\mathbf{y}_+, \mathbf{y}_-)[\mathbf{y}_-, i-j].$$
In order to have finite sums we need to assume {\em weak admissibility}, which can be rephrased by asking that the curves $\bs{\alpha}$ and $\bs{\beta}$ be exact Lagrangian submanifolds for some primitive of $\omega$ on $\Sigma \setminus \{ z \}$.
Weak admissibility can always be achieved up to isotopy of the $\alpha$- and $\beta$-curves and deformation of $\omega$. See \cite[Section~4.2.2]{OSz1} or \cite[Section 5]{Li}.

The homologies of $\widehat{CF}(\Sigma, \bs{\alpha}, \bs{\beta}, z)$ and $CF^+(\Sigma, \bs{\alpha}, \bs{\beta}, z)$ are denoted by $\widehat{HF}(M)$ and $HF^+(M)$ respectively.
The fact that they are well defined invariants of three-manifolds up to diffeomorphism was proved in \cite{OSz1} and reproved in \cite{Li} for the cylindrical reformulation we are using in this article. Naturality is addressed in \cite{naturality}.

The map $U \colon CF^+(\Sigma, \bs{\alpha}, \bs{\beta}, z) \to  CF^+(\Sigma, \bs{\alpha}, \bs{\beta}, z)$ defined as
$$ U([\mathbf{y}, i])= \begin{cases}
[\mathbf{y}, i-1] & \text{if } i \ge 1, \\
0 &  \text{if } i=0 \end{cases}$$
is a chain map. The short exact sequence
$$0 \longrightarrow \widehat{CF}(\Sigma, \bs{\alpha}, \bs{\beta}, z) \longrightarrow CF^+(\Sigma, \bs{\alpha}, \bs{\beta}, z) \stackrel{U} \longrightarrow CF^+(\Sigma, \bs{\alpha}, \bs{\beta}, z) \longrightarrow 0$$
induces the exact triangle
\begin{equation}\label{successione tautologica}
\xymatrix{
HF^+(M) \ar[rr]^{U} & & HF^+(M) \ar[ld] \\
& \widehat{HF}(M) \ar[ul]
}
\end{equation}
in homology. Heegaard Floer homology decomposes as a direct sum of groups indexed by Spin$^c$ structures and the triangle \eqref{successione tautologica} holds for every summand. For simplicity we will not discuss this decomposition.

\subsection{A geometric interpretation of the $U$-map}\label{subsec: geometric U}
This subsection is taken from \cite[Section 3]{III}, but some of the arguments are slightly modified.  In \cite{OSz1,Li}, the $U$-map
$$U : CF^+ (\Sigma,\bs\alpha,\bs\beta,z)\rightarrow CF^+ (\Sigma,\bs\alpha,\bs\beta,z)$$
is defined algebraically as $U([{\bf y},i])=[{\bf y}, i-1]$, but we need to redefine it geometrically to be able to compare it with the $U$-map in ECH.

Let us fix\footnote{What is called $z$ and $\mathbf{z}$ here is called $z^f$ and $z$ respectively in \cite{III}.} $\mathbf{z} = (0, \frac 12 ,z)\in X= \R \times [0,1] \times \Sigma$, and let $J^\Diamond$ be a generic small perturbation of $J$ supported near $\mathbf{z}$ such that $\R \times [0,1] \times \{ z \}$ remains holomorphic. This perturbation is needed so that $J^\Diamond$-holomorphic curves do not have a closed irreducible component passing through $\mathbf{z}$ .

Given $\mathbf{y}_+, \mathbf{y}_- \in {\mathcal S}_{\bs{\alpha}, \bs{\beta}}$ we denote by
${\mathcal M}_X(\mathbf{y}_+, \mathbf{y}_-; \mathbf{z})$ the moduli space of
$J^\Diamond$-holomorphic curves in $X$ with boundary on $L_{\bs{\alpha}} \cup L_{\bs{\beta}}$
which are positively asymptotic to $[0,1] \times \mathbf{y}_+$, negatively
asymptotic to $[0,1] \times \mathbf{y}_-$ and pass through $\mathbf{z}$.
Note that $n_z(u) \ge 1$ for $u \in {\mathcal M}_X(\mathbf{y}_+, \mathbf{y}_-;
\mathbf{z})$.
\begin{dfn}\label{geometric U}
The {\em geometric $U$-map} with respect to the point $\mathbf{z}$ is the map
$$U_{\mathbf{z}}([\mathbf{y}_+,i])= \sum \limits_{\mathbf{y}_- \in
{\mathcal S}_{\bs{\alpha}, \bs{\beta}}} \sum \limits_{j=1}^i \# {\mathcal M}_X^{I=2,
n_z=j}(\mathbf{y}_+, \mathbf{y}_-; \mathbf{z}) [\mathbf{y}_-, i-j].$$
\end{dfn}

Standard arguments in symplectic geometry based on
the compactness and gluing results of \cite{Li} show that the geometric $U$-map
$$U_{\mathbf{z}} \colon CF^+(\Sigma, \bs{\alpha}, \bs{\beta}, z) \to
CF^+(\Sigma, \bs{\alpha}, \bs{\beta}, z)$$
is a chain map. The main result of the subsection is the following.

\begin{thm}\label{thm: U-map}
There exists a chain homotopy
$$H:  CF^+ (\Sigma,\bs\alpha,\bs\beta,z)\rightarrow CF^+ (\Sigma,\bs\alpha,\bs\beta,z)$$
such that
\begin{equation} \label{eqn: Us}
  U_{\mathbf{z}} -U = H\circ \partial^+ + \partial^+ \circ H
\end{equation}
and $H([{\bf y} ,0])=0$ for all ${\bf y} \in \mathcal{S}_{\bs{\alpha}, \bs{\beta}}$.
\end{thm}
\begin{proof}
We will define a chain homotopy between $U_{\mathbf{z}}$ and $U$ by moving $\mathbf{z}$
towards the boundary and identifying a count of holomorphic curves in the limit configuration
with the map $U$. More precisely, for $\tau \in [0, 1)$ we define $\mathbf{z}_\tau = \{ 0 \} \times \{ \frac{1- \tau}{2} \} \times \{z \}$, so that  $\mathbf{z}_0 = \mathbf{z}$, and consider a generic smooth family of almost complex structures $J^\Diamond_\tau$, each obtained by a small perturbation of $J$ supported in a  neighbourhood of $\mathbf{z}_\tau$,  such that $J_0^\Diamond=J^\Diamond$ and $\R \times [0,1] \times \{ z \}$ remains holomorphic for all $\tau$.

\begin{figure} \centering
\begin{overpic}[width=6cm]{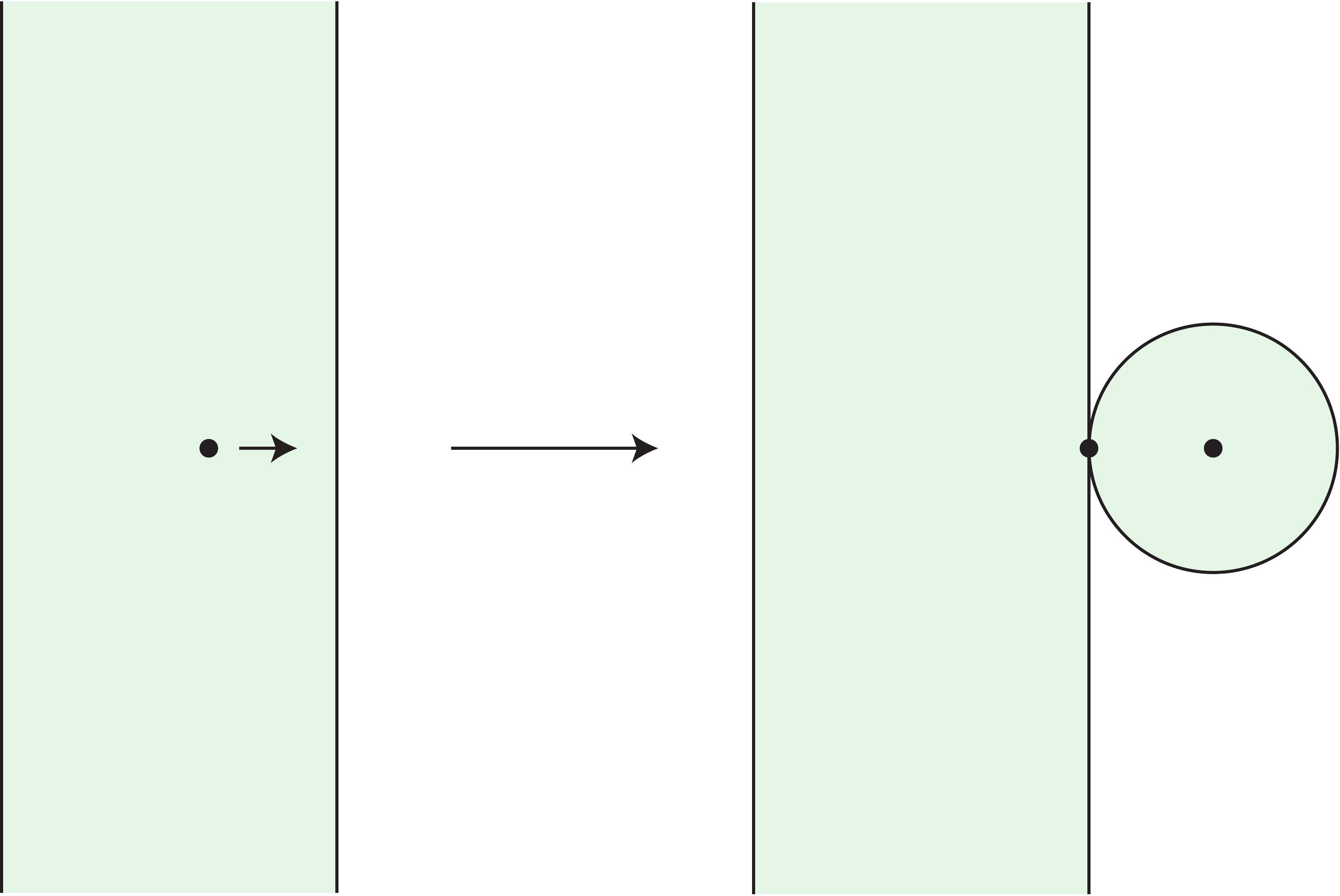}
\put(7, 27){\tiny $\left (0, \frac{1-\tau}{2} \right )$}
\put(35, 35){\tiny $\tau \to 1$}
\end{overpic}
\caption{Moving the basepoint toward the boundary.}
\label{figure: bubbling}
\end{figure}

Let ${\mathcal M}_X(\mathbf{y}_+, \mathbf{y}_-; \mathbf{z}_*)$ be the moduli space of
$J^\Diamond_\tau$-holomorphic curves in $X$ with boundary on $L_{\bs{\alpha}} \cup L_{\bs{\beta}}$ which are positively asymptotic to $[0,1] \times \mathbf{y}_+$, negatively
asymptotic to $[0,1] \times \mathbf{y}_-$ and pass through $\mathbf{z}_\tau$ for some $\tau \in (0,1)$. The map $H$ is defined as
$$H([\mathbf{y}_+, i])= \sum \limits_{\mathbf{y}_- \in
{\mathcal S}_{\bs{\alpha}, \bs{\beta}}} \sum \limits_{j=1}^i \# {\mathcal M}_X^{I=1,
n_z=j}(\mathbf{y}_+, \mathbf{y}_-; \mathbf{z}_*) [\mathbf{y}_-, i-j].$$

The structure of the boundary of the compactification of the moduli spaces ${\mathcal M}_X^{I=2}(\mathbf{y}_+, \mathbf{y}_-; \mathbf{z}_*)$ implies that $H$ is a chain homotopy between $U_{\mathbf{z}}$ and a map defined by counting certain broken holomorphic curves in the limit $\tau \to 1$. To finish the proof we need to identify this map with $U$. This will be done in the rest of the section.
\end{proof}
In the limit $\tau \to 1$ we obtain $\widetilde{X} = \widetilde{B}\times \Sigma$ where
 $\widetilde{B}=(B \sqcup D)/ \! \! \sim$ with $B= \R \times [0,1]$, $D= \{ |x| \le 1 \} \subset \C$ and $\sim$ identifies $(0,0) \in B$ with $-1 \in D$.  See Figure \ref{figure: bubbling}. For an appropriate choice of $J_\tau^\Diamond$, the limit almost complex structure is $J$ on $B \times \Sigma$ and a small perturbation $J^\Diamond_D$ of a product almost complex structure $J_D$ on $D \times \Sigma$ such that $J_D^\Diamond$ and $J_D$ coincide outside of a small neighbourhood of $\mathbf{z}_1= (0, z) \in D \times \Sigma$. From now on we write $\mathbf{z}_1= \mathbf{z}$. 

We define the ECH-type index $I$ --- in fact a relative version of Taubes's index from \cite{T4} in this case --- for a homology class $A \in H_2(D \times \Sigma, \partial D \times \bs{\beta})$ which admits a representative $F$ such that each component of $\partial F$ is mapped to a distinct component of $\partial D \times \bs{\beta}$.  Let  $(\tau, \tau')$ be the trivialisation of $T(D \times \Sigma)$ along $\partial D \times \bs{\beta}$ such that $\tau$ is induced by a radial and outward-pointing vector field along $\partial D$ tangent to $D$,  and $\tau'$ is induced by a vector field along $\bs{\beta}$ tangent to $\Sigma$ and transverse to $\bs{\beta}$.
Let $Q_{(\tau, \tau')}(A)$ be the intersection number between $F$ and a push-off of $F$, where
$\partial F$ is pushed along $\tau'$. We define\footnote{The formula in \cite[Definition 3.2.1]{III} has an extra term $\mu_{(\tau, \tau')}(\partial A)$, which vanishes here by our choice of trivialisation.}
\begin{equation}\label{relative Taubes index}
I(A)= c_1(T(D \times \Sigma)|_A, (\tau, \tau')) + Q_{(\tau, \tau')}(A).
\end{equation}

If $u$ is a holomorphic curve in $D \times \Sigma$ with boundary on $\partial D \times \bs{\beta}$ representing a relative homology class $A$ as above, then the index inequality gives
\begin{equation}\label{adjunction in D times Sigma}
I(A)= \op{ind}(u)+ 2\delta(u).
\end{equation}
In particular $u$ is an embedding if and only if $\op{ind}(u)=I(A)$. See \cite[Theorem 4.5.13]{I} for the proof of a similar result. Let $A_{k_0, k_1}= k_0 [D \times \{ pt \}] + k_1 [\{ pt \} \times \Sigma]$. If $k_0 \le h$, where $h$ is the genus of $\Sigma$, an easy computation gives
\begin{equation}\label{Index in D times Sigma}
I(A_{k_0, k_1})= k_0+ k_1(2-2h)+2k_0k_1.
\end{equation}
In particular, $I(A_{h, k_1})= 2 k_1 +h$.

We fix points\footnote{These points are called $w_i$ in \cite{III}.} $y_i \in \beta_i$ for $i=1, \ldots, h$ and denote by ${\mathcal M}_{D \times \Sigma}(A_{k_0, k_1}; \mathbf{z}, \mathbf{y})$ the moduli space of $J_D^\Diamond$-holomorphic curves in $D \times \Sigma$ with boundary on $\partial D \times \bs{\beta}$ representing the homology class $A_{k_0, k_1}$ and passing through $(0,z)$ and
$k_0$ points of the form $(-1, y_i)$.
\begin{rem}\label{rem: empty for k_1 le 0}
If $k_1 \le 0$, a standard compactness argument shows that  the moduli spaces ${\mathcal M}_{D \times \Sigma}(A_{k_0, k_1}; \mathbf{z}, \mathbf{y})$ are empty, provided that $J_D^\Diamond$ is close enough to a product almost complex structure.
\end{rem}
\begin{lemma}\label{compactmess for widetilde X}
Let $u_{\tau_i}$, for $\tau_i \to 1$ as $i \to \infty$, be a sequence of $J_{\tau_i}^\Diamond$-holomorphic curves in ${\mathcal M}_X^{I=2, n_z=j}(\mathbf{y}_+, \mathbf{y}_-; z_{\tau_i})$ for some fixed $\mathbf{y}_\pm \in {\mathcal S}_{\bs{\alpha}, \bs{\beta}}$ and $j >0$.  Then, up to passing to a subsequence, $u_{\tau_i}$ converges to a pair of holomorphic curves $(u_B, u_D)$ where $u_B$ is a union of trivial strips $B \times \{ \mathbf{y} \}$ in $B \times \Sigma$ over an intersection point $\mathbf{y} \in {\mathcal S}_{\bs{\alpha}, \bs{\beta}}$ and $u_D \in {\mathcal M}_{D \times \Sigma}(A_{h,1}; \mathbf{z}, \mathbf{y})$. In particular
$\mathbf{y}_+ = \mathbf{y}_- = \mathbf{y}$ and $j=1$.
\end{lemma}
\begin{proof}
By Gromov compactness a subsequence of $u_{\tau_i}$ converges to a pair $(u_B, u_D)$ where $u_B$ is a $J$-holomorphic curve in ${\mathcal M}_X(\mathbf{y}_+, \mathbf{y}_-)$ and $u_D$ is a $J_D^\Diamond$-holomorphic curve in ${\mathcal M}_{D \times \Sigma}(A_{h, k_1}; \mathbf{z}, \mathbf{y})$ for some $\mathbf{y}=(y_1, \ldots, y_h)$. A simple computation shows that
\begin{equation}\label{almost additivity of the index}
I(u_B)+ I(u_D)=2 - h.
\end{equation}

By Equation\eqref{Index in D times Sigma} we can rewrite Equation \eqref{almost additivity of the index} as $I(u_B)+2k_1=2$.
Since $k_1 \ge 1$ by Remark \ref{rem: empty for k_1 le 0} and $I(u_B) \ge 0$, this implies that $I(u_B)=0$ and $k_1=1$; in particular $u_B$ is a union of trivial strips and
therefore $\mathbf{y}_\pm = \mathbf{y}$ and $j=1$.
\end{proof}

The moduli spaces ${\mathcal M}_{D \times \Sigma}(A_{h,1}; \mathbf{z}, \mathbf{y})$ are zero-dimensional and consist of embedded $J^\Diamond_D$-holomorphic curves by Equations \eqref{adjunction in D times Sigma} and \eqref{Index in D times Sigma}. We denote
$$G(h) = \# {\mathcal M}_{D \times \Sigma}(A_{h,1}; \mathbf{z}, \mathbf{y}).$$
Thus, by Lemma \ref{compactmess for widetilde X}, $H$ is a homotopy between $U_{\mathbf{z}}$ and $G(h)U$. In order to compute $G(h)$ we further degenerate $D \times \Sigma$. The first step is to degenerate $D$ into $D \cup S^2$, where $0 \in D$ is identified with $\infty \in S^2 \cong \C \cup \{ \infty \}$ and $\mathbf{z} = (0, z) \in S^2 \times \Sigma$. Let $J_S^\Diamond$ denote the limit almost complex structure on $S^2 \times \Sigma$, which we assume to be a small perturbation of a product almost complex structure $J_S$ in a small neighbourhood of $\mathbf{z}$.

Holomorphic curves in ${\mathcal M}_{D \times \Sigma}(A_{h,1}; \mathbf{z}, \mathbf{y})$ degenerate into pairs of curves consisting in the trivial multisection $D \times \{ \mathbf{y} \}$ in $D \times \Sigma$ and a $J_S^\Diamond$-holomorphic curve in $S^2 \times \Sigma$ representing the homology class $B_{h,1}= h[S^2] + [\Sigma]$ passing through $\mathbf{z}=(0,z)$ and $(\infty, y_1), \ldots, (\infty, y_h)$. We denote the moduli space of such curves by ${\mathcal M}_{S^2 \times \Sigma}(B_{h,1}; \mathbf{z}, \mathbf{y})$.

Holomorphic curves in ${\mathcal M}_{S^2 \times \Sigma}(B_{h,1}; \mathbf{z}, \mathbf{y})$ can be reducible and only the irreducible component passing through $(0,z)$ has to be regular. We denote the subset of ${\mathcal M}_{S^2 \times \Sigma}(B_{h,1}; \mathbf{z}, \mathbf{y})$ consisting of irreducible $J_S^\Diamond$-holomorphic curves by  ${\mathcal M}_{S^2 \times \Sigma}^{irr}(B_{h,1}; \mathbf{z}, \mathbf{y})$. Simple index considerations taking into account the homological and point constraints imply that the elements in ${\mathcal M}_{S^2 \times \Sigma}(B_{g,1}; \mathbf{z}, \mathbf{y}) \setminus {\mathcal M}_{S^2 \times \Sigma}^{irr}(B_{g,1}; \mathbf{z}, \mathbf{y})$ consist of $\{ \infty \} \times \Sigma$ with $h$ spheres in the class $[S^2]$, one of which passes through $(0, z)$. However this irreducible configuration cannot appear in the limit of a sequence of curves in ${\mathcal M}_{D \times \Sigma}(A_{h,1}; \mathbf{z}, \mathbf{y})$ as $D$ degenerates into $D \cup S^2$. This can be proved by a careful analysis of the limit in the SFT sense, but it can also be seen intuitively as follows: if $J_D^\Diamond$ is close enough to the product almost complex structure $J_D$, then the $J_D^\Diamond$-holomorphic curves in ${\mathcal M}_{D \times \Sigma}(A_{h,1}; \mathbf{z}, \mathbf{y})$ are $C^0$-close to $\{ 0 \} \times \Sigma \cup D \times \{ y_1 \} \cup \ldots \cup D \times \{ y_h \}$, and therefore the $J_S^\Diamond$-holomorphic component of the limit is $C^0$-close to $\{ 0 \} \times \Sigma \cup S^2 \times \{ y_1 \} \cup \ldots \cup S^2 \times \{ y_h \}$. Thus we have proved that
$$G(h)= \# {\mathcal M}_{S^2 \times \Sigma}^{irr}(B_{h,1}; \mathbf{z}, \mathbf{y}).$$
The following lemma completes the proof of Theorem \ref{thm: U-map}.
\begin{lemma} $G(h)=1$ for all $h \ge 1$.
\end{lemma}
The proof will follow by combining the two lemmas below.
\begin{lemma}
For every $h \ge 1$ we have $G(h)=G(1)^h$.
\end{lemma}
\begin{proof}
  In the proof we will degenerate $\Sigma$ along $h-1$ separating curves in order to obtain a nodal surface $\widetilde{\Sigma}$ whose irreducible components are tori. We choose the curves so that each irreducible component contains exactly one of the points $y_i$. Since the basepoint $\mathbf{z}$ remains in one component, the almost complex structure on $S^2 \times \widetilde{\Sigma}$ is a product almost complex structure in all but one of the irreducible components of $S^2 \times \widetilde{\Sigma}$. Since a product almost complex structure is not generic enough, before degenerating $\Sigma$ we need to modify $J_S^\Diamond$. To this aim we introduce a generic almost complex structure $J_S^\heartsuit$ on $D \times \Sigma$ among those making the projection $S^2 \times \Sigma \to \Sigma$ into a holomorphic map and keeping the section $\{ \infty \} \times \Sigma$ holomorphic. By a standard continuation argument the cardinality of the set ${\mathcal M}_{S^2 \times \Sigma}^{irr}(B_{h,1}; \mathbf{z}, \mathbf{y})$ is the same for $J_S^\Diamond$ and $J_S^\heartsuit$; from now on we will work with the latter almost complex structure.

As $\Sigma$ degenerates towards $\widetilde{\Sigma}$, holomorphic curves in ${\mathcal M}_{S^2 \times \Sigma}^{irr}(B_{1,h}; \mathbf{z}, \mathbf{y})$ degenerate into holomorphic curves in $S^2 \times \widetilde{\Sigma}$, with the same point constraints, which are irreducible in each irreducible component of  $S^2 \times \widetilde{\Sigma}$, and such that two irreducible components meet at one point when two irreducible components of $S^2 \times \widetilde{\Sigma}$ meet. This shows that $G(h)=G(1)^h$.
\end{proof}
\begin{rem}
The section $\{ \infty \} \times \Sigma$ is not regular, and thus neither $J_S^\Diamond$ nor $J_S^\heartsuit$ are generic almost complex structures. What we are computing here is a simple instance of {\em relative Gromov-Witten invariant} in the sense of \cite{IP1}.
\end{rem}
\begin{lemma} $G(1)=1$.
\end{lemma}
\begin{proof}
The lemma follows from \cite[Example 8.6.12]{MS}, but one can also argue more explicitly 
by degenerating $\Sigma = T^2$ into two spheres touching each other in two points, one of which contains $z$ and the other one $y_1$. We call $\Sigma_0$ the component containing $z$ and $\Sigma_1$ the component containing $y_1$. We also call $0$ and $\infty$ the points where $\Sigma_0$ and $\Sigma_1$ meet.

As we degenerate $\Sigma$, the $J^\heartsuit_S$-holomorphic curves in   ${\mathcal M}_{S^2 \times \Sigma}^{irr}(B_{1,1}; \mathbf{z}, \mathbf{y})$ converge to pairs of holomorphic curves $(u_0, u_1)$, where $u_0$ takes values in $S^2 \times \Sigma_0$ and passes through $(z, 0)$, and $u_1$ takes values in $S^2 \times \Sigma_1$ and passes through $(y_1, \infty)$. 
Since $J^\heartsuit_S$ is close to a product almost complex structure, $J^\heartsuit_S$-holomorphic curves in   ${\mathcal M}_{S^2 \times \Sigma}^{irr}(B_{1,1}; \mathbf{z}, \mathbf{y})$ are $C^0$-close to $\{ 0 \} \times \Sigma \cup S^2 \times \{ y_1 \}$. This implies that the curve $u_0$ represents the homology class $[S^2] + [\Sigma_0]$, while the  curve $u_1$ represents the homology class $[\Sigma_0]$. Moreover the images of $u_0$ and $u_1$ match at $S^2\times  \{ 0, \infty \}$. The image of $u_1$ is a small perturbation of the graph of a degree zero holomorphic map $\Sigma_1 \to S^2$ and the image of $u_0$ is a small perturbation of the graph of a degree one holomorphic map $\Sigma_0 \to S^2$. Then elementary complex analysis implies that there is a unique choice for $u_1$, while the choice for $u_0$ becomes unique once the intersection of its image with $S^2\times \{ 0, \infty \}$ is fixed. This implies that $G(1)=1$.
\end{proof}

\subsection{Adapting $\widehat{HF}$ to an open book decomposition}
\label{HFhat on a page}
Let $M$ be a closed $3$-manifold, $B \subset M$ a link, $S$ a compact, oriented, connected surface with nonempty boundary, and $\hh \colon S \rightarrow S$ an orientation-preserving diffeomorphism such that $\hh|_{\partial S}=id$.
\begin{dfn}\label{dfn: open book}
An {\em open book decomposition} of $M$ with {\em binding} $B$, {\em page} $S$ and {\em monodromy} $\hh$ is a locally trivial fibration
$\pi \colon M \setminus B \to S^1$ with fibre $\op{int}(S)$ and monodromy $\hh|_{\op{int}(S)}$ such that the closure of any fibre is a surface with boundary whose boundary is $B$. The pair $(S, \hh)$ is called an {\em abstract open book decomposition}.
\end{dfn}
We will call a page not only the abstract surface $S$, but also all surfaces $S_t=\overline{\pi^{-1}(t)}$ for $t \in S^1$.

In this section we explain how to associate a pointed Heegaard diagram to an open book decomposition and compute $\widehat{HF}(-M)$ from the page and the monodromy, using a construction from \cite{HKM}. From now on we assume that $\partial S$ is connected and $S$ has genus $g$.

We identify $S^1\cong [0,2]/0 \sim 2$. An open book decomposition with page $S$ and monodromy $\hh$ gives rise to a Heegaard decomposition $M=H_{\bs{\alpha}} \cup H_{\bs{\beta}}$, where $H_{\bs{\alpha}} =\overline{\pi^{-1}([0,1])}$, $H_{\bs{\beta}} =\overline{\pi^{-1}([1,2])}$, and the Heegaard surface $\Sigma =S_1 \cup -S_0$ is the union of two pages glued along the binding.

A {\it basis of arcs} for $S$ is a collection of $2g$ pairwise disjoint properly embedded arcs ${\bf a} =\{a_1,\dots,a_{2g}\}$ in $S$ such that $S \setminus {\bf a}$ is a connected polygon. Starting from a basis ${\bf a}$ for $S$, we can construct $\alpha$- and $\beta$-curves for $\Sigma$ as follows:
$$\bs{\alpha} = ({\bf a} \times \{ 1 \}) \cup ({\bf a} \times \{ 0\}) \quad  \text{and} \quad \bs{\beta} = ({\bf b} \times \{ 1 \}) \cup (\hh({\bf a} )\times \{ 0\}).$$
Here ${\bf b}$ is a small deformation of ${\bf a}$ relative to its endpoints, so that each pair $a_i$ and $b_i$ intersects each other transversely at three points: two of the intersections are their endpoints $x_i$ and $x_i'$ on $\partial S$ and the third intersection is an interior point $x_i''$; see Figure~\ref{darkside}.\footnote{\cite[Figure 1]{I} is flipped with rispect to Figure~\ref{darkside} because it shows $S_1$ instead of $-S_1$.} We assume that $\hh$ is chosen so that the $\alpha$- and $\beta$-curves are smooth and intersect transversely. These conditions can be easily achieved by an isotopy of $\hh$ relative to the boundary.
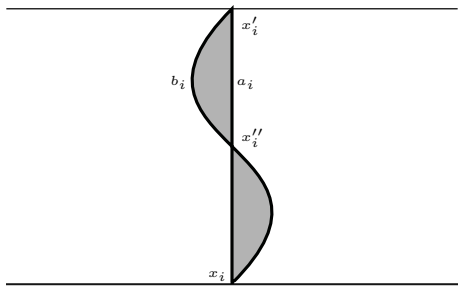
\begin{figure}\centering
\begin{overpic}[width=6cm, grid=false]{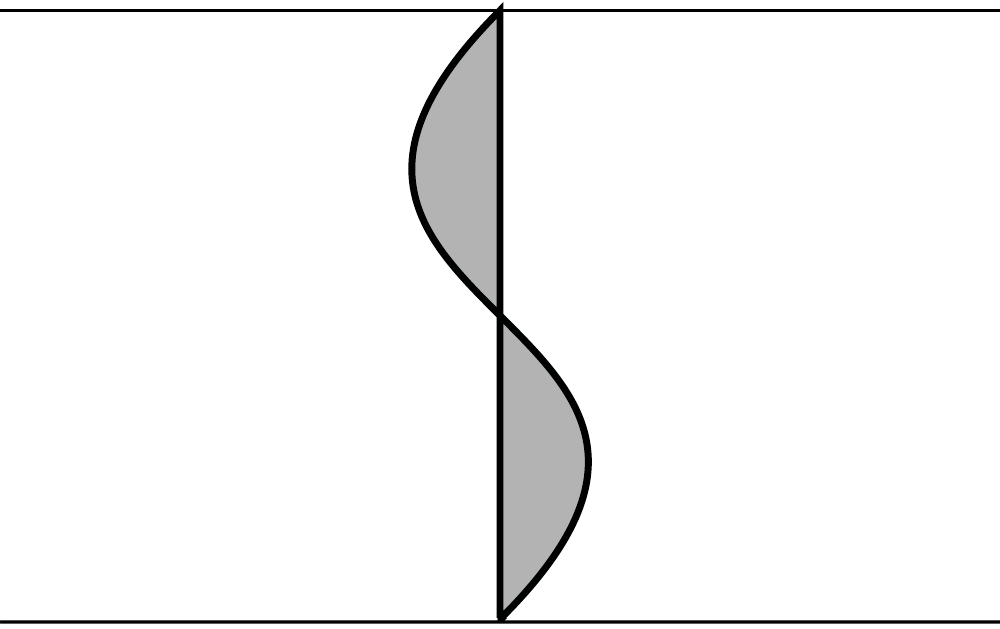}
\put(36.4,45){\tiny $b_i$} \put(51,45) {\tiny $a_i$} \put(52,32.5)
{\tiny $x_i''$} \put(44.7,2.4){\tiny $x_i$} \put(52,57.3){\tiny
$x_i'$}
\end{overpic}
\caption{A portion of $-S_1$. The shaded regions are the ``thin strips'' $D_i$ and
$D_i'$ which connect $x_i''$ to $x_i$ or $x_i'$.} \label{darkside}
\end{figure}

All the intersection points of $\bs{\alpha}$ and $\bs{\beta}$ lie in $S_0$, with the exception of the points $x_i''=x_i''\times \{ 1\}$.  We place the basepoint $z$ on
$S_1$, away from the ``thin strips'' $D_i$ and $D_i'$, $i=1,\dots,2g$, given in
Figure~\ref{darkside}. The positioning of $z$ prevents holomorphic curves involved in
the differential for $\widehat{CF} (-\Sigma, \bs{\alpha}, \bs{\beta}, z)$ --- besides the ones corresponding to the thin strips --- from entering $-S_1$.  Hence all of the nontrivial holomorphic curve information is concentrated on $S_0$. The homology of $\widehat{CF}(-\Sigma, \bs{\alpha}, \bs{\beta},z)$ is isomorphic to $\widehat{HF}(-M)$ since we reversed the orientation of $\Sigma$.

Let $\mathcal{S}_{{\bf a}, \hh({\bf a})}\subset \mathcal{S}_{\bs{\alpha}, \bs{\beta}}$ consist of the $2g$-tuples of intersection points $\mathbf{y}$ all of whose components are in $S_0$. We define\footnote{The same chain complex is called $\widehat{CF}^\prime(S, \mathbf{a}, \hh(\mathbf{a}))$ in \cite{I}.} $\overline{CF}(S, {\bf a}, \hh({\bf a}))$ as the chain complex generated by $\mathcal{S}_{{\bf a},\hh({\bf a})}$ and whose differential counts HF-curves  in $\R\times [0,1]\times S$. Then we define $\widehat{CF}(S, {\bf a}, \hh({\bf a}))$ as the quotient of $\overline{CF}(S, {\bf a}, \hh({\bf a}))$ modulo the identifications $\{ x_i\}\cup \mathbf{y}_0 \sim\{x_i'\}\cup \mathbf{y}_0$ for all $(2g-1)$-tuples of chords $\mathbf{y}_0$. The quotient inherits a differential because no holomorphic curve in $\R \times [0,1] \times S$ can have a positive end at one of the chords over $x_i, x_i'$ unless it has an irreducible component which is a trivial strip over that chord by \cite[Claim 4.9.2]{I}, and therefore the elements of the form $\{ x_i\}\cup \mathbf{y}_0 + \{x_i'\}\cup \mathbf{y}_0$ generate a subcomplex.

The identifications above are the algebraic consequence of the thin strips $D_i$ and $D_i'$. The following proposition was proved in \cite[Theorem 4.9.4]{I}:

\begin{prop}\label{prop:hf}
$\widehat{HF}(S,\mathbf{a},\hh(\mathbf{a}),z)\cong \widehat{HF}(-\Sigma, \bs{\alpha}, \bs{\beta}, z)$.\footnote{We consider $\widehat{HF}(-\Sigma, \bs{\alpha}, \bs{\beta}, z)$ here and $\widehat{HF}(\Sigma, \bs{\beta}, \bs{\alpha}, z)$ in \cite{I}. The two groups are isomorphic.}
\end{prop}

\begin{rem}
$\mathbf{x} =\{x_1, \ldots, x_{2g}\}$ is a cycle and its homology class $[\mathbf{x}]
\in \widehat{HF} (-M)$ is the contact invariant of the contact structure supported by the open book decomposition; see~\cite{HKM}.
\end{rem}
\section{Embedded contact homology}\label{sec: ECH}
\subsection{A review of embedded contact homology} \label{subsec: ECH}
In this section we briefly review the embedded contact homology groups associated to a closed $3$-manifold $M$.  Embedded contact homology was defined by Hutchings~\cite{Hu1}, partially in collaboration with Taubes~\cite{HT1,HT2}, and is intimately connected with the dynamics of a Reeb vector field. 

Let $\alpha$ be a contact form on $M$. We assume that every closed Reeb orbit of $\alpha$ is {\em nondegenerate}, which means that its linearised first return map does not have $1$ as an eigenvalue. Contact forms satisfying this property are called {\em nondegenerate}.
The linearised first return map is a symplectic transformation, and therefore its eigenvalues are $\{ \lambda, \lambda^{-1} \}$, where $\lambda$ is either real or in the unit circle. A closed Reeb orbit is:
\begin{itemize}
\item {\em hyperbolic} if the eigenvalues of its linearised first return map are real, or
\item {\em elliptic} if they lie on the unit circle.
\end{itemize}
The chain complex $ECC(M, \alpha)$ is generated by finite  sets $\bs{\gamma}=\{(\gamma_i,m_i)\}$, called {\em orbit sets}, where:
\begin{itemize}
\item $\gamma_i$ is a simple closed Reeb orbit,
\item $m_i$ is a positive integer, and
\item if $\gamma_i$ is a hyperbolic orbit, then $m_i=1$.
\end{itemize}
We will denote the set of orbit sets by ${\mathcal O}$. An orbit set $\bs{\gamma}$ will also be written multiplicatively as $\prod \gamma_i^{m_i}$, with the convention that
$\gamma_i^2 = 0$ whenever $\gamma_i$ is hyperbolic. The empty orbit set
$\varnothing$ will be written multiplicatively as $1$.

We choose an almost complex structure $J$ on $\R\times M$ compatible with $\alpha$.
Let $\bs{\gamma}_+ = \{ (\gamma_i^+, m_i^+) \}$ and $\bs{\gamma}_- = \{ (\gamma_i^-, m_i^-) \}$ be orbit sets. We denote by ${\mathcal M}_{\R \times M}(\bs{\gamma}_+, \bs{\gamma}_-)$ the moduli space of $J$-holomorphic curves in $\R \times M$ which are positively asymptotic to covers of the closed Reeb orbits $\gamma_i^+$ with total multiplicity $m_i^+$ as $s \to + \infty$ and negatively asymptotic to covers of the closed Reeb orbits $\gamma_i^-$ with total multiplicity $m_i^-$ as $s \to - \infty$. Moreover, we consider equivalent two $J$-holomorphic curves which differ only by (branched) covers of trivial cylinders with the same multiplicity. Such curves are called {\em connectors}.
Translations in the $\R$ direction acts on the moduli spaces; we denote the quotient by $\widehat{\mathcal M}_{\R \times M}(\bs{\gamma}_+, \bs{\gamma}_-)$.

The ECH index for $J$-holomorphic curve in $\R \times M$ is defined in \cite[Definition~1.5]{Hu1}. The following lemma is a consequence of the index inequality proved in \cite[Theorem 4.15]{Hu2}.
\begin{lemma}[{\cite[Proposition 7.15]{HT1}}]
Let $J$ be a generic almost complex structure on $\R \times M$ compatible with $\alpha$. Then:
\begin{enumerate}
\item A $J$-holomorphic curve $u$ with $I(u)=0$ is a union of branched covers of trivial cylinders over simple closed Reeb orbits (i.e. a connector).
\item A $J$-holomorphic curve $u$ with $I(u)=1$ (respectively $I(u)=2$) is a disjoint union of a connector and an embedded $J$-holomorphic curve $u'$ with $I(u')=\op{ind}(u')=1$ (respectively $I(u')=\op{ind}(u')=2$).
\end{enumerate}
\end{lemma}
The ends of a $J$-holomorphic curve $u$ in ${\mathcal M}_{\R \times M}(\bs{\gamma}_+, \bs{\gamma}_-)$ without connector components determine partitions of the multiplicities of the elliptic orbits in $\bs{\gamma}_+$ and $\bs{\gamma}_-$. It turns out that, when $I(u)=1$ or $I(u)=2$, these partitions must coincide with preferred partitions called the {\em outgoing} and {\em incoming} partitions for positive and negative ends, respectively.
The incoming and outgoing partitions can be computed from the dynamics of the linearised Reeb flow. For their definition see \cite[Section 4.1]{Hu1} or \cite[Definition 4.14]{Hu2}. For the relation between these partitions and the ECH index see \cite[Theorem 4.15]{Hu2}, for example. In this article we will not need the precise definition of those partitions, except for the following fact, which is a direct consequence of \cite[Definition 4.14]{Hu2}.

\begin{lemma}\label{basic fact about partitions}
Let $\gamma$ be a simple elliptic orbit and suppose that its linearised Reeb flow is conjugated to a rotation by an angle $2 \pi \theta$. If $0< \theta < {1\over m}$, then the incoming partition of $(\gamma, m)$ is $(m)$ and the outgoing partition is $(\underbrace{1, \ldots, 1}_{m})$. On the other hand, if $- {1\over m} < \theta < 0$, then the incoming partition of $(\gamma, m)$ is $(\underbrace{1, \ldots, 1}_m)$ and the outgoing partition is $(m)$.
\end{lemma}

The differential on $ECC(M, \alpha)$ is defined as
$$\partial \bs{\gamma}_+ = \sum \limits_{\bs{\gamma}_-\in {\mathcal O}} \# \widehat{\mathcal M}_{\R \times M}^{I=1}(\bs{\gamma}_+, \bs{\gamma}_-) \bs{\gamma}_-.$$
The map $\partial$ was shown to satisfy $\partial^2=0$ by Hutchings  and Taubes in \cite{HT1,HT2}. The homology of
$ECC(M,\alpha)$ is the {\em embedded contact homology}
group $ECH(M)$. Independence of the choice of the contact
form $\alpha$, the  contact structure $\xi$, and the compatible almost
complex structure $J$ was shown by Taubes in \cite{T2.1, T2.2, T2.3, T2.4, T2.5} by proving an isomorphism between embedded contact homology and monopole Floer homology.
\begin{rem}
The moduli spaces ${\mathcal M}(\bs{\gamma}_+, \bs{\gamma}_-)$ can be nonempty only if $\bs{\gamma}_+$ and $\bs{\gamma}_-$ define the same homology class. Thus $ECC(M, \alpha)$ decomposes as a direct sum of subcomplexes indexed by classes in $H_1(M; \Z)$. This decomposition is analogous to the decomposition of the Heegaard Floer (or monopole) chain complexes according to Spin$^c$-structures. The induced decomposition in $ECH(M)$ is independent of all choices, except for the indexing by elements in $H_1(M; \Z)$ which depends on the contact structure only through its Euler class.
\end{rem}

Given a point $\mathbf{z} \in \R \times M$, we denote by ${\mathcal M}(\bs{\gamma}_+, \bs{\gamma_-}; \mathbf{z})$ the moduli space of $J$-holomorphic curves asymptotic to $\bs{\gamma}_\pm$ and passing through $\mathbf{z}$ where, as before, we identify curves which differ by a connector component.  For a generic point $\mathbf{z}$, the map\footnote{This map is called $U$ in \cite{binding} and $U'$ in \cite{III}.} $U_\mathbf{z} \colon ECC(M, \alpha) \to ECC(M, \alpha)$ is defined as
$$U_{\mathbf{z}} \bs{\gamma}_+ = \sum \limits_{\bs{\gamma}_- \in {\mathcal O}} \# {\mathcal M}^{I=2}_{\R \times M}(\bs{\gamma}_+, \bs{\gamma}_-; \mathbf{z}) \bs{\gamma}_-.$$
The same techniques used to show that $\partial^2=0$ also show that $U$ is a chain map; see \cite[Section 2.5]{HT5} for more details. Then $\widehat{ECC}(M, \alpha)$ is defined as the cone of $U_{\mathbf{z}}$ and $\widehat{ECH}(M)$ is its homology.
As before, the work of Taubes in \cite{T2.1} and the following articles shows that $\widehat{ECH}(M)$ is invariant of all choices.

\subsection{Morse-Bott theory for embedded contact homology}
In several steps of the proof of Theorem \ref{main} we will use Morse-Bott techniques extensively. Here we give a brief description of those techniques and refer the reader to \cite{Bo2} and \cite{BEHWZ} for a general treatment and to \cite[Section 4]{binding} for one which is more adapted to embedded contact homology. For the purpose of this paper, a contact form is {\em Morse-Bott} if every closed orbit of its Reeb vector field is either isolated and nondegenerate, or belongs to an $S^1$-family and is nondegenerate in the normal direction.\footnote{In general, there is also the case where the Reeb orbits come in two-dimensional families; however this will not occur here.} We denote a Morse-Bott $S^1$-family of simple closed Reeb orbits by $\mathcal{N}$ and the Morse-Bott torus corresponding to $\mathcal{N}$ by $T_{\mathcal{N}}= \cup_{x\in \mathcal{N}} x$.

Let $\{v_1,v_2\}$ be an oriented basis for $\xi$ at some point $p \in T_{\mathcal{N}}$ so that $v_1$ is transverse to $T_{\mathcal{N}}$ and $v_2$ is tangent to $T_{\mathcal{N}}$.
The linearised first return map of the Reeb flow on $\xi_p$ is given, in the basis $\{v_1,v_2\}$, by the matrix
$\begin{pmatrix} 1 & 0
\\ a & 1 \end{pmatrix}$
with $a \ne 0$.

\begin{dfn} \label{defn: positive negative torus}
$T_{\mathcal N}$ is called a {\em positive} Morse-Bott torus if $a>0$ and a {\em negative} Morse-Bott torus if $a<0$.
\end{dfn}
A Morse-Bott contact form $\alpha$ can be perturbed into nondegenerate forms $\alpha_\epsilon$, for $\epsilon >0$ sufficiently small, which depend on the choice of a Morse function on each Morse-Bott $S^1$-family. The close Reeb orbits of the forms $\alpha_\epsilon$ will be the nondegenerate closed Reeb orbits of $\alpha$ together with the closed Reeb orbits in the Morse-Bott tori corresponding to the critical points of the Morse functions on the Morse-Bott families. We will sweep under the carpet the actual construction of the contact forms $\alpha_\epsilon$ and the fact that they are nondegenerate only up to orbits of some action (i.e.\ period) $L_\epsilon$, with $L_\epsilon \to +\infty$ as $\epsilon \to 0$, so that a limiting procedure is involved in computing $ECH(M)$ from a Morse-Bott contact form.

If we choose a family of almost complex structures $J_\epsilon$ compatible with $\alpha_\epsilon$ and converging to an almost complex structure $J$ compatible with $\alpha$ as $\epsilon \to 0$, a sequence of $J_\epsilon$-holomorphic curves converges to a ``cascade'' of $J$-holomorphic curves some of whose ends are connected to negative gradient trajectories in the Morse-Bott families. The negative gradient trajectories which can appear are of three types:
\begin{itemize}
\item flow-lines between critical points,
\item semi-infinite trajectories between a critical point and an end of a holomorphic piece of the cascade, and
\item finite trajectories joining a positive and a negative end of different holomorphic pieces.
\end{itemize}
See \cite[Section 11.2]{BEHWZ} for a precise statement.

The converse is more delicate: a cascade of $J$-holomorphic curves whose ends are joined by negative gradient flow trajectories in general cannot be deformed to a $J_\epsilon$-holomorphic curve without using some abstract perturbation theory which has not been rigorously established yet. The main problem is that a family of simply covered $J_\epsilon$-holomorphic curves --- which are thus regular for a generic $J_\epsilon$ --- could degenerate into a cascade in which some of the holomorphic curves are multiply covered and could even have negative Fredholm index.

We will therefore restrict our attention to a special class of Morse-Bott cascades in which finite negative gradient trajectories are not allowed and at most one of the holomorphic pieces is not a cover of a trivial cylinder. Those cascades are called {\em very nice Morse-Bott buildings} in \cite{binding}.
Let $\bs{\gamma}_\pm = \{ (\gamma_i^\pm, m_i^\pm) \}$ be orbit sets where each $\gamma_i^\pm$ is either a nondegenerate closed Reeb orbit of $\alpha$ or a closed Reeb orbit corresponding to a critical point of the Morse function on a Morse-Bott family. We denote by\footnote{In \cite{binding} the same moduli spaces are denoted  by ${\mathcal M}^{MB, vn}(\bs{\gamma}_+, \bs{\gamma}_-)$ because they consist of very nice Morse-Bott buildings.} ${\mathcal M}^{MB}_{\R \times M}(\bs{\gamma}_+, \bs{\gamma}_-)$ the moduli space of very nice Morse-Bott buildings in $\R \times M$ with positive ends at $\bs{\gamma}_+$ and negative ends at $\bs{\gamma}_-$.
We denote by $\widehat{\mathcal M}^{MB}_{\R \times M}(\bs{\gamma}_+, \bs{\gamma}_-)$ the quotient by translations and by ${\mathcal M}^{MB}_{\R \times M}(\bs{\gamma}_+, \bs{\gamma}_-; \mathbf{z})$ the subspace of those cascades passing through a generic point $\mathbf{z} \in \R \times M$. In the Appendix to \cite{binding} we prove the following.
\begin{thm}
For a generic almost complex structure $J$ compatible with $\alpha$ and $\epsilon$ sufficiently small there are bijections
\begin{align*}
\widehat{\mathcal M}^{MB, I=1}_{\R \times M}(\bs{\gamma}_+, \bs{\gamma}_-) & \cong \widehat{\mathcal M}^{I=1}_{\R \times M}(\bs{\gamma}_+, \bs{\gamma}_-) \quad \text{and} \\
{\mathcal M}^{MB, I=2}_{\R \times M}(\bs{\gamma}_+, \bs{\gamma}_-; \mathbf{z}) & \cong {\mathcal M}^{I=2}_{\R \times M}(\bs{\gamma}_+, \bs{\gamma}_-; \mathbf{z})
\end{align*}
where the moduli spaces on the left-hand side are defined using the almost complex structure $J$ and the moduli spaces on the right-hand side are defined using a perturbed almost complex structure $J_\epsilon$.
\end{thm}
\subsection{Embedded contact homology of manifolds with torus boundary}\label{subsec: ECH with boundary}
In this subsection we define several flavours of embedded contact homology for manifolds with
torus boundary; see \cite[Section 7]{binding} for more details. Let $N$ be a
three-manifold with $\partial N \cong T^2$ and let $\alpha$ be a contact form on $N$ such
that $\partial N$ is foliated by Reeb orbits. We assume that the foliation is linear for some choice of coordinates in $\partial N$. Then, we say that $\alpha$ is {\em rational} if the Reeb orbits on $\partial N$ are closed, and {\em irrational} if they are dense. If $\alpha$ is rational, we assume that $\partial N$ is a Morse-Bott torus. We choose a Morse function with unique maximum and minimum on the corresponding Morse-Bott $S^1$-family. After perturbing $\alpha$ using this Morse function, the Morse-Bott family of orbits foliating $\partial N$ is replaced by a pair of orbits: one elliptic called $e$ and one hyperbolic called $h$. To perform the perturbation it is convenient to enlarge $N$ slightly; we will ignore this technical point. If $\partial N$ is a positive Morse-Bott torus, then $e$ comes from the maximum and $h$ from the minimum, while if $\partial N$ is negative, then $h$ comes from the maximum and $e$ from the minimum. In the following definitions we will assume that $\partial N$ is a negative Morse-Bott torus when $\alpha$ is rational. The similar definitions in the positive case are left to the reader.

\begin{dfn}
The chain complex $ECC(\op{int}(N), \alpha)$ is generated by orbit sets contained in the interior
of $N$ and the differential is defined from  $J$-holomorphic curves in $\R \times \op{int}(N)$. 
\end{dfn}
$ECC(\op{int}(N), \alpha)$ is a chain complex because the foliation by Reeb orbits in $\partial N$ prevents $J$-holomorphic curves with asymptotics in $\op{int}(N)$ from touching $\partial N$ by positivity of intersection. (Alternatively, we can say that $\R \times \partial N$ is a Levi-flat\footnote{Named after Eugenio Elia Levi, an Italian mathematician who was killed in action during WWI. His brother Beppo Levi, also a mathematician, was forced into exile by the fascist racial laws of 1938.} surface.)

\begin{dfn}
The chain complex $ECC(N, \alpha)$ coincides with $ECC(\op{int}(N), \alpha)$ if $\alpha$ is irrational, and is generated by orbit sets built from closed Reeb orbits in $\op{int}(N)$ plus $e$ and $h$ on $\partial N$. The differential counts very nice Morse-Bott buildings contained in
$\R \times N$. 
\end{dfn}
The proof that $ECC(N, \alpha)$ is a chain complex when $\alpha$ is rational consists in showing that a family of very nice Morse-Bott buildings in $\R \times N$ cannot break into a non very nice Morse-Bott building; see \cite[Lemma 7.12]{binding}. This is a consequence of the fact that, by the trapping lemma \cite[Lemma 5.3.2]{binding}, positive ends of $J$-holomorphic curves in $\R \times N$ asymptotic to the closed Reeb orbits foliating $\partial N$ must be trivial.

\begin{dfn}
The chain complex $ECC^\flat(N, \alpha)$ is generated by orbit sets built  from closed Reeb orbits in $\op{int}(N)$ plus $e$ and the chain complex $ECC^\sharp(N, \alpha)$ is generated by orbit sets built  from closed Reeb orbits in $\op{int}(N)$ plus $h$. 
\end{dfn}
Both $ECC^\flat(N, \alpha)$ and $ECC^\sharp(N, \alpha)$ are subcomplexes of $ECC(N, \alpha)$ because $e$ and $h$ can appear only at the negative end of a very nice Morse-Bott building, except for connectors and two negative gradient flow trajectories from $h$ to $e$ in their Morse-Bott family.

Invariance was addressed only for $ECH(N, \alpha)$ when $\alpha$ is an irrational contact form; for the other flavours it will not be needed.
\begin{prop}[{\cite[Proposition 7.2.1]{binding}}]\label{prop: ECH of two contact forms in M}
Let $\alpha_1$ and $\alpha_2$ be contact forms on $N$  which agree on $\partial N$ to first order (and in particular the Reeb vector fields  and the characteristic foliations of $\alpha_1$ and $\alpha_2$ on $\partial N$ are equal)  and define contact structures $\xi_i = \ker \alpha_i$ which are isotopic relative to the boundary. If $\partial N$ is foliated by Reeb orbits of irrational slope, then there is an isomorphism $$ECH(N, \alpha_1) \cong ECH(N, \alpha_2).$$
\end{prop}

The strategy of the proof is to extend $(N, \alpha_i)$, $i=1,2$, to closed contact manifolds so that the closed Reeb orbits not contained in $N$ have much larger action than those contained in $N$, and to use the action properties of the continuation maps for embedded contact homology of closed three-manifolds.

When $\partial N$ is a negative Morse-Bott torus, we define two further versions of embedded contact homology for $(N, \alpha)$ which are, in some sense, embedded contact homology groups relative to the boundary. We recall that $ECC(N, \alpha)$ can be seen as a commutative algebra generated by simple closed Reeb orbits with the relation that $\gamma^2=0$ if $\gamma$ is hyperbolic. Let $\langle e-1 \rangle$ be the ideal generated by $e-1$. Even if the differential does not respect the multiplicative structure, this ideal is a subcomplex because only connectors can have $e$ at a positive end. 
\begin{dfn}
We define
\begin{align*}
ECC(N, \partial N, \alpha)&= ECC^\flat(N, \alpha)/ \langle e-1 \rangle \quad \text{and} \\
 \widehat{ECC}(N, \partial N, \alpha)&= ECC(N, \alpha)/ \langle e-1 \rangle.
\end{align*}
\end{dfn}
We denote the differential in $ECC(N, \partial N, \alpha)$ by $\partial_{rel}$ and in $\widehat{ECC}(N, \partial N, \alpha)$ by $\widehat{\partial}_{rel}$. Factoring out $h$, we can decompose $\widehat{\partial}_{rel}$ as follows. If $\bs{\gamma}$ is an orbit set not containing $h$, we can write
\begin{align*}
\widehat{\partial}_{rel}(\bs{\gamma})&= \partial_{rel}(\bs{\gamma})+hU_{rel}(\bs{\gamma}) \\
 \widehat{\partial}_{rel}(h\bs{\gamma})&= h \widehat{\partial}_{rel}(\bs{\gamma}).
\end{align*}
From $(\widehat{\partial}_{rel})^2=0$ we deduce that the map
$$U_{rel} \colon ECC(N, \partial N, \alpha) \to ECC(N, \partial N, \alpha)$$
is a chain map. Its name is motivated by the fact that it plays the role of the map $U$ in embedded contact homology relative to the boundary.

Given $a \in H_1(N; \Z)$, let $ECC_a(N, \alpha)$ be the subcomplex of $ECC(N, \alpha)$ generated by orbit sets in the homology class $a$. The same notation is used for the sharp ($\sharp$) and flat ($\flat$) flavours. From the fact that only connectors can have $e$ at a positive end, it is easy to see that the map
\begin{align*}
ECC_a(N, \alpha) & \to ECC_{a+[e]}(N, \alpha) \\
\bs{\gamma} & \mapsto e \bs{\gamma}
\end{align*}
is a chain map.

 If $\overline{a} \in H_1(N; \Z)/ \langle [e] \rangle$, let $ECC_{\overline{a}}(N, \partial N, \alpha)$ be the subcomplex of $ECC(N, \partial N, \alpha)$ generated by orbit sets in the class $\overline{a}$. Note that $H_1(N; \Z)/ \langle [e] \rangle$ is the first homology group of the closed manifold obtained by Dehn filling $N$ along the slope of $e$. The following lemma is almost immediate.
\begin{lemma}\label{hat as direct limit}
Suppose $[e] \neq 0$ in $H_1(N; \Z)$. If $a \in H_1(N; \Z)$ and $\overline{a}$ is its image in $H_1(N; \Z)/ \langle [e] \rangle$, then
\begin{align*}
ECH_{\overline{a}}(N, \partial N, \alpha) & = \varinjlim \left \{ ECH_{a+j[e]}^\flat(N, \alpha) \right \}_{j \in \N} \\
\widehat{ECH}_{\overline{a}}(N, \partial N, \alpha) & = \varinjlim \left \{ ECH_{a+j[e]}(N, \alpha) \right \}_{j \in \N}.
\end{align*}
\end{lemma}
The relation between embedded contact homology of $N$ relative to the boundary and embedded contact homology of the Dehn filling of $N$ along the slope of $e$ is the main result of \cite{binding}. The following theorem is a generalisation of \cite[Theorem 1.1.1]{binding}, expressed in a slightly different language.
\begin{thm}\label{main theorem of binding}
Let $N$ be a three-manifold with torus boundary and let $\alpha$ be a contact form on $N$ such that $\partial N$ is a negative Morse-Bott torus for the Reeb vector field.
Let moreover $M$ be the Dehn filling of $N$ with respect of the slope of $e$.
If there is a cohomology class $\varphi \in H^1(N; \Z)$ such that $\varphi([\gamma]) \ge 0$ for all closed Reeb orbits $\gamma$ and $\varphi([e])>0$, then there exist isomorphisms
\begin{align*}
\sigma_* & \colon ECH(N, \partial N, \alpha) \to ECH(M), \\
\widehat{\sigma}_* & \colon \widehat{ECH}(N, \partial N, \alpha) \to \widehat{ECH}(M)
\end{align*}
such that the diagram
$$\xymatrix{
\ldots  \ar[r] &\widehat{ECH}(N, \partial N) \ar[r] \ar[d]^{\widehat{\sigma}_*} & ECH(N, \partial N) \ar[r]^-{U_{rel}} \ar[d]^{\sigma_*} & ECH(N, \partial N) \ar[d]^{\sigma_*} \ar[r] & \ldots \\
\ldots \ar[r] & \widehat{ECH}(M) \ar[r] & ECH(M) \ar[r]^{U} &  ECH(M) \ar[r] & \ldots
}$$
commutes. Moreover, the maps $\sigma_*$ and $\widehat{\sigma}$ are compatible with the direct sum decompositions according to homology classes.
\end{thm}
The proof of Theorem \ref{main theorem of binding} will be sketched in the next subsection. A heuristic reason why it is expected to hold is the following. Let $\alpha_\delta$, $0 <\delta \ll 1$, be a family of contact forms on $M=N \cup V$ such that, on the solid torus $V \cong D^2 \times (\R/\Z)$ with cylindrical coordinates $(r,\theta,z)$, their Reeb vector fields $R_\delta$
are tangent to the concentric tori $\{r=const\}$ and have constant slope $\frac{1}{\delta}$
(away from the core).  As we send $\delta\to 0$, the Conley-Zehnder index of the core
goes to $+\infty$ and we should therefore be able to ignore it.
At the same time, one expects that $I_{ECH}=1$ holomorphic curves that cross the core
when $\delta>0$ are in one-to-one correspondence with holomorphic curves which
have $e$ at the negative end when $\delta=0$.  This is the reason for identifying $e=1$. Similarly, if we place the marked point $z$ on the core and $h$ correspondingly, then $I_{ECH}=2$ holomorphic curves that pass through $z$ should be in one-to-one correspondence with $I_{ECH}=1$ holomorphic curves which have $h$ at a negative end when $\delta=0$.
\subsection{How to prove Theorem \ref{main theorem of binding}}
\label{subsec: proof of binding}
\subsubsection{A special contact form on $M$}
Let $\alpha$ be a contact form on $N$ which is nondegenerate in $\op{int}(N)$ and has a negative Morse-Bott torus of closed Reeb orbits at $\partial N$. To keep the discussion closer to \cite{binding} we assume that there is a properly embedded surface $S \subset N$ such that $[\partial S] \cdot [e]=1$ and the Reeb vector field of $\alpha$ is positively transverse to $S$.
This ensures the existence of a cohomology class $\varphi \in H^1(N; \Z)$ such that
$\varphi(\gamma) \ge 0$ and $\varphi([e])=1$.

We decompose $M= N \cup (T^2 \times [1,2]) \cup V$ where $V$ is a solid torus whose meridian is attached to the slope of $e$. We identify $T^2 \times \{ 1 \}$ with $\partial N$ and $T^2 \times \{ 2 \}$ with $\partial V$. The region $T^2 \times (1,2)$ will be called the {\em no man's land}.\footnote{The name was suggested to us by a visit to the remains of Berlin's wall.}  We choose a contact form $\alpha_M$ on $M$ such that:
\begin{enumerate}
\item $\alpha_M$ restricts to $\alpha$ on $N$, so that, in particular, $\partial N$ is a negative Morse-Bott torus,
\item the closed Reeb orbits in the no man's land have arbitrarily large action and foliate the tori $T^2 \times \{ c \}$ for $c \in (1,2)$,
\item  $\alpha_V=\alpha_M|_{V}$ is nondegenerate in $\op{int}(V)$ and $\partial V$ is a positive Morse-Bott torus,
\item every closed Reeb orbit in $\op{int}(V)$ is transverse to a meridional disc of $V$,
\item there is an exhaustion of concentric solid tori $V_0 \subset V_1 \subset \ldots \subset \op{int}(V)$ such that $\partial V_i$ is linearly foliated by Reeb orbits of irrational slope $r_i$ with $r_i \to + \infty$ as $i \to + \infty$ (for some choice of coordinates in which $e$ has slope $+\infty$) and all closed Reeb orbits in $V \setminus V_i$ have slope larger that $r_i$.
\end{enumerate}
On the Morse-Bott $S^1$-family of closed Reeb orbits corresponding to $\partial V$ we choose a Morse function with a unique minimum and a unique maximum and denote by $e'$ the (elliptic) orbit corresponding to the maximum and $h'$ the (hyperbolic) orbit corresponding to the minimum.

Condition (2) on $\alpha_M$ implies that the closed Reeb orbits in the no man's land can effectively be ignored in computing $ECH(M)$. In order to prove this rigorously we need a direct limit argument, and therefore we need a family of contact forms $\alpha_M$ for which the closed Reeb orbits in the no man's land have larger and larger action. This direct limit is somewhat involved because it must be repeated at each step of the proof of Theorem \ref{main theorem of binding}, and therefore we prefer to ignore it. For this reason, from now on, every statement in this section will hold after an unexpressed direct limit. See \cite[Sections 9.2 and 9.3]{binding} for the detailed construction of the contact forms.

After ignoring the closed Reeb orbits in the no man's land, $ECC(M, \alpha_M)$ becomes isomorphic to $ECC(V, \alpha_V) \otimes ECC(N, \alpha)$ with a differential $\partial_M$ which is, however, not the usual differential of a tensor product of complexes.

\subsubsection{A filtration}
The next step is to introduce a filtration\footnote{The idea of this filtration was suggested to us by Michael Hutchings.} ${\mathcal F}$ on $ECC(V, \alpha_V) \otimes ECC(N, \alpha_N)$ to simplify the differential. To construct this filtration, we choose the generator $\eta \in H^1(V; \Z)$ which evaluates positively on the closed Reeb orbits in $\op{int}(V)$ and define ${\mathcal F}^p$ as the subspace generated by orbit sets $\bs{\gamma} \otimes \bs{\Gamma}$ with $\eta([\bs{\gamma}]) \le p$.
Note that ${\mathcal F}^p = \{ 0 \}$ for $p <0$ and ${\mathcal F}^0$ is generated by orbit sets of the form $\bs{\gamma} \otimes \bs{\Gamma}$ where $\bs{\gamma}$ contains only orbits in $\partial V$.
\begin{lemma}[{\cite[Corollary 9.4.2]{binding}}]
The differential $\partial_M$ preserves the vector spaces ${\mathcal F}^p$ for all $p$.
\end{lemma}
\begin{proof}[Sketch of proof]
Let $u$ be an embedded holomorphic curve in $\R \times M$ between $\bs{\gamma}_+ \otimes \bs{\Gamma}_+$  and $\bs{\gamma}_-  \otimes \bs{\Gamma}_-$. We denote by $[u] \in H_2(M, \bigcup \gamma^+_i \cup \bigcup \gamma_j^-)$ the relative homology class determined by the projection of $u$ to $M$. If $\bs{\gamma}_+$ and $\bs{\gamma}_-$ do not contain orbits in $\partial V$, then it is easy to see that
$$\eta([\bs{\gamma}_+])- \eta([\bs{\gamma}_-]) = [u] \cdot [e'] \ge 0$$
by positivity of the intersections between the image of $u$ and the holomorphic cylinder $\R \times e'$. If  $\bs{\gamma}_+$ or $\bs{\gamma}_-$ contain orbits in $\partial V$ the proof is more subtle because one has to use positivity of intersection with the nonclosed Reeb orbits on the tori $\partial V_i$.
\end{proof}
The next step is to describe the differential induced in the graded complex or, in fancy words, in the zero page of the spectral sequence associated to the filtration. For that we use Morse-Bott techniques, which are justified by the following lemma.
\begin{lemma}[See {\cite[Corollary 9.5.2]{binding}}] \label{noia}
A Morse-Bott building $u$ between $\bs{\gamma}_+ \otimes \bs{\Gamma}_+$ and $\bs{\gamma}_- \otimes \bs{\Gamma}_-$ with $I(u)=1$ and $\eta(\bs{\gamma}_+)=\eta(\bs{\gamma}_-)$ is very nice. Moreover its projection to $M$ is contained either in $N$ or in $V$ or in the no man's land.
\end{lemma}
\begin{proof}
If $I(u)=1$ and $u$ is not very nice, it must have two irreducible components joined by a Morse trajectory of finite length: the condition $I(u)=1$ rules out the possibility that $u$ is not very nice because of two irreducible components which are not connectors and are not linked by a Morse trajectory, since those two components could be translated independently in the $\R$-direction. Then $u$ must have either a positive end at an orbit of $\partial N$ or a negative end at an orbit of $\partial V$. In either case, positivity of intersection with nearby Reeb orbits implies that $u$ approaches that orbit with a framing different from the framing induced by the Morse-Bott torus, and therefore a topological argument forces $\eta(\bs{\gamma}_+) > \eta(\bs{\gamma}_-)$. For the details see \cite[ Lemma~9.5.1 and Corollary~9.5.2]{binding}. The last claim holds because the tori in the no man's land foliated by Reeb orbits form a barrier for the holomorphic curve by the
blocking lemma \cite[Lemma 5.2.3]{binding}, which is a consequence of positivity of intersection between the projection of $u$ to $M$ and the Reeb vector field. For the details of this last argument see the proof of \cite[Lemma 9.5.3]{binding}.
\end{proof}

The holomorphic curves in the symplectisation of the no man's land can be described using a finite energy foliation first constructed by Wendl in \cite{We}; see \cite[Section 8.4]{binding}: there is a foliation of $\R \times T^2 \times [1,2]$ by a cylinders with positive end at the closed Reeb orbits in $\partial V$ and negative end at the closed Reeb orbits in $\partial N$. After a Morse-Bott perturbation, this foliation gives two $I=1$ holomorphic curves in $\R \times T^2 \times [1,2]$: one cylinder from $e'$ to $h$ and one cylinder from $h'$ to $e$. Moreover, the foliation obstructs the existence of any other embedded holomorphic curve in $\R \times T^2 \times [1,2]$.
The same holds for Morse-Bott buildings $u$ with $I(u)=2$ and constrained to pass through a generic point in $\R \times M$.

 Given two  orbit sets $\bs{\gamma}'=\prod\gamma_i^{m_i'}$ and $\bs{\gamma}=\prod \gamma_i^{m_i}$  (in multiplicative notation), we set $\bs{\gamma}/\bs{\gamma}'= \prod\gamma_i^{m_i-m_i'}$ if $m_i'\leq m_i$ for all $i$; otherwise we set $\bs{\gamma}/\bs{\gamma}'=0$.
 We denote $E^p_0= {\mathcal F}^p/ {\mathcal F}^{p+1}$; then $E_0= \bigoplus E_0^p$ is a chain complex with differential $\partial_0$ induced by $\partial_M$. The differential $\partial_0$ respects the direct sum decomposition by construction, and we denote by $E_1^p$ the homology of $E_0^p$.

As a vector space we can still identify $E_0 \cong ECC(V, \alpha_V) \otimes ECC(N, \alpha)$.
More precisely, if we denote the subcomplex of $ECC(V, \alpha_V)$ generated by orbit sets $\bs{\gamma}$ such that $\eta([\bs{\gamma}])=p$ by $ECC_p(V, \alpha_V)$, then $E_0^p \cong ECC_p(V, \alpha_V) \otimes ECC(N, \alpha)$.

From Lemma \ref{noia} and the description of the $I=1$ holomorphic curves in the symplectisation of the no man's land, it follows that the differential $\partial_0$ is
\begin{equation}\label{differential in degree 0}
 \partial_0(\bs{\gamma} \otimes \bs{\Gamma}) =  (\partial_V \bs{\gamma})\otimes \bs{\Gamma }+ (\bs{\gamma}/e') \otimes h\bs{\Gamma} + (\bs{\gamma}/h')\otimes e \bs{\Gamma} + \bs{\gamma} \otimes (\partial_{N} \bs{\Gamma}),
\end{equation}
where $\partial_N$ and $\partial_V$ denote the differential of $ECC(N, \alpha)$ and $ECC(V, \alpha_V)$ respectively. See \cite[Lemma 9.5.3]{binding}.

\subsubsection{Computation of $E_1$}
By factoring out the terms $h'$ and $h$, we can write the differentials $\partial_V$ and $\partial_N$ as:
\begin{equation} \label{eqn: decomposition of boundary}
\left \{
\begin{array}{l} \partial_V \bs{\gamma} = \partial^{\flat}_V \bs{\gamma} \\
\partial_V (h' \bs{\gamma}) = h' \partial_V^{\flat} \bs{\gamma} +
\partial'_V(h' \bs{\gamma}) \end{array} \right. \qquad \left \{
\begin{array}{l} \partial_N \bs{\Gamma} = \partial^{\flat}_N
\bs{\Gamma} + h \partial_N' \bs{\Gamma} \\ \partial_N (h \bs{\Gamma})= h
\partial^{\flat}_N \bs{\Gamma}
\end{array} \right.
\end{equation}
where $\bs{\gamma} \in ECC^{\flat}(V, \alpha_V)$, $\bs{\Gamma} \in ECC^{\flat}(N, \alpha)$, $\partial^\flat_V$ and $\partial^\flat_N$ are the differentials for the chain complexes $ECC^{\flat}(V, \alpha_V)$ and $ECC^{\flat}(N, \alpha)$ respectively, and the terms $\partial'_V(h' \bs{\gamma})$ and $\partial'_N \bs{\Gamma}$ do not contain $h'$.

 If we write
$$C_{k,k'}^p= (h')^{k'} ECC^{\flat}_p(V, \alpha) \otimes h^k ECC^{\flat}(N, \alpha)$$
with $(k,k') \in \{ 0,1 \}^2$, then
we can describe the differential $\partial_0$ on $E_0^p$ by the following diagram
\begin{equation} \label{diagramma 1} \xymatrix{
C_{0,1}^p \ar[rrr]^{1 \otimes h\partial'_N + \cdot / e' \otimes h}
\ar[d]_{\partial'_V \otimes 1 + \cdot / h' \otimes e} & & &
C_{1,1}^p \ar[d]^{\partial'_V \otimes 1 + \cdot / h' \otimes e} \\
C_{0,0}^p \ar[rrr]^{1 \otimes h\partial_N' + \cdot / e' \otimes h} & &
& C_{1,0}^p} \end{equation}
where each $C_{k,k'}^p$ carries the internal differential $\partial^\flat_V \otimes 1 + 1 \otimes \partial_N^\flat$. Thus we can filter each $E_0^p$ by $k'-k$ and the first term of the corresponding spectral sequence (i.e.\ the homology of the graded complex associated to the filtration) is $ECH^\flat_p(V, \alpha_V) \otimes ECH^\flat(N, \alpha)$.

We make a digression into the embedded contact homology of the solid torus.
\begin{lemma}[{\cite[Lemma 8.1.2(4)]{binding}}] \label{ECH of V} $ECH^\flat(V, \alpha_V) \cong \Z / 2 \Z[e']$.
 \end{lemma}
\begin{proof}[Sketch of proof] We start by computing $ECH(\op{int}(V), \alpha_V)$. We recall the exhaustion $V_0 \subset V_i \subset \ldots \subset \op{int}(V)$ by concentric solid tori. Each $\partial V_i$ is linearly foliated by Reeb orbits with irrational slope. Positivity of intersection with the Reeb flow implies that holomorphic curves with ends in $V_i$ are contained in $V_i$ and thus, by \cite[Lemma 8.3.1]{binding}, we have
$$ECH(\op{int}(V), \alpha_V) \cong \varinjlim ECH(V_i, \alpha_V|_{V_i}).$$
Since $\alpha_V|_{V_i}$ is irrational, Proposition \ref{prop: ECH of two contact forms in M} applies, and therefore  we can compute $ECH(V_i, \alpha_V|_{V_i})$ using a different contact form with a unique simple closed Reeb orbit; see \cite[Section 8.2]{binding}. For those contact forms the ECH index $I$ can be lifted to an absolute index on closed orbits, which is moreover preserved by continuation maps; see  \cite[Lemma~8.2.2 and Lemma~8.2.3]{binding}. A direct computation shows that the ECH index of elements in $ECC(V_i, \alpha_V|_{V_i})$, besides the class of the empty set, grows with $i$. Hence in the direct limit
only the empty set survives, and therefore we have $ECH(\op{int}(V), \alpha_V) \cong \Z / 2 \Z$. See \cite[Proposition 8.3.2]{binding}.

In order to compute $ECH^\flat(V, \alpha_V)$ we restrict the filtration ${\mathcal F}$ induced by $\eta \in H^1(V; \Z)$ to $ECC^\flat(V, \alpha_V)$. The induced spectral sequence collapses at $E^1 = \Z / 2 \Z[e'] \otimes ECH(\op{int}(V), \alpha_V) \cong \Z / 2 \Z[e']$. See \cite[Section 8.5]{binding}.
\end{proof}
This implies that $ECH^\flat(V, \alpha_V) \otimes ECH^\flat(N, \alpha)=0$ for $p>0$ and therefore $E_0^p=0$ for $p>0$. Thus the spectral sequence induced by ${\mathcal F}$ collapses at the $E_1$ term and the inclusion ${\mathcal F}^0 \subset ECC(M, \alpha_M)$ induces an isomorphism $E_1^0 = H({\mathcal F}^0) \cong ECH(M)$.

Now we compute $E_1^0$. Given a finite set of simple closed Reeb orbits $\gamma_1, \ldots, \gamma_n$ we define ${\mathcal R}[\gamma_1, \ldots, \gamma_n]$ as the polynomial algebra generated by those orbits over $\Z / 2 \Z$ with the relation that $\gamma^2=0$ if $\gamma$ is hyperbolic. Using this notation, $ECC_0(V, \alpha_V)={\mathcal R}[e', h']$. There are only three holomorphic curves in $V$ with $I=1$ and asymptotic to $e'$ or $h'$: two cylinders from $e'$ to $h'$ coming from gradient flow trajectories in the Morse-Bott family corresponding to $\partial V$ and one holomorphic plane positively asymptotic to $h'$. See \cite[Proposition~8.4.4 and Proposition~8.4.5]{binding} for the construction of this plane. Thus $\partial_V((e')^i)=0$ and $\partial_V((e')^ih')=(e')^i$, so we can write
$E_0^0 \cong {\mathcal R}[e', h'] \otimes ECC(N, \alpha)$ with differential
\begin{equation}\label{differential of E_0^0}
\partial_0 (\bs{\gamma} \otimes \bs{\Gamma})= \bs{\gamma}\otimes (\partial_N \bs{\Gamma}) + (\bs{\gamma}/h') \otimes (1+e) \bs{\Gamma} + (\bs{\gamma}/e') \otimes h \bs{\Gamma}.
\end{equation}
A simple algebraic argument now implies that $E_0^1 \cong ECH(N, \partial N, \alpha)$. In fact $\bs{\gamma} \otimes \bs{\Gamma}$ is a cycle only if $\bs{\gamma}=1$ and $\partial_N \bs{\gamma}=0$; the third term of Equation \eqref{differential of E_0^0} implies that, if $h \bs{\gamma}$ is a cycle, then it is a boundary, and the second term implies that $e \bs{\Gamma}= \bs{\Gamma}$ in homology. 

\subsubsection{Definition of $\sigma$}
The argument of the previous paragraphs  gives a noncanonical isomorphism $ECH(M, \alpha_M) \cong ECH(N, \partial N, \alpha)$. Now we show that this isomorphism is in fact induced by a geometrically meaningful chain map. We define the complex
$$ECC^\natural(N, \alpha) = {\mathcal R}[h'] \otimes ECC^\flat(N, \alpha)$$
with differential
\begin{equation}\label{boundary natural}
\partial^\natural(\bs{\gamma} \otimes \bs{\Gamma})= \bs{\gamma} \otimes \partial^\flat_N(\bs{\Gamma}) + \bs{\gamma}/h' \otimes (1+e) \bs{\Gamma}.
\end{equation}
It is easy to show that $ECH^\natural(N, \alpha) \cong ECH(N, \partial N, \alpha)$.

We define
\begin{equation}\label{definition of sigma}
\sigma (\bs{\gamma} \otimes \bs{\Gamma})= \sum \limits_{i=0}^\infty (e')^i \bs{\gamma} \otimes (\partial_N')^i \bs{\Gamma}.
\end{equation}
The sum is well defined because $(\partial_N')^{k+1}(\bs{\Gamma})=0$ if $\varphi([\bs{\Gamma}])=k$, where $\varphi$ is a class in $H^1(N; \Z)$ such that $\varphi([e])=1$.

Since $\bs{\gamma}=(h')^j$ for $j=0,1$, the image of $\sigma$ is contained in ${\mathcal F}^0$, which is the first nonzero group of the filtration, and therefore to show that $\sigma$ is a chain map it is enough to verify that $\sigma \circ \partial^\natural = \partial_0 \circ \sigma$. This is an easy verification using Equation \eqref{differential in degree 0} and the fact that $\partial_N'$ commutes with $\partial_N^\flat$ and the multiplication by $e$. Finally one can prove without much effort by algebraic considerations similar to those of the previous paragraph that the map $\sigma$ induces an isomorphism in homology.
\subsubsection{The $U$ map and the hat version}
We recall that the $U$ map in $ECC(M, \alpha_M)$ counts $I=2$ holomorphic curves in $\R \times M$ passing through a generic base point $z \in \R \times M$. In order to simplify the computation, we put the base point in the  symplectisation of the no man's land $\R \times T^2 \times (1,2)$. The following lemma is proved by the same techniques we used to study the differential.
\begin{lemma}[See {\cite[Lemma 9.9.3]{binding}}]
The map $U$ preserves the filtration ${\mathcal F}$. In the lowest filtration level ${\mathcal F}^0$, it is given by
\begin{equation}\label{U map on F^0}
U(\bs{\gamma} \otimes \bs{\Gamma})= \bs{\gamma}/e' \otimes \bs{\Gamma}.
\end{equation}
\end{lemma}
The unique holomorphic curve contributing to $U$ in the lowest filtration level is a cylinder from $e'$ to $e$ belonging to the finite energy foliation of $\R \times T^2 \times [1,2]$ which was already used to understand the differential $\partial_0$. This foliation is described in \cite[Section 8.4]{binding}.

On the other hand, we define the map $U^\natural : ECC^\natural(N, \alpha) \to ECC^\natural(N, \alpha)$ by
\begin{equation}\label{U natural}
U^\natural(\bs{\gamma} \otimes \bs{\Gamma}) = \bs{\gamma} \otimes \partial_N' \bs{\Gamma}.
\end{equation}
By comparing Equation \eqref{U natural} with Equations \eqref{boundary natural}, \eqref{definition of sigma} and \eqref{U map on F^0}, we obtain that $U^\natural$ is a chain map and that the diagram
$$\xymatrix{
ECC^\natural(N, \alpha) \ar[r]^{U^\natural} \ar[d]_{\sigma} & ECC^\natural(N, \alpha) \ar[d]_{\sigma} \\
ECC(M, \alpha_M) \ar[r]^U & ECC(M, \alpha_M)
}$$
commutes. The map $U^\natural$ corresponds to $U_{rel}$ under the isomorphism $ECH^\natural(N, \alpha) \cong ECH(N, \partial N, \alpha)$, and therefore we have established part of Theorem \ref{main theorem of binding}. 

We also define the chain complex $$\widehat{ECC}{}^\natural (N, \alpha) =  {\mathcal R}[h'] \otimes ECC(N, \alpha)$$ with differential
$$\widehat{\partial}^\natural (\bs{\gamma} \otimes \bs{\Gamma}) =\bs{\gamma} \otimes \partial_N \bs{\Gamma} + \bs{\gamma} /h' \otimes (1+e) \bs{\Gamma}.$$
 The decomposition of the differential $\partial_N$ described in Equation~\eqref{eqn: decomposition of boundary} implies that $\widehat{ECC}{}^\natural (N, \alpha)$ is the cone of the map $U^\natural$. Moreover one can easily prove that $\widehat{ECH}{}^\natural(N, \alpha) \cong \widehat{ECH}(N, \partial N, \alpha)$, and therefore the remaining part of Theorem \ref{main theorem of binding} follows from the naturality of the mapping cone construction.
\subsection{Sutured embedded contact homology}\label{subsec: sutured}
Sutured embedded contact homology is another version of embedded contact homology for manifolds with boundary which was defined in \cite{CGHH}.
Sutured manifolds are three-manifolds with decorated boundary which were introduced by Gabai in \cite{gabai} for the study of foliations. Convex surfaces were introduced by Giroux in \cite{giroux} as a tool for studying contact structures. The parallel between the theory of sutured manifolds and convex surfaces has been explored since  \cite{HKM: taut}.

 In this subsection we relate sutured embedded contact homology to the versions of embedded contact homology introduced in Subsection \ref{subsec: ECH with boundary}. First, we recall the definition of a family of sutured manifolds which is less general than those considered by Gabai but closer to the theory of convex surfaces.
\begin{dfn}
A {\em balanced sutured manifold} $(M, \Gamma)$ is a three-manifold with boundary $M$ decorated by an embedded one-dimensional oriented submanifold $\Gamma \subset \partial M$ such that:
\begin{itemize}
\item each connected component of $M$ intersects $\partial M$,
\item each connected component of $\partial M$ intersects $\Gamma$,
\item if $N(\Gamma)$ denotes a closed tubular neighbourhood of $\Gamma$ in $\partial M$, then
$$\partial M \setminus \op{int}(N(\Gamma)) = R_+(\Gamma) \sqcup R_-(\Gamma),$$
\item if we orient $R_+(\Gamma)$ with the orientation of $\partial M$ and $R_-(\Gamma)$ with the opposite orientation, then the orientation of $\Gamma$ coincides with the orientation of $\partial R_+(\Gamma)$ and is opposite to the orientation of $R_-(\Gamma)$, and finally
\item $\chi(R_+(\Gamma)) = \chi(R_-(\Gamma))$.
\end{itemize}
\end{dfn}
\begin{rem}
If $(M, \xi)$ is a contact three-manifold with convex boundary and $\Gamma$ is the dividing set $\partial M$, then $(M, \Gamma)$ is a balanced sutured manifold by \cite{giroux}. 
\end{rem}

It is often useful to think of a sutured manifold as a manifold with corners, where $R_+(\Gamma) \cup R_-(\Gamma)$ is the ``horizontal boundary'', $N(\Gamma)$ is the ``vertical boundary'' and $\partial  R_+(\Gamma) \cup \partial R_-(\Gamma) = \partial N(\Gamma)$ are the corners. The following definition is best visualised in this point of view.
\begin{dfn}[See {\cite[Definition 2.8]{CGHH}}]
A contact form $\alpha$ is {\em adapted} to the balanced sutured manifold $(M, \Gamma)$ if $\alpha|_{R_\pm(\Gamma)}$ are Liouville forms on $R_\pm(\Gamma)$ and the Reeb vector field is transverse and outward pointing on $R_+(\Gamma)$, transverse and inward pointing at $R_-(\Gamma)$ and tangent to $N(\Gamma)$, so that $N(\Gamma)$ is foliated
by Reeb trajectories going from $R_-(\Gamma)$ to $R_+(\Gamma)$.
\end{dfn}
The simplest example of a sutured manifold with an adapted contact form is the following: we take a compact,  oriented surface with boundary $P$ and we define $M=P \times [-1,1]$ and $\Gamma = \partial P \times \{ 0 \}$. We identify $R_\pm(\Gamma)=P \times \{ \pm 1 \}$ and $N(\Gamma)=\partial P \times [-1,1]$. The contact form is $\alpha= \lambda + dt$, where $\lambda$ is a Liouville form on $P$. However many more interesting examples exist: in fact, contact sutured manifolds are a fairly general concept, as the following proposition shows.
\begin{prop}[See {\cite[Lemma 4.1]{CGHH}}]
Let $(M, \xi)$ be a contact manifold with convex boundary and no closed connected component, and let $\Gamma$ be the dividing set of $\partial M$. Then 
$\xi$ admits a contact form which is adapted to the balanced sutured manifold $(M, \Gamma)$.
\end{prop}
The interest of sutured contact manifolds is that they form a fairly large class of contact manifolds for which SFT and ECH compactness hold: see \cite[Corollary 5.19 and Corollary 5.21]{CGHH}. In \cite{CGHH} we considered sutured contact manifolds of any dimension, and that level of generality introduced extra complications in the proof of compactness. In dimension three, however, the story is fairly simple: all we need to show is that a $J$-holomorphic curve in $\R \times M$ which is asymptotic to closed Reeb orbits of $\alpha$ cannot touch $\R \times \partial M$ if $J$ belongs to a suitable class of almost complex structures. Roughly speaking, we define a function $t$ on a neighbourhood of $\partial M$ by integrating the Reeb vector field, and a function $\tau$ in a neighbourhood on $N(\Gamma)$ such that, in that neighbourhood, we can write $\alpha = dt + e^\tau d \theta$ with $\theta$ a coordinate on $\Gamma$. We say that an almost complex structure $J$ on $\R \times M$ is {\em tailored} to $(M, \Gamma, \alpha)$ if it is compatible with $\alpha$ and moreover $t$ is harmonic and $\tau$ is subharmonic with respect to the Laplacian induced by $J$; see \cite[Section 5.1, Section 5.2 and Section 5.3]{CGHH}. Thus, if $u$ is a $J$-holomorphic curve which is asymptotic to closed Reeb orbits, the maximum principle applied to $t \circ u$ and $\tau \circ u$ is enough to show that $u$ cannot intersect the region where $t$ and $\tau$ are defined. It is easy to see that the space of tailored almost complex structures is nonempty and contractible.

All other aspects of the theory of $J$-holomorphic curves (i.e.\ regularity, Fredholm theory, transversality, gluing, etc.) are semi-local in nature, and therefore do not change if the ambient manifold is not closed. Thus, we can define the {\em sutured embedded contact homology} complex $ECC(M, \Gamma, \alpha)$ as the vector spaces over $\Z / 2 \Z$ generated by orbit sets in $M$ with differential counting $J$-holomorphic curves in $\R \times M$ with ECH index $I=1$ for a tailored almost complex structure $J$. These complexes were defined in \cite{CGHH} and were inspired by sutured Floer homology, which was defined by Juh\'asz in the context of Heegaard Floer homology in \cite{juhasz: sfh}. We conjectured the following.
\begin{conj}[{\cite[Conjecture 1.5]{CGHH}}] \label{isomorphism for sutured manifolds}
Let $(M, \Gamma, \alpha)$ be a sutured contact manifold. Then there is an isomorphism
$$ECH(M, \Gamma, \alpha) \cong SFH(-M, -\Gamma),$$
which moreover respects the decomposition into first homology classes on the left-hand side and relative Spin$^c$-structures on the right-hand side.
\end{conj}
As an evidence to this conjecture, we constructed gluing maps which are formally analogous to Juh\'asz's map for sutured decompositions from \cite{juhasz: decomposition} and Honda, Kazez and Mati\'c's map for gluing along convex surfaces from \cite{HKM: gluing}: see \cite[Theorem 1.9 and Theorem 1.10]{CGHH}. The proof is currently a work in progress of the authors with Gilberto Spano.

The sutured embedded contact homology groups are independent of the contact form and the almost complex structure in the following sense.
\begin{thm}[{\cite[Theorem 10.2.2]{binding}}]
Let $\alpha_1$ and $\alpha_2$ be contact forms adapted to the balanced sutured manifold
$(M, \Gamma)$ and let $J_1$ and $J_2$ be almost complex structures on $\R \times M$
such that $J_i$ is tailored to $(M, \Gamma, \alpha_i)$ for $i=1,2$. If $\xi_1 = \ker \alpha_1$ and $\xi_2 = \ker \alpha_2$ are isotopic through contact structures  making $\partial M$ convex with dividing set $\Gamma$, then
$$ECH(M, \Gamma, \alpha_1) \cong ECH(M, \Gamma, \alpha_2),$$
where the first group is defined using $J_1$ and the second using $J_2$.
 Moreover this isomorphism preserves the decomposition of the sutured embedded contact homology groups as direct sums of subgroups indexed by homology classes in $H_1(M)$.
\end{thm}
The proof is similar to the proof of Proposition \ref{prop: ECH of two contact forms in M}. This result has been proved independently, but with similar techniques, by Kutluhan, Sivek and Taubes \cite{KST}, who also prove naturality of the isomorphism.
In light of Conjecture \ref{isomorphism for sutured manifolds} we expect that the sutured contact homology groups should be independent also of the contact structure, but at the moment we are not able to find a direct proof of this more general invariance.

Two more basic but important examples of sutured manifolds are the following:
\begin{itemize}
\item given a closed three-manifold $M$ and an embedded codimension zero ball $B \subset M$, we define the balanced sutured manifold $(M(1), \Gamma_1)$, where $M(1)=M \setminus \op{int}(B)$ and $\Gamma_1$ is a connected curve in $\partial M(1)$;
\item given a closed three-manifold $M$ and a knot $K \subset M$, we define a balanced sutured manifold $(M(K), \Gamma_K)$, where $M(K)$ is the complement of an open tubular neighbourhood of $K$ and $\Gamma_K$ consists of two parallel copies of the meridian of $K$ with opposite orientations.
\end{itemize}
Sutured embedded contact homology for these two sutured manifolds is related to the embedded contact homology groups defined in the previous sections as follows.
\begin{figure} \centering
\begin{overpic}[height=4.5cm]{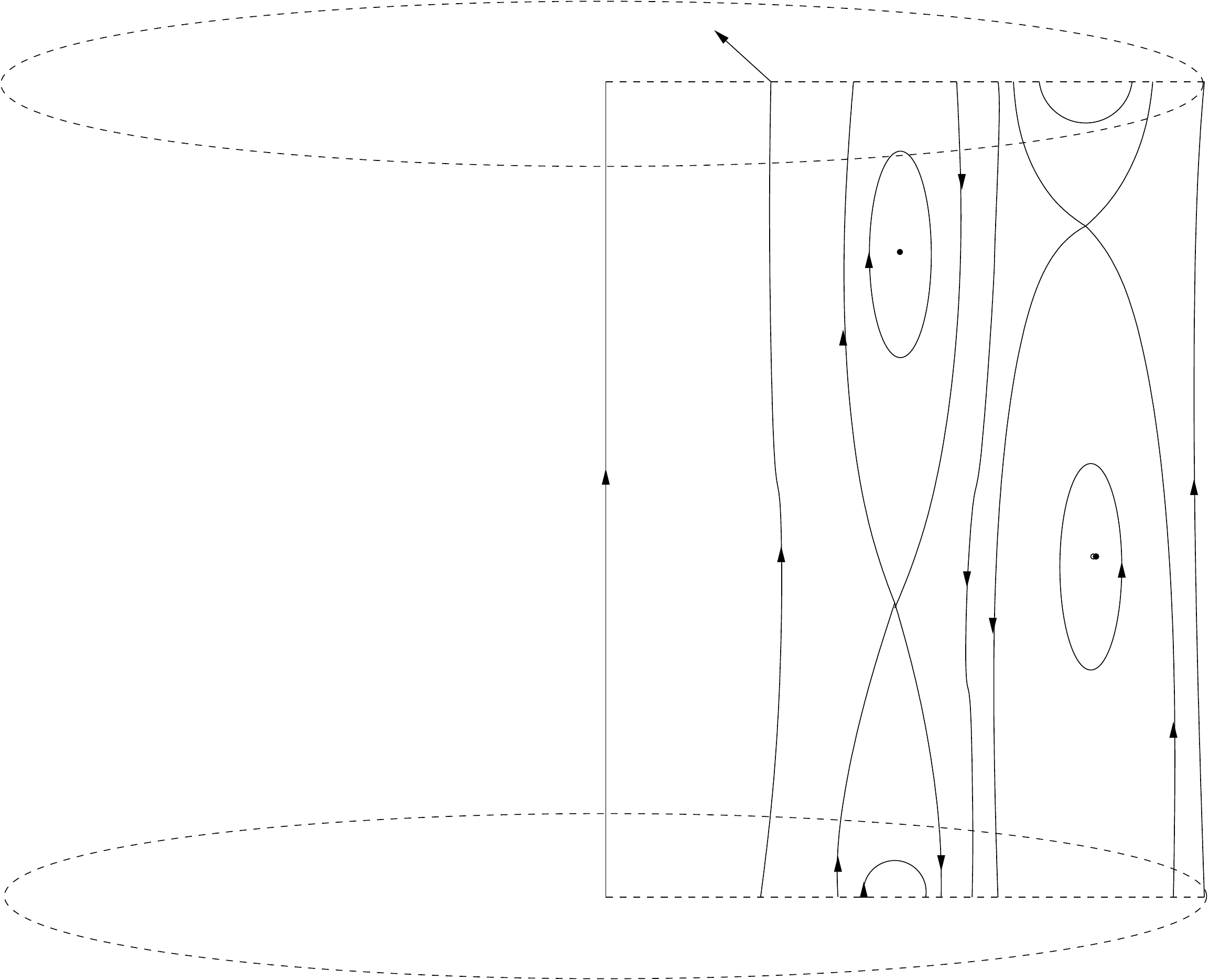}
\put(65.3,59){\tiny $e'$} \put(67.5,30.5) {\tiny $h'$}
\put(92,61){\tiny $h$} \put(84.2,33.6){\tiny $e$}
\put(44,45){\tiny $K$} 
\end{overpic}
\caption{The first return map of the Reeb vector field of $\alpha_M$ on a longitudinal section of a neighbourhood of $K$. The top and the bottom are identified.}
\label{fig: suture}
\end{figure}
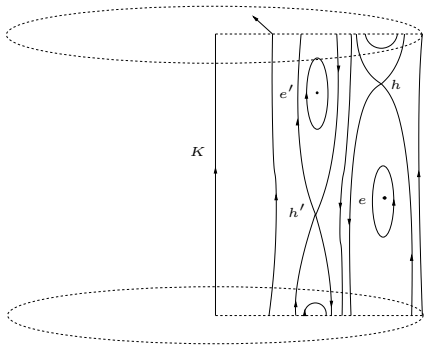
\begin{thm}[{\cite[Theorem 10.3.1 and Theorem 10.3.2]{binding}}]\label{main result of section 10}
Let $M$ be a closed three-manifold, $K \subset M$ a knot and $N$ the complement of an open tubular neighbourhood of $K$ in $M$. Then there exist a contact form on $N$ as in Subsection \ref{subsec: ECH with boundary} and adapted contact forms $\alpha_1$ on $(M(1), \Gamma_1)$ and $\alpha_K$ on $(M(K), \Gamma_K)$ such that
\begin{align*}
ECH(M(1), \Gamma_1, \alpha_1) & \cong \widehat{ECH}(N, \partial N, \alpha) \quad \text{and} \\
ECH(M(K), \Gamma_K, \alpha_K) & \cong ECH^\sharp(N, \alpha).
\end{align*}
\end{thm}
\begin{proof}[Sketch of proof]
Let $\alpha_M$ be a contact form on $M$ as in Subsection \ref{subsec: proof of binding}. The first return map of the Reeb vector field of $\alpha_M$ in a neighbourhood of $K$ containing $\partial N$ and after the Morse-Bott perturbation is depicted in Figure \ref{fig: suture}.
The sutured manifold $(M(1), \Gamma_1)$ is identified to the subset of $M$ described in the left side of Figure \ref{fig: suture12}.
\begin{figure}
\begin{tabular}{lcr}
\includegraphics[height=4.5cm]{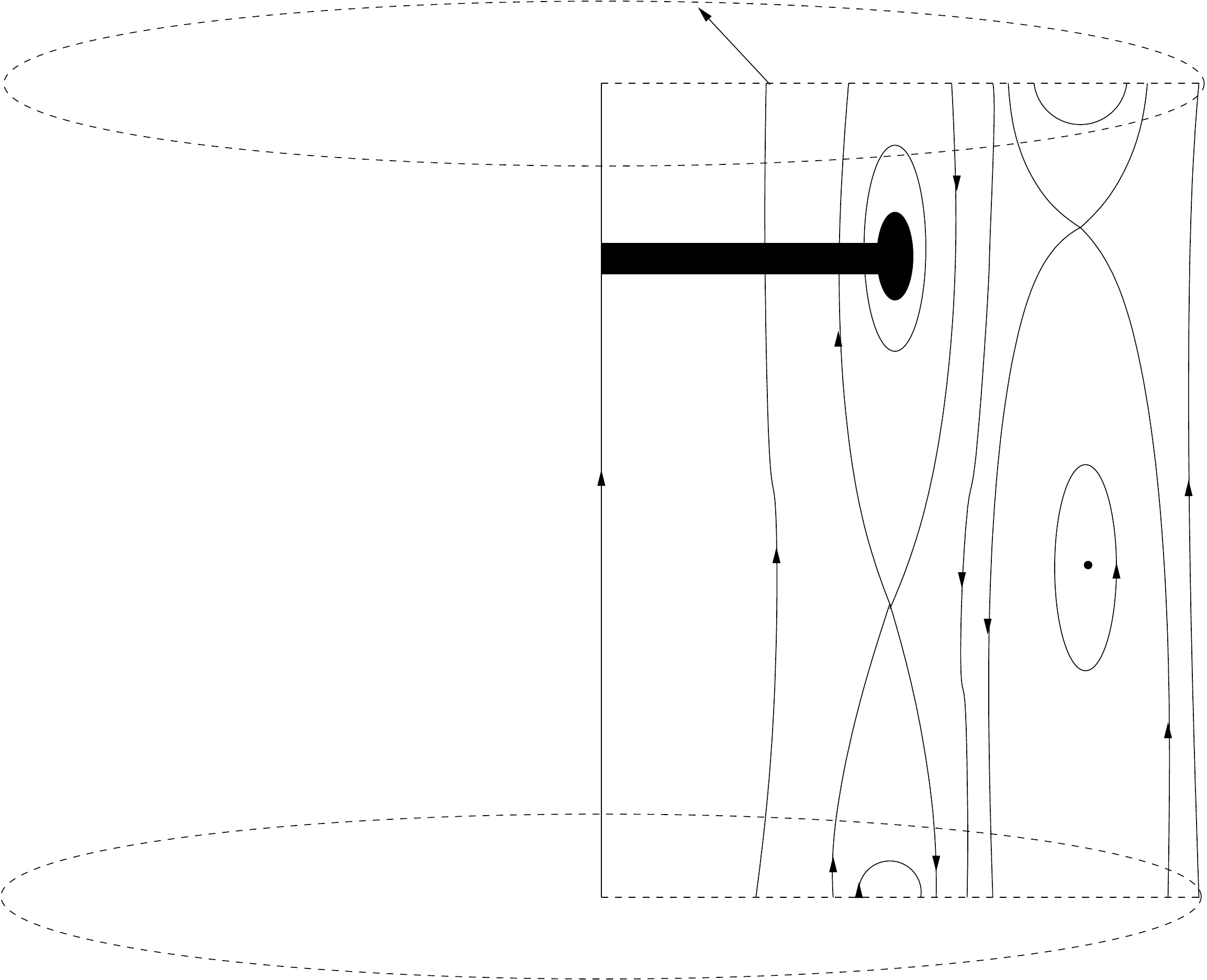} & & \includegraphics[height=4.5cm]{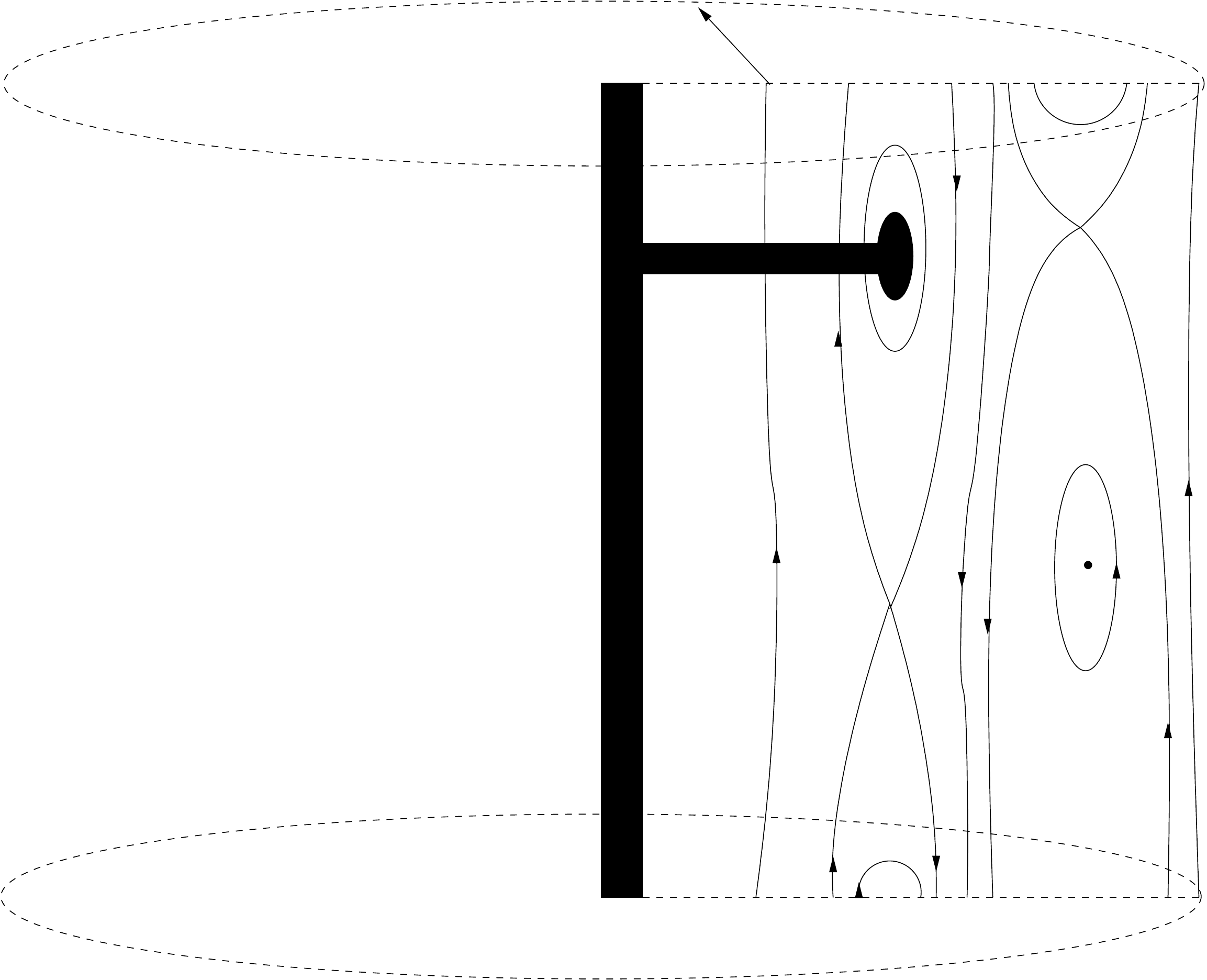}
\end{tabular}
\caption{The ball $B$ and the neighbourhood of $K$ are obtained by rotating the shaded regions on the left and on the right, respectively, around the vertical axis. The top and the bottom are identified.}
 \label{fig: suture12}
\end{figure}
The adapted contact forms $\alpha_1$ is the restrictions of $\alpha_M$ to $M(1)$. However, the contact form $\alpha_K'$ obtained by restricting $\alpha_M$ to $M(K)$ is not adapted to the suture because its restriction to $R_\pm(\Gamma_K)$ is not a Liouville form. This can be seen from the fact that the contact structure is negatively transverse to the component of $\partial R_\pm(\Gamma_K)$ which is closer to $K$. The issue can be corrected by adding a layer to that component in which the Reeb vector field remains constant and the contact structure rotates until it satisfies the sutured condition. The new sutured manifold is diffeomorphic to $(M(K), \Gamma_K)$ and the contact form $\alpha_K$ we obtain is adapted to $(M(K), \Gamma_K)$.

Both $ECC(M(1), \Gamma_1, \alpha_1)$ and $ECC(M(K), \Gamma_K, \alpha_K)$ are generated by orbit sets containing closed Reeb orbits in $N$ and $h'$; moreover the holomorphic plane in $\R \times M$ which is positively asymptotic to $h'$ does not contribute to the differential of $ECC(M(K), \Gamma_K, \alpha_K)$. At this point, the result follows from algebraic manipulations similar to those in the proof of Theorem \ref{main theorem of binding}, but easier.
\end{proof}
Combining Theorem \ref{main theorem of binding} with Theorem \ref{main result of section 10} we obtain the following corollary.
\begin{cor} $ECH(M(1), \Gamma_1, \alpha_1) \cong \widehat{ECH}(M, \alpha_M)$.
\end{cor}
Combining Conjecture \ref{isomorphism for sutured manifolds} with Theorem \ref{main result of section 10}  and the relation between sutured Floer homology and knot Floer homology, we obtain the following conjecture.
\begin{conj}
If $K \subset M$ is a null-homologous knot of genus $g$, $N$ is the complement of an open tubular neighbourhood of $K$ and $\alpha$ is a contact form on $N$ for which $\partial N$ is a negative Morse-Bott torus foliated by meridians of $K$, then
$$ECH^\sharp_i(N, \alpha) \cong \widehat{HFK}(M, K, i-g),$$
where $ECH^\sharp_i(N, \alpha)$ denotes the homology of the subcomplex of  $ECH^\sharp(N, \alpha)$ generated by orbit sets in $N$ with total linking number $i$ with $K$.
\end{conj}
Both the Alexander grading in knot Floer homology and  the linking number depend on the choice of a Seifert surface if $H_2(M; \Z) \ne 0$. In the conjecture we assume that we have made the same choice in both cases.
\subsection{Periodic Floer homology and open books}\label{subsec: PFH}
We recall from Giroux \cite{giroux-obd} that a contact form is supported by an open book decomposition if its Reeb vector field is tangent to the binding and positively transverse to the interior of the pages. In the same article he also sketched an equivalence between contact structures up to isotopy and open book decompositions up to positive stabilisation, but here we will need only the easy part of the equivalence for our purposes: namely the existence of contact forms supported by open book decomposition and
the isotopy between contact structures supported by open book decompositions related by positive stabilisations.

Let $(S, \hh)$ be an abstract open book decomposition for $M$. We assume, without loss of generality, that $\partial S$ is connected, and identify the mapping torus of $\hh$, denoted by $N$, with the complement of an open tubular neighbourhood of the binding. 
\begin{assumption}
On $\hh$ we assume the following:
\begin{itemize}
\item $\hh^* \beta - \beta$ is exact for some Liouville form $\beta$ on $S$, and
\item $\hh(y, \theta)= (y, \theta - y)$ for coordinates $(y, \theta) \in (- \varepsilon, 0] \times \R / \Z$ in a collar of $\partial S$.
\end{itemize}
\end{assumption}

By a Moser's trick argument one can always find a representative of the monodromy satisfying the conditions above (see \cite[Lemma 9.3.2]{binding}) and, by a refinement of Thurston and Winkelnkemper's construction from \cite{TW}, one can construct a contact form $\alpha$ on $M$ which is supported by the open book decomposition $(S, \hh)$, and such that the first return map of the Reeb flow in $N$ is $\hh$ and $\partial N$ is a negative Morse-Bott torus; see  \cite[Lemma 9.3.3]{binding} and \cite[Lemma 2.12]{I}. Thus the Reeb flow of $\alpha$ on $N$ satisfies the conditions of Subsections \ref{subsec: ECH with boundary} and \ref{subsec: proof of binding} permitting to define $ECH(N, \partial N, \alpha)$ and prove the isomorphism with $ECH(M, \alpha)$.

We define
$$ECH_i(N, \alpha)= \bigoplus \limits_{a \cdot [S]=i} ECH_a(N, \alpha),$$
where $a \in H_1(N)$ and $[S] \in H_2(N, \partial N)$ is the relative homology class of a page of the open book decomposition. Thus, by Lemma \ref{hat as direct limit},
$$\widehat{ECH}(N, \partial N, \alpha) = \varinjlim ECH_i(N, \alpha).$$
In the definition of the isomorphism between $\widehat{HF}(-M)$ and $\widehat{ECH}(M)$ it will be useful to pass to periodic Floer homology groups.

There is a family of stable Hamiltonian structures $(\alpha_\varsigma, \omega_\varsigma)_{\varsigma \in [0,1]}$ such that $\omega_\varsigma = d \alpha_\varsigma$ for $\varsigma \in (0,1]$ (i.e. $\alpha_\varsigma$ is a contact form for $\varsigma \ne 0$), $\alpha_1=\alpha$ and $(\alpha_0, \omega_0)$ is the stable Hamiltonian structure induced by the fibration $\pi \colon N \to S^1$. See \cite[Section 3.1]{I}. Let $J_\varsigma$ be a smooth family of almost complex structures on $\R \times N$ compatible with $(\alpha_\varsigma, \omega_\varsigma)$.

The {\em Periodic Floer homology} complexes $PFC_i(N, \alpha_0, \omega)$ are defined in the same way as the embedded contact homology complexes $ECC_i(N, \alpha)$: the different names have only a historical motivation. We can also define all decorated version of periodic Floer homology as in Subsection \ref{subsec: ECH with boundary}.
By comparing regular $J_0$-holomorphic curves with $J_\varsigma$-holomorphic curves with $\varsigma>0$ sufficiently small, we obtain the following result.
\begin{lemma}[{See \cite[Theorem 3.6.1]{I}}]
For every $i \in \N$ there exists $\varsigma_i \in (0,1]$ such that, for all $\varsigma \in (0, \varsigma_i]$, there is an isomorphism of chain complexes
$$PFC_i(N, \alpha_0, \omega_0) \cong ECC_i(N, \alpha_\varsigma).$$
\end{lemma}
This is not yet enough for taking direct limits because $\varsigma_i$ becomes smaller as $i$ increases and might tend to zero.\footnote{This issue was pointed out to us by Thomas Brown.} However, arguing with more care, we can prove the following.
\begin{lemma} \label{lemma: PFH=ECH} For $\alpha$ and $(\alpha_0, \omega_0)$ as above, there is an isomorphism
$$\varinjlim PFH_i(N, \alpha_0, \omega_0) = \widehat{PFH}(N, \partial N, \alpha_0, \omega_0) \cong \widehat{ECH}(N, \partial N, \alpha).$$
\end{lemma}
\begin{proof}
We need to define continuation maps $ECC_i(N, \alpha) \to ECC_i(N, \alpha_{\varsigma_i})$ for every $i \in \N$ such that the diagram
$$\xymatrix{
ECC_i(N, \alpha) \ar[d]_{\cdot e} \ar[r] & ECC_i(N, \alpha_{\varsigma_i}) \ar@{=}[r] & PFC_i(N, \alpha_0, \omega_0) \ar[d]_{\cdot e} \\
ECC_{i+1}(N, \alpha) \ar[r] & ECC_{i+1}(N, \alpha_{\varsigma_{i+1}}) \ar@{=}[r] & PFC_i(N, \alpha_0, \omega_0)
}$$
commutes. To define the continuation maps $ECC_i(N, \alpha) \to ECC_i(N, \alpha_{\varsigma_i})$, we slightly enlarge $N$ to $N'$ so that $\partial N'$ is foliated by Reeb orbits of irrational slope and no new closed Reeb orbit intersecting a page less than $i+2$ times is created, and invoke Proposition \ref{prop: ECH of two contact forms in M}. The continuation maps are supported on holomorphic curves which are contained in $\R \times M$ endowed with a symplectic form which interpolates between $\alpha$ and $\alpha_{\varsigma_i}$. See \cite[Theorem 3.1.2]{binding} for the definition of ``supported on holomorphic curves'' and an overview on continuation maps in ECH and \cite{HT4} for the full story. By the trapping lemma \cite[Lemma 5.3.2]{binding}, curves with a positive end at $e$ must contain a trivial cylinder over $e$, and therefore the commutativity of the diagram follows.
\end{proof}
\section{The open-closed maps $\Phi$ and $\Phi^+$}\label{sec: OC}
In this section we define maps from Heegaard Floer homology to embedded contact homology. As it often happens in symplectic geometry, these maps will be defined by counting holomorphic curves in symplectic cobordisms. Here the symplectic cobordisms will come from an open book decomposition $(S, \hh)$ of $M$.
\subsection{The maps $\Phi$ and $\widehat{\Phi}_*$}
Let $B_+$ be the unit disc with one puncture in the interior and one puncture on the boundary, which we identify biholomorphically with the subset of the cylinder $\R \times \R/2 \Z$ obtained by rounding the corners of  $ \R \times \R/2 \Z \setminus (2, \infty) \times (1,2)$. We define a fibration
$$\pi_{W_+} \colon W_+ \to B_+$$
with fibre $S$ and monodromy $\hh$, and equip the total space with the symplectic form $\Omega_+= ds \wedge dt + \omega$, where $(s,t)$ are coordinates on $B_+ \subset \R \times \R / 2 \Z$ and $\omega = d \beta$ is an area form on $S$ which is preserved by $\hh$.
We can view $(W_+, \Omega_+)$ as a symplectic cobordism (with boundary) between
$[0,1] \times S$ and $N$. We refer to \cite[Section~5]{I} for more details about this construction.

Let $\mathbf{a}= \{ a_1, \ldots, a_{2g} \}$ be a basis of arcs of $S$. We define a Lagrangian submanifold $L_{\mathbf{a}} \subset \partial W$ as the trace of the parallel transport of $\{ 3 \} \times \{ 1 \} \times \mathbf{a}$ along $\partial B_+$ with respect to the symplectic connection defined as the $\Omega_+$-orthogonal to the tangent spaces of the fibres. Then,
\begin{align*}
L_{\mathbf{a}} \cap [3, + \infty) \times \{ 1 \} \times S & =  [3, + \infty) \times \{ 1 \} \times \mathbf{a} \\
L_{\mathbf{a}} \cap [3, + \infty) \times \{ 0 \} \times S & =  [3, + \infty) \times \{ 0 \} \times \hh(\mathbf{a}).
\end{align*}

In the following $\widehat{\mathcal O}_k$ will denote the set of orbit sets in $N$ with total intersection number $k$ with a fibre.
We fix an almost complex structure $J_+$ on $W_+$ which is compatible with $\Omega_+$ and with the stable Hamiltonian structures at the ends.
 Given a Heegaard Floer generator $\mathbf{y} \in {\mathcal S}_{\mathbf{a}, \hh(\mathbf{a})}$ and a periodic Floer homology generator $\boldsymbol{\gamma} \in \widehat{\mathcal O}_{2g}$, we denote by ${\mathcal M}_{W_+}(\mathbf{y}, \bs{\gamma})$ the moduli space of $J_+$-holomorphic curves in $W_+$ with boundary on $L_{\mathbf{a}}$ which are positively asymptotic to $\mathbf{y}$ and negatively asymptotic to $\boldsymbol{\gamma}$. 
Holomorphic curves in ${\mathcal M}_{W_+}(\mathbf{y}, \bs{\gamma})$ are somewhere injective because distinct connected components of the boundary are on distinct connected components of $L_{\mathbf{a}}$.
Thus, for generic $J_+$, the moduli spaces  ${\mathcal M}_{W_+}(\mathbf{y}, \bs{\gamma})$ are (disjoint unions of) smooth manifolds of the dimension predicted by the Fredholm index.

The ECH-type index $I$ for $J_+$-holomorphic curves in ${\mathcal M}_{W_+}(\mathbf{y}, \bs{\gamma})$ is defined in \cite[Definition 5.6.5]{I}. By the index inequality \cite[Theorem 5.6.9]{I}, $J_+$-holomorphic curves $u$ with $I(u)=0$ are embedded, have Fredholm index $\op{ind}(u)=0$, and therefore are isolated. We define
$$\overline{\Phi} \colon \overline{CF}(S, \mathbf{a}, \hh(\mathbf{a})) \to PFC_{2g}(N, \alpha_0, \omega)$$
such that, on each $\mathbf{y} \in {\mathcal S}_{\mathbf{a}, \hh(\mathbf{a})}$, it is given by
\begin{equation} \label{eqn: definition of Phi}
\overline{\Phi}(\mathbf{y}) = \sum \limits_{\bs{\gamma} \in \widehat{\mathcal O}_{2g}} \# {\mathcal M}_{W_+}^{I=0}(\mathbf{y}, \bs{\gamma}) \bs{\gamma},
\end{equation}

Compactness and gluing for holomorphic curves in Heegaard Floer homology and embedded contact homology extend to holomorphic curves in $W_+$. Therefore the usual argument on the ends of one-dimensional moduli spaces proves that $\overline{\Phi}$ is a chain map; see \cite[Proposition 6.2.2]{I}.

\begin{lemma}[{\cite[Theorem 6.2.4]{I}}] \label{Phi and the contact class}
The map $\overline{\Phi}$ induces a chain map
$$\Phi \colon \widehat{CF}(S, \mathbf{a}, \hh(\mathbf{a})) \to PFC_{2g}(N, \alpha_0, \omega)$$
and moreover $\Phi(\mathbf{x})=e^{2g}$, where $\mathbf{x}$ represents the contact class.
\end{lemma}
\begin{proof}[Sketch of proof]
The statement is a consequence of the following fact, proved in \cite[Lemma 6.2.3]{I}: a $J_+$-holomorphic curve in $W_+$ with boundary on $L_{\mathbf{a}}$ which is positively asymptotic to a chord over an intersection point $x$ on $\partial S$ (i.e. $x=x_i$ or $x=x_i'$ for some $i$) has an irreducible component consisting of a trivial section $B_+ \times \{ x \}$ followed by a gradient flow trajectory connecting the orbit over $x$ to $e$ in the Morse-Bott family associated to $\partial N$. In fact a holomorphic curve with a nonconstant positive end at $x$ cannot be contained in $N$.
\end{proof}
Finally we define the map $\widehat{\Phi}_* \colon \widehat{HF}(-M) \to \widehat{ECH}(M)$ of Theorem \ref{main} as the composition of $\Phi_* \colon \widehat{HF}(S, \mathbf{a}, \hh(\mathbf{a})) \to ECH_{2g}(N, \alpha_0, \omega)$ with
the inclusion
$$PFC_{2g}(N, \alpha_0, \omega) \to \widehat{PFC}(N, \partial N, \alpha_0, \omega) \cong \widehat{ECH}(N, \partial N, \alpha)$$
and the map $\widehat{\sigma}_* \colon \widehat{ECH}(N, \partial N, \alpha) \to \widehat{ECH}(M)$ of Theorem \ref{main theorem of binding}.
\begin{rem}
We could have defined the map $\Phi$ using the contact form $\alpha$ instead of the stable Hamiltonian structure $(\alpha_0, \omega)$ at the negative end of $W_+$. However, working with the stable Hamiltonian structure induced by the fibration will be useful in the construction of the inverse map $\Psi$ in Section \ref{sec: CO}.
\end{rem}
\subsection{The map $\Phi^+$}
The map $\Phi^+$ will be defined by counting holomorphic curves in a symplectic cobordism with boundary $(X_+, \Omega_{X_+})$ from $[0,1] \times \Sigma$ to $M$. The symplectic form $\Omega_{X_+}$ will not be exact, and therefore the count will require some extra care.

\subsubsection{The symplectic cobordism $(X_+, \Omega_{X_+})$}
 We describe a simplified construction of $(X_+, \Omega_{X_+})$ and refer the reader to \cite[Section 4]{III} for the actual details, with the warning that the notations will not correspond exactly.

Let $(\Sigma, \bs{\alpha}, \bs{\beta}, z)$ be the pointed Heegaard diagram for $-M$ associated to the open book decomposition $(S, \hh)$ as in Subsection \ref{HFhat on a page}.\footnote{What we call $\Sigma$ here was called $-\Sigma$ in \cite{III}.} In particular $\Sigma = S_0 \cup -S_1$. For technical reasons we need to assume that $S$ has genus at least two.
We also fix an area form $\omega^+$ on $\Sigma$, whose properties will be specified in the construction of $X_+$.

We decompose $M=N \cup (T^2 \times [1,2]) \cup V$ as in Subsection \ref{subsec: proof of binding} and choose a contact form $\lambda$ on $M$ which is supported by the open book decomposition $(S, \hh)$ and moreover satisfies the following properties:
\begin{enumerate}
\item $\lambda|_N=\alpha$ (in particular $\partial N$ is a negative Morse-Bott torus),
\item $\partial V$ is a positive Morse-Bott torus,
\item Reeb orbits in $\op{int}(V)$ and in the no man's land $T^2 \times (1,2)$ are long\footnote{What ``long'' means will be made clear in Lemma \ref{paperino}.}, and
\item $\alpha$ is close to the stable Hamiltonian structure $(\alpha_0, \omega)$ induced by the fibration $N \to S^1$ with fibre $S$, so that $ECC_{2g}(N, \lambda)$ is isomorphic to $PFC_{2g}(N, \lambda_0, d \lambda)$.
\end{enumerate}
At least at a first approximation, the symplectic cobordism $(X_+, \Omega_+)$ will satisfy the following properties:
\begin{enumerate}
\item the positive end is
$$X_+^{>}= [3, + \infty) \times [0,1] \times \Sigma, \quad \Omega_{X_+}|_{X_+^{>}}= ds \wedge dt+ \omega^+;$$
\item the negative end is
$$X_+^{<}= (- \infty, -1] \times M, \quad \Omega_{X_+}|_{X_+^{<}} = d(e^s \lambda);$$
\item there is a symplectic surface $S_z \subset X_+$ such that
$$S_z \cap X_+^{>}= [3, + \infty) \times [0,1] \times \{ z \} \quad \text{and} \quad S_z \cap X_+^{<}= \emptyset;$$
\item $\Omega_{X_+}|_{X_+ \setminus S_z}= d \Theta$ for some one-form $\Theta$ such that $\Theta|_{X_+^{<}}= e^s \lambda$;
\item there is a Lagrangian submanifold $L_{\bs{\alpha}} \subset \partial X_+$ such that
$$X_+^{>} \cap L_{\bs{\alpha}}= ([3, + \infty) \times \{ 0 \} \times \bs{\beta}) \cup  ([3, + \infty) \times \{ 1 \} \times \bs{\alpha})$$
and $\Theta$ is exact on $L_{\bs{\alpha}}$; and
\item there is an embedding of $W_+$ into $X_+$ such that $W_+ \cap S_z= \emptyset$,
$$X_+^{>} \cap W_+ = [3, + \infty) \times [0,1] \times S_0, \quad X_+^{<} \cap W_+= \R \times N,$$
and the restriction of $\Omega_{X_+}$ to $W_+$ is close to $\Omega_+$ (see the proof of Lemma \ref{paperino}).
\end{enumerate}
We refer to \cite[Lemma 4.1.1]{III} for a more precise description of $(X_+, \Omega_{X_+})$.
We will give here only a sketch of the construction; the details, which are quite delicate, can be found in \cite[Section 4.1]{III}.

We identify $S_0 \cong S$ and extend $\hh \colon S_0 \to S_0$ to a diffeomorphism $\hh^+ \colon \Sigma \to \Sigma$ which, on $S_1$, restricts to a perturbation of the identity by a small Hamiltonian isotopy. We require that $\bs{\beta}= \hh^+(\bs{\alpha})$.
Let $B_+^0= B_+ \cap [0, +\infty) \times S^1$; i.e.\ we cut the negative end out of $B_+$. We define $X_+^0$ as the bundle over $B_+^0$ with fibre $\Sigma$ and monodromy $\hh^+$. If $\omega^+$ is an area form on $\Sigma$ preserved by $\hh^+$, then we define $\Omega_{X^0_+}= ds\wedge dt + \omega^+$. The component of the boundary of $X_+^0$ over $\{ 0 \} \times S^1$ is a three-manifold $N^+$ obtained as mapping torus of $(\Sigma, \hh^+)$. It contain the mapping torus $N$ of $(S, \hh)$ as a closed submanifold of codimension zero.

We recall that in the following construction the orientation of $S_1$ will be the opposite of the orientation as a page of the open book decomposition $(S, \hh)$ of $M$. We define $X_+^1= D^2 \times S_1$ and $\Omega_{X_+^1}= \omega_{D^2} + \omega^+|_{S_1}$ where, of course, $\omega_{D^2}$ is the standard symplectic form on $D^2$. Then we glue $X^0_+$ and $X^1_+$ by identifying $\partial D^1 \times S_1 \subset \partial X^1_+$ with $N^+ \setminus \op{int}(N)$, which is the mapping torus of $(S_1, \hh^+|_{S_1})$. We can glue the symplectic forms $\Omega_{X_+^0}$ and $\Omega_{X_+^1}$ because $\hh^+|_{S_1}$ is Hamiltonian isotopic to the identity. The gluing of $X^0_+$ and $X^1_+$ is schematically depicted in Figure \ref{fig: Xplus}.
\begin{figure}
\begin{center}
\includegraphics[width=6cm]{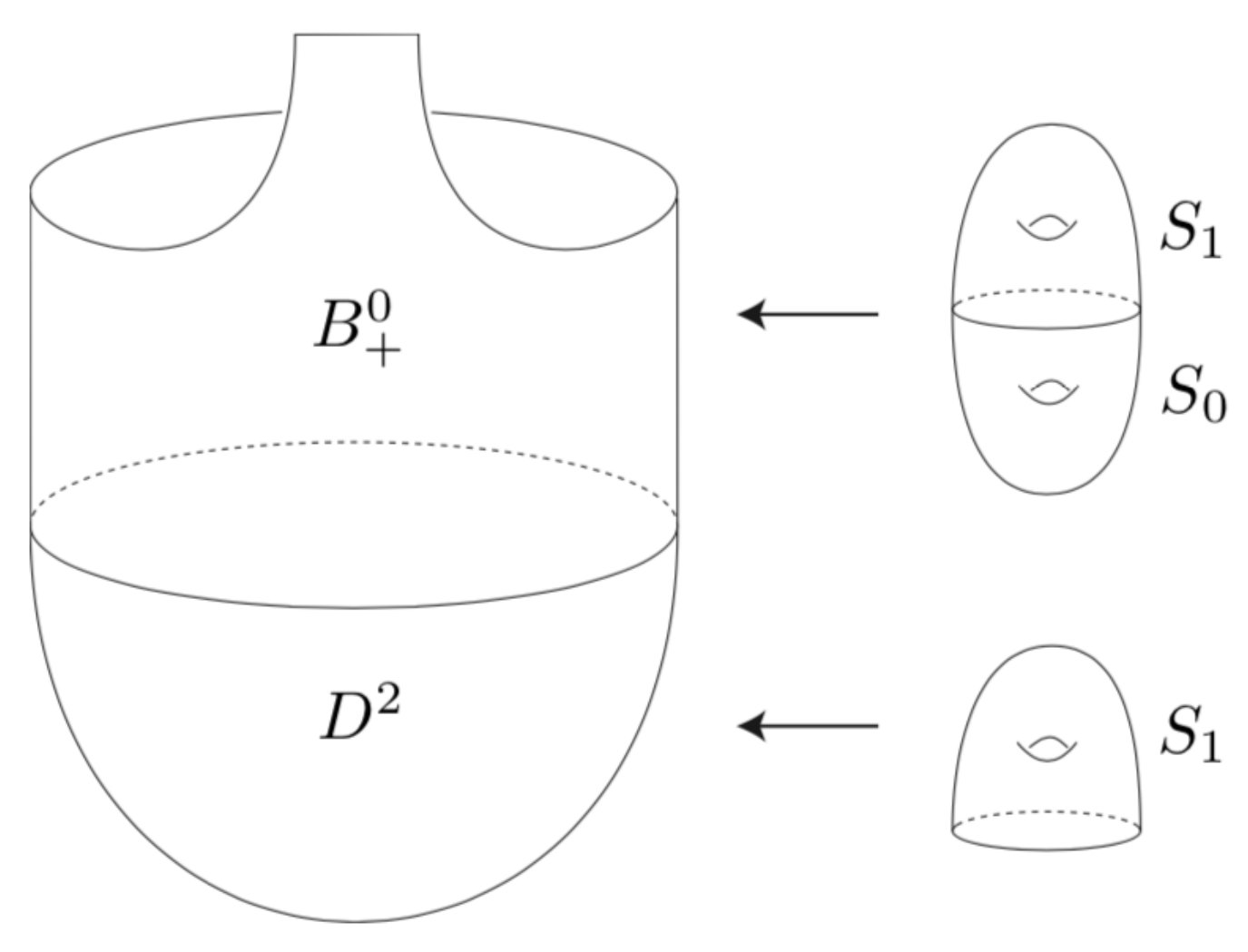}
\end{center}
\caption{Schematic diagram for $X_+^0\cup X_+^1$ which indicates the fibres over each subsurface.}
\label{fig: Xplus}
\end{figure}
Thus $X^0_+ \cup X^1_+$ is a manifold with boundary and (concave) corners and the compact boundary component of $X^0_+ \cup X^1_+$ is homeomorphic to $M$. We smooth the corner and modify the symplectic form near the compact boundary component so that, in a collar, it looks like a piece of the symplectisation of $(M, \lambda)$. Finally we obtain $(X_+, \Omega_{X_+})$ by completing $X^0_+ \cup X^1_+$ with the negative half-symplectisation $((- \infty, 0] \times M, d(e^r \lambda))$ of $(M, \lambda)$.

We assume that the base point $z \in S_1$ is a fixed point for $\hh^+$; then we define the symplectic surface
$$S_z = (B^0_+ \cup D^2) \times \{ z \}.$$
Finally, we define a Lagrangian submanifold $L_{\bs{\alpha}} \subset \partial X_+$ such that
$$X_+^{>} \cap L_{\bs{\alpha}} = ([3, + \infty) \times \{ 0 \} \times \bs{\beta}) \cup ([3, + \infty) \times \{ 1 \} \times \bs{\alpha}).$$
Note that $\partial X_+ \subset X_+^0$, which is the portion of $X_+$ which has the structure of a symplectic surface bundle, and therefore has a symplectic connection.
As it was the case for $L_{\mathbf{a}} \subset \partial W_+$, we define $L_{\bs{\alpha}}$ as the trace of the parallel transport of $\{ 3 \} \times \{ 1 \} \times \bs{\alpha}$ along $\partial B_+$ with respect to this symplectic connection.
\subsubsection{Definition of $\Phi^+$}
We fix a generic almost complex structure $J^+$ on $X_+$ which is compatible with $\Omega_{X_+}$ and the stable Hamiltonian structures at the ends, and makes $S_z$ holomorphic. Given a $2g$-tuple of intersection points $\mathbf{y} \in {\mathcal S}_{\bs{\alpha}, \bs{\beta}}$ and an orbit set $\bs{\gamma} \in {\mathcal O}$ (i.e.\ contained in $M$) we denote by ${\mathcal M}_{X_+}(\mathbf{y}, \bs{\gamma})$ the moduli space of $J^+$-holomorphic curves in $X_+$ with boundary on $L_{\bs{\alpha}}$ which are positively asymptotic to $[0,1] \times \mathbf{y}$ and negative asymptotic to $\bs{\gamma}$. To a curve $u \in {\mathcal M}_{X_+}(\mathbf{y}, \bs{\gamma})$ we associate an ECH-type index $I(u)$ and the intersection number ${\mathcal F}(u)$ between the image of $u$ and the surface $S_z$. By positivity of intersection, ${\mathcal F}(u) \ge 0$.

We define the map $\Phi^+ \colon CF^+(\Sigma, \bs{\alpha}, \bs{\beta}, z) \to ECC(M, \lambda)$ as
\begin{equation}
\Phi^+([\mathbf{y}, i]) = \sum \limits_{\bs{\gamma} \in {\mathcal O}} \sum \limits_{j=0}^i \# {\mathcal M}_{X_+}^{I=0, {\mathcal F}=j}(\mathbf{y}, \bs{\gamma}) \bs{\gamma}.
\end{equation}
\begin{thm}
$\Phi^+$ is a well defined chain map.
\end{thm}
This follows from the usual  analysis of moduli spaces. However, the symplectic form $\Omega_{X_+}$ is not exact, and therefore some extra care is needed. For example,  fixing the intersection number ${\mathcal F}$ is an important ingredient for proving compactness of the moduli spaces. Moreover, $J^+$-holomorphic curves with no positive ends can exist, and we have to exclude them from the count. This is done by showing that they have large ECH-type index; see \cite[Lemma 5.6.2 and Lemma 5.6.3]{III}. The assumption on the genus of $S$ is used there.

We need to show that the maps $\widehat{\Phi}_*$ and $\Phi^+_*$ fit into the commutative diagram \eqref{commutativity of isomorphisms}. In \cite[Section 6]{III} this is done by comparing the mapping cones of $U$ in Heegaard Floer homology and embedded contact homology. For simplicity we prefer to ignore algebraic complications and illustrate only the geometric ideas of  \cite[Section 6]{III}. Thus, instead to work with cones, we will sketch only the proofs of commutativity of the two squares displayed in \eqref{commutativity of isomorphisms}. The next theorem shows commutativity of the right-hand square.
\begin{thm}[{\cite[Theorem 5.9.1]{III}}] The map $\Phi^+$ commutes up to homotopy with the maps $U$ on Heegaard Floer homology and embedded contact homology.
\end{thm}
\begin{proof}
The maps $U$ are defined by counting $I=2$ holomorphic curves passing through a generic point: in embedded contact homology by definition and in Heegaard Floer homology by Theorem \ref{thm: U-map}. We define a chain homotopy between $\Phi^+ \circ U$ and $U \circ \Phi^+$ by moving a base point from the positive end of $X_+$ to the negative end, and counting $I=1$ holomorphic curves in $X_+$ passing through the moving point at some isolated instant.
\end{proof}
Now we prove commutativity of the left-hand square.
\begin{lemma}[{\cite[Lemmas 5.8.1 and 5.8.2]{III}}]\label{paperino}
If $\mathbf{y} \in {\mathcal S}_{\mathbf{a}, \hh(\mathbf{a})}$ and $u \in {\mathcal M}_{X_+}^{{\mathcal F}=0}(\mathbf{y}, \bs{\gamma})$, then $\bs{\gamma} \in \widehat{\mathcal O}_{2g}$ (i.e.\ it is built from closed Reeb orbits in $N$) and the image of $u$ is contained in $W_+$.
\end{lemma}
\begin{proof}
Since ${\mathcal F}(u)=0$, the image of $u$ is contained in $X_+ \setminus S_z$, where $\Omega_{X_+}$ is exact. Then an energy estimate shows that $\bs{\gamma}$ is built from closed Reeb orbits in $N$ plus $e'$ and $h'$ because orbits in $\op{int}(V)$ and in the no man's land are long. Then, using positivity of intersection with the $J^+$-holomorphic surfaces $C_\theta = B_+ \times \{ \theta \} \subset W_+$, where $\theta \in \partial S$, we obtain that $\bs{\gamma}$ is built from closed orbits in $N$ and the image of $u$ is contained in $W_+$.
\end{proof}
If $\Omega_{X_+}|_{W_+}$ is sufficiently close to $\Omega_+$, not only we can identify $ECC_{2g}(N, \alpha)$ and $PFC_{2g}(N, \alpha_0, \omega)$, but we can also choose $J^+$ on $X_+$ and $J_+$ on $W_+$ so that there is a bijection between $J^+$-holomorphic curves contained in $W_+$ and $J_+$-holomorphic curves. Thus Lemma \ref{paperino} implies that there is a commutative diagram
$$\xymatrix{
\overline{CF}(S, \mathbf{a}, \hh(\mathbf{a})) \ar[r] \ar[d]^{\overline{\Phi}} & CF^+(\Sigma, \bs{\alpha}, \bs{\beta}, z) \ar[d]^{\Phi^+} \\
ECC_{2g}(N) \ar[r] & ECC(M)
}$$
where the top arrow is the inclusion $\mathbf{y} \mapsto [\mathbf{y}, 0]$. In homology, both this map and $\overline{\Phi}$ factor through $\widehat{HF}(S, \mathbf{a}, \hh(\mathbf{a}))$, and this proves the commutativity of the left-hand square.

The next theorem is proved by an algebraic trick which generalises Ozsv\'ath and Szab\'o's observation that $\widehat{HF}(M)$ is trivial if and only if $HF^+(M)$ is trivial\footnote{$\widehat{HF}(M)$ and $HF^+(M)$ are never trivial, but subgroups associated to some Spin$^c$ structures can be trivial.} in \cite[Proposition 2.1]{OSz2}.
\begin{prop}[See {\cite[Theorem 6.1.5]{III}}] $\Phi^+_*$ is an isomorphism if and only if $\widehat{\Phi}_*$ is an isomorphism.
\end{prop}
\section{The closed-open map $\Psi$}\label{sec: CO}
In this section we will define a closed-open map
$$\Psi \colon PFC_{2g}(N, \omega, \alpha_0) \to \widehat{CF}(S, \mathbf{a}, \hh(\mathbf{a})).$$
First, we observe that the naive map that we would obtain by turning the cobordism $W_+$ ``upside-down'' does not work: if we stack $W_+$ on top of the upside-down cobordism $W_-$, the map we obtain is not homotopic to the identity because holomorphic curves from an Heegaard Floer generator to itself have negative index. A similar phenomenon happens in \cite{ganatra}, where the composition of the open-closed map from Hochschild homology to symplectic cohomology with the closed-open map from symplectic cohomology to Hochschild cohomology is a ``Poincar\'e duality'' rather than the identity. In our case, to amend the naive map, we compactify the cobordism $W_-$ and count holomorphic curves in the compactified cobordism. However the price of this operation is to work with holomorphic curves with boundary conditions on a {\em singular} Lagrangian {\em with boundary}, which introduces an incredible amount of complication.

\subsection{Definition of $\Psi$}
Let $D$ be a disc. On $D$ we consider polar coordinates $(\rho, \phi)$ with $\rho \in [0,1]$ and write $z_\infty$ for the origin.
Let $\overline{S}$ be the capped-off surface obtained by gluing $D$ to $S$ along their boundary. For every $m \gg 0$ (i.e.\ large enough so that all conditions we impose on it make sense) we define a diffeomorphism $\overline{\hh}_m \colon \overline{S} \to \overline{S}$
such that
$$\overline{\hh}_m|_{S}= \hh \quad \text{and} \quad \overline{\hh}_m|_{D}(\rho, \phi)=(\rho, \phi+ \nu_m(\rho))$$
where $\nu_m(\rho)= \frac{2 \pi}{m}$ for $\rho \le \frac 12$, $\nu_m(1)=0$ and $\nu_m(\rho) \ge 0$ for $\rho \in [0,1]$.  We also assume that $\overline{\hh}_m$ converges to a diffeomorphism $\overline{\hh}_{\infty}$ for  $m \to \infty$. For the complete list of properties of $\nu_m$ see \cite[Subsection 5.1.2]{I} and for a graphical description of the action of $\overline{\hh}_m$ on $D$ see Figure \ref{fig: around_z_infty}. Most of the time we will write $\overline{\hh}$ for $\overline{\hh}_m$ and assume $m \gg 0$. Similarly, every object depending on $\overline{\hh}_m$ may or may not have an index $m$ depending on the context.

We extend the arcs $a_i$ to arcs $\overline{a}_i$  by straight lines in $D$, so that they all meet at $z_\infty$. We assume, among other things, that the angles at $z_\infty$ between two consecutive arcs is an integer multiple of $\frac{2 \pi}{m}$  larger than $\frac{4 \pi g}{m}$ and that, as $m \to \infty$, the portions of the arcs in $D_{1/2}= \{ \rho \le 1/2 \}$ converge to the segment $\{ \phi = 0 \}$. This last property will be used in the proof of Theorem \ref{thm: difficile}. See \cite[Subsection 5.2.2]{I} for the complete description of the arcs $\overline{a}_i$.

\begin{figure}\centering
\includegraphics[width=4cm]{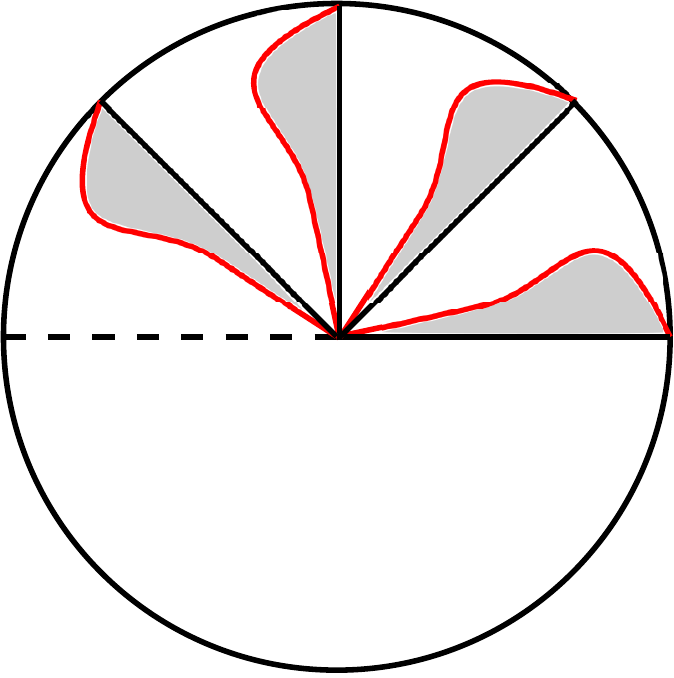}
\caption{The extended arcs and the action of $\overline{\hh}_m$ on $D$. The black arcs are portions of the $\overline{a}_i$ and the red ones are portions of the $\overline{\hh}(\overline{a}_i)$. The gray regions are the bigons covered by the {\em thin strips} of Definition \ref{dfn: thin strip}.}
\label{fig: around_z_infty}
\end{figure}

Let $B_-$ be the unit disc with one puncture in the interior and one puncture on the boundary, which we identify biholomorphically with the subset of the cylinder $\R \times \R/2 \Z$ obtained by rounding the corners of  $ (\R \times \R/2 \Z) \setminus ((- \infty, -2) \times (1,2))$. We define a fibration
$$\pi_{\overline{W}_-} \colon \overline{W}_- \to B_-$$
with fibre $\overline{S}$ and monodromy $\overline{\hh}$, and equip the total space with the symplectic form $\overline{\Omega}_-= ds \wedge dt + \overline{\omega}$, where $(s,t)$ are coordinates on $B_- \subset \R \times \R / 2 \Z$ and $\overline{\omega}$ is an area form on $\overline{S}$ which is preserved by $\overline{\hh}$ and extends $\omega$.
We can view $(\overline{W}_-, \overline{\Omega}_-)$ as a symplectic cobordism (with boundary) between $[0,1] \times \overline{S}$ and $\overline{N}$, where the latter is the mapping torus of $(\overline{S}, \overline{\hh})$. We refer to \cite[Section~5]{I} for more details about this construction. Note that $\overline{N}$ carries a stable Hamiltonian structure $(\overline{\alpha}_0, \overline{\omega})$ which is induced by the fibration $\overline{\pi} \colon \overline{N} \to S^1$. Its only closed Reeb orbit in $\overline{N} \setminus N$ intersecting a fibre at most $2g$ times is the orbit $\delta_0$ corresponding to $z_\infty$.

In $\overline{W}_-$ there is the subset $W_-$ which is the total space of the fibration over $B_-$ with fibre $S$ and monodromy $\hh$. Since $\overline{\hh}|_D$ is isotopic to the identity, there is an identification $\overline{W}_- \setminus \op{int}(W_-) \cong B_- \times D$ which induces a map\footnote{This map is called $\overline{\pi}_{D^2}$ in \cite{I}.}
$$\pi_D \colon \overline{W}_- \setminus \op{int}(W_-) \to D.$$
We will denote similar maps for $\R \times [0,1] \times \overline{S}$ and $\R \times \overline{N}$ by the same symbol.

We define an immersed Lagrangian submanifold with boundary $L_{\overline{\mathbf{a}}}$ by parallel transport of $\overline{\mathbf{a}}= (\overline{a}_1, \ldots, \overline{a}_{2g})$ placed at the fibre over $(-3,1)$. Then,
\begin{align*}
L_{\overline{\mathbf{a}}} \cap (- \infty, -3] \times \{ 1 \} \times \overline{S} & =  (-\infty, -3] \times \{ 1 \} \times \overline{\mathbf{a}} \\
L_{\overline{\mathbf{a}}} \cap (-\infty, -3] \times \{ 0 \} \times \overline{S} & =  (- \infty, -3] \times \{ 0 \} \times \overline{\hh}(\overline{\mathbf{a}}).
\end{align*}
The manifold $L_{\overline{\mathbf{a}}}$ is embedded in the interior and has a clean self-intersection at the boundary.

We fix an almost complex structure $J_-$ on $\overline{W}_-$  which is generic among those
whose restriction to $\overline{W}_- \setminus \op{int}(W_-)$ is the pull back of a split almost complex structure on $B_- \times D$ so that, in particular, the map $\pi_D$ is holomorphic.

We call $\sigma_\infty^-= B_- \times \{ z_\infty \} \subset \overline{W}_-$ the {\em section at infinity}. By our choice of almost complex structure, $\sigma_\infty^-$ is holomorphic, regular (i.e.\ its linearised Cauchy-Riemann operator is surjective) and $\op{ind}(\sigma_\infty^-)=0$; see \cite[Lemma 5.8.9]{I}.

We denote by $\overline{\mathcal O}_{2g}$ the set of orbit sets built from closed Reeb orbits in $\overline{N}$ with total intersection number $2g$ with a fibre and by $\overline{\mathcal S}_{\mathbf{a}, \hh(\mathbf{a})}$ the set of $2g$-tuples of intersection points between $\overline{\mathbf{a}}$ and $\overline{\hh}(\overline{\mathbf{a}})$.  A subtlety in the definition of $\overline{\mathcal S}_{\mathbf{a}, \hh(\mathbf{a})}$ is that $z_\infty$ can be repeated, unlike all other intersection points, because it can be seen as an intersection point between different pairs of arcs. We denote by ${\mathcal M}_{\overline{W}_-}(\bs{\gamma}, \mathbf{y})$ the moduli space of $J_-$-holomorphic curves in $\overline{W}_-$ with boundary on $L_{\overline{\mathbf{a}}}$ which are positively asymptotic to $\bs{\gamma} \in \overline{\mathcal O}_{2g}$ and negatively asymptotic to $\mathbf{y} \in \overline{\mathcal S}_{\mathbf{a}, \hh(\mathbf{a})}$. We also fix the basepoint $\mathfrak{m}=\{ (0, \frac 32 ) \} \times \{ z_\infty \} \in \overline{W}_-$ and denote the set of $J_-$-holomorphic curves in ${\mathcal M}_{\overline{W}_-}(\bs{\gamma}, \mathbf{y})$ which pass through $\mathfrak{m}$ by ${\mathcal M}_{\overline{W}_-}(\bs{\gamma}, \mathbf{y}; \mathfrak{m})$.

Let $z_\infty^\dagger \in \overline{S}$ be a point near $z_\infty$ in the complement of the arcs $\overline{\mathbf{a}}$ and $\overline{\hh}(\overline{\mathbf{a}})$. We denote by $(\sigma_\infty^-)^\dagger$ the $J_-$-holomorphic multisection which intersects each fibre in the orbit of $z_\infty^\dagger$ under the monodromy $\overline{\hh}$. We recall that $\overline{\hh}$ has order $m \gg 0$ near $z_\infty$ and therefore $(\sigma_\infty^-)^\dagger$ intersects each fibre $m$ times. Given a $J_-$-holomorphic map $u \in {\mathcal M}_{\overline{W}_-}(\bs{\gamma}, \mathbf{y})$, we define $n_*(u)$ as the algebraic intersection number between $(\sigma_\infty^-)^\dagger$ and the image of $u$. The orbit of $z_\infty^\dagger$ gives rise to holomorphic multisections and to intersection numbers, still denoted by $n_*$, also in $\R \times \overline{N}$ and $\R \times [0,1] \times \overline{S}$. Since both $(\sigma_\infty^-)^\dagger$ and $u$ are $J_-$-holomorphic, $n_*(u) \ge 0$ by positivity of intersection. To $u$ we associate also an ECH-type index $I(u)$; see \cite[Definition 5.6.6 and Subsection 5.7.7]{I}. If $u$ has no ends at $\delta_0$ or at the chord over $z_\infty$, then $I(u) \ge 0$ for a generic $J_-$ by the index inequality \cite[Theorem 5.6.9]{I}. Moreover, $I(\sigma_\infty^-)=0$.

We define a map $\overline{\Psi} \colon ECC_{2g}(N, \alpha_0, \omega) \to \overline{CF}(S, \mathbf{a}, \hh(\mathbf{a}))$ by
$$\overline{\Psi}(\bs{\gamma})= \sum \limits_{\mathbf{y} \in {\mathcal S}_{\mathbf{a}, \hh(\mathbf{a})}} \#{\mathcal M}_{\overline{W}_-}^{I=2, n_*=m}(\bs{\gamma}, \mathbf{y}; \mathfrak{m}) \mathbf{y}$$
(where $\bs{\gamma} \in \widehat{\mathcal O}_{2g}$) and, finally, the map $\Psi \colon ECC_{2g}(N, \alpha_0, \omega) \to \widehat{CF}(S, \mathbf{a}, \hh(\mathbf{a}))$ by composing $\overline{\Psi}$ with the projection $\overline{CF}(S, \mathbf{a}, \hh(\mathbf{a})) \to \widehat{CF}(S, \mathbf{a}, \hh(\mathbf{a}))$.

By a standard argument involving the index inequality we can prove that the moduli spaces ${\mathcal M}_{\overline{W}_-}^{I=2, n_*=m}(\bs{\gamma}, \mathbf{y}; \mathfrak{m})$ are zero-dimensional manifolds. However, showing that they are compact, and even more showing that $\Psi$ is a chain map, occupy a large portion of \cite{I} because of the unusual structure of the Lagrangian boundary condition $L_{\mathbf{a}}$. In the next subsection we will give an outline of the argument.

 \subsection{$\Psi$ is a chain map}\label{Psi si a chain map}
To prove that $\Psi$ is a chain map we analyse the boundary of the compactification of the one-dimensional moduli spaces ${\mathcal M}_{\overline{W}_-}^{I=3, n_*=m}(\bs{\gamma}, \mathbf{y}; \mathfrak{m})$. The first step is to verify that SFT compactness holds.
\begin{thm}[See {\cite[Section 7.3]{I}}] \label{SFT compactness for Psi}
Let $u_n$ be a sequence of holomorphic curves in ${\mathcal M}_{\overline{W}_-}^{I=i, n_*=m}(\bs{\gamma}, \mathbf{y}; \mathfrak{m})$. Then, up to taking a subsequence, $u_n$ converges in the SFT sense to a building $u_\infty = (u_\infty^{-h}, \ldots, u_\infty^0, \ldots ,u_\infty^l)$, where $u_\infty^k$ is one of the following:
\begin{itemize}
\item a holomorphic curve in $\overline{W}_-$ with boundary on $L_{\overline{\mathbf{a}}}$ and passing through $\mathfrak{m}$ for $k=0$,
\item a holomorphic curve in $\R \times \overline{N}$ for $k>0$, or
\item a holomorphic curve in $\R \times [0,1] \times \overline{S}$ with boundary on $(\R \times \{ 1 \} \times \overline{\mathbf{a}}) \cup \R \times (\{ 0 \} \times \overline{\hh}(\overline{\mathbf{a}}))$ for $k<0$.
\end{itemize}
Moreover $u_\infty^l$ has positive ends at $\bs{\gamma}$, $u_\infty^{-h}$ has negative ends at $\mathbf{y}$ and the positive ends of $u_{\infty}^k$ match with the negative ends of $u_{\infty}^{k+1}$ for $k=-h, \ldots, l-1$. Finally, $I(u_\infty^{-h}) + \ldots + I(u_{\infty}^l)=i$ and $n_*(u_\infty^{-h}) + \ldots + n_*(u_{\infty}^l)=m$.
\end{thm}
Despite the unusual situation, the proof of Theorem \ref{SFT compactness for Psi} follows the proof of the usual SFT compactness theorem in \cite{BEHWZ} very closely. See \cite[Section 7.3]{I} for the details. Even if $\bs{\gamma} \in \widehat{\mathcal O}_{2g}$ and $\mathbf{y} \in {\mathcal S}_{\mathbf{a}, \hh(\mathbf{a})}$, the holomorphic curves $u_\infty^k$ can have one of the following features, which we will call ``bad degenerations'':
\begin{enumerate}
\item[(i)] ends at $\delta_0$ (for $k \ge 0$) or at the chord over $z_\infty$ (for $k \le 0$),
\item[(ii)] boundary points at the boundary\footnote{This type of bad degeneration is called ``boundary points at $z_\infty$'' in \cite{I}.} of $L_{\overline{\mathbf{a}}}$, or
\item[(iii)] closed irreducible components.
\end{enumerate}
Any bad degeneration appearing in the compactification of ${\mathcal M}_{\overline{W}_-}^{I=3, n_*=m}(\bs{\gamma}, \mathbf{y}; \mathfrak{m})$ gives rise to a boundary component which contributes neither to $\partial \circ \overline{\Psi}$ nor $\overline{\Psi} \circ \partial$. In fact, it will turn out that $\overline{\Psi}$ is not a chain map, but we will be able to exclude enough bad degenerations and show that the contribution of the remaining ones is in the kernel of the projection from $\overline{CF}(S, \mathbf{a}, \hh(\mathbf{a}))$ to $\widehat {CF}(S, \mathbf{a}, \hh(\mathbf{a}))$.

In \cite[Section 7.5]{I} we use topological tools --- the intersection numbers $n_*$ and the ECH-type index --- to prove that bad degenerations can follow only a limited number of patterns, which are described in \cite[Theorem 7.6.1]{I}, and then we exclude all patterns but one by studying, via a rescaling argument inspired by \cite{IP1}, how the sequence $u_n$ approaches the limit. For sake of simplicity we wil describe only the main ideas of the proof.

\begin{dfn}\label{dfn: thin strip}
A {\em thin strip} is a holomorphic curve $u$ in $\R \times [0,1] \times D$ such that $\pi_D \circ u$ covers the small bigon in $D$ between an arc in $\overline{a}_i$ and its image under $\overline{\hh}$. See Figure \ref{fig: around_z_infty}.
\end{dfn}
A thin strip is positively asymptotic to the chord over $z_\infty$ and negatively asymptotic to a chord over one among the intersection points $x_i$ or $x_i'$. Moreover,  if $u$ is a thin strip, then it is regular and satisfies $I(u)=1$ and $n_*(u)=1$. The next proposition summarises some of the results of \cite[Section 7.5]{I}.
\begin{prop}\label{bad degenerations}
Let $u_n$ be a sequence of $J_-$-holomorphic curves in ${\mathcal M}_{\overline{W}_-}^{n_*=m}(\bs{\gamma}, \mathbf{y}; \mathfrak{m})$ converging to $u_\infty = (u_\infty^{-h}, \ldots, u_\infty^0, \ldots, u_\infty^l)$. Then the following hold:
\begin{enumerate}
\item there is at most one nontrivial negative end  (i.e.\ which does not belong to a trivial cylinder) at a cover of $\delta_0$ in all $u_\infty^k$ with $k>0$; \label{oe}
\item if one of the holomorphic curves $u_\infty^k$ has a negative end at a cover of $\delta_0$, then $u_\infty^0$ has an irreducible component which is a cover of the section at infinity $\sigma_\infty^-$; \label{sai}
\item if some $u_\infty^k$ with $k < 0$ has a positive end at the chord over $z_\infty$, then some $u_\infty^k$ with $k>0$ has a  negative end at a cover of $\delta_0$; \label{oid}
\item any nontrivial irreducible component of $u_\infty^k$ (with $k<0$) with a positive end at the chord over $z_\infty$ is a thin strip; \label{ts}
\item no $u_\infty^k$ has a boundary point at the boundary of $L_{\overline{\mathbf{a}}}$; and \label{nbp}
\item no $u_\infty^k$ has a closed irreducible component. \label{ncc}
\end{enumerate}
\end{prop}
\begin{proof}[Sketch of proof]
The main tool is the analysis of the contributions of ends at covers of $\delta_0$, ends at the chord over $z_\infty$ and boundary points at the boundary of $L_{\overline{\mathbf{a}}}$ to $n_*(u_\infty)$. By \cite[Lemma 7.4.1]{I}, a nontrivial positive end at a $p$-fold cover of $\delta_0$ contributes at least $p$ to $n_*(u_\infty)$, while a nontrivial negative end contributes at least $m-p$. By \cite[Lemma 7.4.2]{I} a nontrivial positive end at the chord over $z_\infty$ contributes $1$ to $n_*(u_\infty)$ if it belongs to a thin strip, and at least $2g$ otherwise, while a nontrivial negative end contributes at least $2g$. Finally, by \cite[Lemma 7.4.4]{I} a boundary point at the boundary of $L_{\overline{\mathbf{a}}}$ contributes at least $2g$. All these claims are easy consequences of positivity of intersection between holomorphic curves. As an example, we will compute the contribution of a boundary point at the boundary of $L_{\overline{\mathbf{a}}}$ in $u_\infty^0$. The composition $\pi_D \circ u_\infty^0$ is defined and holomorphic in a neighbourhood of such a point. Since nonconstant holomorphic functions are open, its image covers a region between two arcs in $\overline{\mathbf{a}}$ in a neighbourhood of $z_\infty$. Thus, if we take $z_\infty^\dagger$ close enough to $z_\infty$, the section $(\sigma_\infty^-)^\dagger$ will intersect the image of $u_\infty^0$ at least $2g$ times near the boundary point at the boundary of $L_{\overline{\mathbf{a}}}$.

Suppose that there are two nontrivial negative ends at $\delta_0$ with multiplicity $p'$ and $p''$ respectively. Then they contribute $2m-p'-p''$ to $n_*(u_\infty)$. Since $p'+p'' \le 2g \ll m$ and $n_*(u_\infty)=m$, we have a contradiction. This proves \eqref{oe}.

The map $\pi_D \circ u_\infty^0$ is holomorphic and nonconstant, and therefore open, in a neighbourhood of the point mapped to $\mathfrak{m}$. Thus, by unique continuation of holomorphic curves, a neighbourhood of that point contributes $m$ to $n_*(u_\infty)$, unless the irreducible component of $u_\infty^0$ passing through $\mathfrak{m}$ is a cover of the section at infinity $\sigma_\infty^-$. This proves \eqref{sai}.

A nontrivial negative end at a chord over $z_\infty$ contributes at least $2g$ to $n_*(u_\infty)$. Thus, by the argument above, if there is a positive end at a chord over $z_\infty$, an irreducible component of $u_\infty^0$ is a cover of $\sigma_\infty^0$, and therefore, for some $k>0$, the holomorphic curve $u_\infty^k$ has a nontrivial negative end at a cover of $\delta_0$. This proves  \eqref{oid}.

A nontrivial positive end at a chord over $z_\infty$ contributes at least $2g$ to $n_*(u_\infty)$ unless it belongs to a thin strip. By \eqref{oid}, there is a nontrivial negative end at a cover of $\delta_0$ which contributes $m-p$ with $p<2g$. This is a contradiction, and thus \eqref{ts} is proved.

Either there is a nontrivial negative end at a cover of $\delta_0$, which contributes $m-p$ to $n_*(u_\infty)$, or the irreducible component of $u_\infty^0$ passing through $\mathfrak{m}$ is not a cover of $\sigma_\infty^-$ and therefore contributes $m$ to  $n_*(u_\infty)$. A boundary point at the boundary of $L_{\overline{\mathbf{a}}}$ contributes at least $2g$, which is a contradiction in both cases because $p<2g$. This proves \eqref{nbp}.
\item A closed irreducible component of $u_\infty$ is homologous to a branched cover of a fibre $\overline{S}$. Its Fredholm index, computed by the Riemann-Roch formula, is negative, and therefore such a component cannot exist for a generic almost complex structure $J_-$.\footnote{In \cite{I} we use an ECH index computation to show that a building $u_\infty$ with a closed component has $I>3$.} This proves \eqref{ncc}.
\end{proof}
Now suppose $\sum \limits_{i=-h}^l I(u_\infty^k)=3$. By \cite[Lemma 7.5.5]{I} $I(u_\infty^0) \ge 0$ and $I(u_\infty^k) \ge 1$ for $i \ne 0$, which implies that we can have buildings with at most four levels in the boundary of the compactification of  ${\mathcal M}_{\overline{W}_-}^{I=3, n_*=m}(\bs{\gamma}, \mathbf{y}; \mathfrak{m})$.
Moreover $I(u_\infty^0)=2$ if no irreducible component of $u_\infty^0$ is a cover of $\sigma_\infty^-$ because passing through $\mathfrak{m}$ is a codimension-two constraint. Therefore, if $u_\infty$ is a two-level building, then $I(u_\infty^0)=2$ and $u_\infty$ contributes to either  $\overline{\Psi} \circ \partial$ or $\partial \circ \overline{\Psi}$. On the other hand, bad degenerations necessarily appear in buildings with three or four levels and have an irreducible component which is a cover of the section at infinity $\sigma_\infty^-$.

A holomorphic building $u_\infty$ in the compactification of ${\mathcal M}_{W_-}^{I=3, n_*=m}(\bs{\gamma}, \mathbf{y}; \mathfrak{m})$ with a bad degeneration is called of type (A) if $I(u_\infty^1)=2$ and of type (B)\footnote{Buildings of type (A) correspond to buildings of type (1) and  buildings of type (B) correspond to buildings of type (2) to (6) in \cite[Theorem 7.6.1]{I}.} if $I(u_\infty^1)=1$. If $u_\infty$ is of type (A), then $u_\infty = (u_\infty^{-1}, u_\infty^0, u_\infty^1)$, where $u_\infty^0$ contains $\sigma_\infty^-$ as an irreducible component and $u_\infty^{-1}$ is a thin strip. Buildings of type (B) can be ignored as a consequence of the next theorem. We will use a limiting argument for $m \to \infty$, and therefore we will have manifolds $\overline{W}_{-,m}$ constructed from $\overline{\hh}_m$, almost complex structures $J_{-,m}$ and so on. However, all manifolds $\overline{W}_{-,m}$ are diffeomorphic, so we fix a model $\overline{W}_-$ and regard all $J_{-,m}$ as almost complex structures on it.
\begin{thm}[{\cite[Theorem 7.10.1]{I}}]\label{thm: difficile}
If $m$  is large enough, buildings of type (B) do not appear in the boundary of the compactification of ${\mathcal M}_{(\overline{W}_-, J_{-,m})}^{I=3, n_*=m}(\bs{\gamma}, \mathbf{y}; \mathfrak{m})$ for any $\bs{\gamma} \in \widehat{\mathcal O}_{2g}$ and $\mathbf{y} \in {\mathcal S}_{\mathbf{a}, \hh(\mathbf{a})}$.
\end{thm}
The proof of this theorem occupies the subsections from 7.7 to 7.10 of \cite{I}; here we will sketch only the main ideas. Assume that there is a sequence $m_k \to \infty$ such that, for any $k$, there is a sequence of $J_{-,m_k}$-holomorphic curves $u_{k,j}$ in ${\mathcal M}_{(\overline{W}_-, J_{-,m_k})}^{I=3, n_*=m_k}(\bs{\gamma}, \mathbf{y}; \mathfrak{m})$ converging to a building $u_{k, \infty}$ with $I(u_{k, \infty}^1)=1$. We know from Proposition \ref{bad degenerations}  that there is only one nontrivial negative end in $u_{k, \infty}$ at a cover of $\delta_0$ for every $k$ and, moreover, it must belong to an irreducible component of index one.
For simplicity, we assume that the end at $\delta_0$ is simple; the case of a nontrivial cover is slightly more complicated, but does not need any new idea. To simplify the situation further, we assume that every $u_{k, \infty}$ contains a level $u_{k, \infty}^1$ with $I(u_{k, \infty}^1)=1$ from an orbit set $\bs{\gamma}_+$ to an orbit set $\delta_0 \bs{\gamma}_-$, a level $u_{k, \infty}^0$ which contains $\sigma_{\infty}^-$ as an irreducible component and has other components contained in $W_-$, and a level $u_{k, \infty}^{-1}$ which is a thin strip from the chord over $z_\infty$ to a chord over one of the intersection points $x_i$ or $x_i'$. There may be one more level, but it will not be relevant to our argument.

The end of $u_{k, \infty}^1$ at $\delta_0$ (which, we recall, is parametrised by $(s,t) \in (-\infty, 0] \times \R / 2 \Z$) satisfies the Fourier expansion
\begin{equation}\label{asymptotic expansion}
\pi_D \circ u_{k, \infty}^1(s,t) = e^{\pi (1-\frac{1}{m_k})s}f_k(t) + o( e^{\pi (1-\frac{1}{m_k})s})
\end{equation}
where $f_k(t)=c_k e^{\pi i t}$, with $c_k \in \C$, is the {\em asymptotic eigenfunction} of the end. See \cite[Lemma 7.7.3]{I}. For a generic choice of almost complex structure, and up to taking a subsequence, the coefficients $c_k$ converge to $c \ne 0$ by Lemma 7.7.6 and Lemma 7.7.9 in \cite{I}. If we translate $s$, we multiply $c_k$ by a positive real number, and therefore we can assume that $|c|=1$. Since the moduli spaces of holomorphic curves in $\R \times \overline{N}$ with $I=1$ and fixed asymptotics are finite up to translations, the set ${\mathcal C}$ of possible limit values $c$ is finite.
\begin{dfn}[See {\cite[Definition 7.7.10]{I}}\footnote{\cite[Definition 7.7.10]{I} differs from the definition we give here in the notation and for the fact that, in \cite{I}, we consider also the case of ends to multiple covers of $\delta_0$.}]
A {\em bad radial ray} in $\C$ is a half-line $-ic \R_+$ for $c \in {\mathcal C}$.
\end{dfn}
We observe that $-i = e^{\pi \frac 32 t}$ and, as we will see later, it is no coincidence that the projection of $\mathfrak{m}$ to $B_-$ is  $(0, \frac 32)$. As explained in \cite[Remark 7.7.11]{I} we can assume that $-1 \R_+$ is not a bad radial ray.

After extracting a diagonal subsequence $u_{k, j(k)}$, we find sequences $R^\pm_k \to + \infty$ for $k \to + \infty$ and holomorphic maps $\tilde{u}_k \colon B_- \cap ([-R_k^-, R_k^+] \times \R / 2\Z) \to \overline{W}_-$  (called {\em truncations} in \cite[Subsection 7.8.1]{I}) which parametrise the portions of $u_{k, j(k)}$ contained in a neighbourhood of the section at infinity $\sigma_\infty^-$. We define holomorphic functions $\widetilde{w}_k = \pi_D \circ \widetilde{u}_k$ and, after fixing a compact neighbourhood $K$ of $(0, \frac 32)$ in $B_-$, we define constants $C_k = \| \widetilde{w}_k \|_{L^\infty(K)}$ and holomorphic functions $w_k = \dfrac{\widetilde{w}_k}{C_k}$. If the diagonal subsequence and the truncations are chosen carefully (i.e.\ as in \cite[Lemma 7.8.4]{I}), the functions $w_k$ converge to a holomorphic function $w_\infty \colon B_- \to \C$ such that:
\begin{itemize}
\item $w_\infty$ has a unique and simple zero at $(0, \frac 32)$,
\item $w_\infty(\partial B_-) \subset \R_+$,
\item $\lim \limits_{s \to + \infty} \frac{w_\infty(s,t)}{|w_\infty(s,t)|}=ce^{i \pi t}$ for some $c \in {\mathcal C}$, and
\item $\lim \limits_{s \to - \infty} w_\infty(s,t) \in \R_+$.
\end{itemize}
See \cite[Theorem 7.8.15]{I}. 

The map $(s,t) \mapsto (s, 1-t)$ defines an anti-holomorphic involution of $B_-$ fixing $(0, \frac 32)$ --- remember that $t \in \R / 2 \Z$. The function
$$f(s,t)= w_\infty(s,t)/\overline{w_\infty(s, 1-t)}$$
is thus holomorphic, bounded, and has real boundary conditions. Moreover, it satisfies $\lim \limits_{s \to - \infty}f(s,t) =1$ and therefore,  by standard one-variable complex analysis, $f(s,t)=1$ for all $(s,t) \in B_-$, so $w_\infty(s,t)= \overline{w_\infty(s,1-t)}$. This implies that $w_\infty(s, \frac 32) \in \R$ for all $s \in [-2, + \infty)= B_- \cap \left (\R \times \left \{ \frac 32  \right\} \right )$.  Since $(0, \frac 32)$ is the unique zero of $w_\infty$ and $w_\infty(\partial B_+) \subset \R_+$, we have $-ic \in -1 \R_+$. This is a contradiction because $-1\R_+$ is not a bad radial ray, and therefore we have shown that, for $k$ sufficiently large, no building of type (B) with a simple end at $\delta_0$ can be in the boundary of the moduli space ${\mathcal M}_{(W_-, J_{-, m_k})}^{I=3, n_*=m_k}(\bs{\gamma}, \mathbf{y}; \mathfrak{m})$.

If the nontrivial end of each $u_{k, \infty}$ at a cover of $\delta_0$ is not simple, the limit must be taken in the SFT sense and yields a building  $w_\infty = (w_\infty^{-a}, \ldots, w_\infty^b)$ of holomorphic functions $w_\infty^h \colon \widetilde{F}_\infty^h \to \C$ where $\widetilde{F}_\infty^h$ is a branched cover of $\R \times [0,1]$ for $h<0$, of $B_-$ for $h=0$ and of $\R \times S^1$ for $h>0$. See \cite[Subsection 7.8.7]{I}. Then, we apply the involution argument component by component and obtain a similar contradiction. This ends the proof of Theorem \ref{thm: difficile}.

From now on we will assume that $m$ is large enough so that Theorem \ref{thm: difficile} holds. In order to prove that $\Psi$ is a chain map, it remains to prove the following.
\begin{thm}[Reformulation of {\cite[Theorem 7.2.2]{I}}]\label{bad degenerations come in pairs}
If $\bs{\gamma} \in \widehat{\mathcal O}_{2g}$ is an orbit set and $\mathbf{y}, \mathbf{y}' \in {\mathcal S}_{\mathbf{a}, \hh(\mathbf{a})}$ are $2g$-tuples of intersection points which differ only by replacing one intersection point $x_i$ with $x_i'$ for some $i$, the numbers of boundary points  corresponding to buildings of type (A) in the compactifications of ${\mathcal M}^{I=3, n_*=m}(\bs{\gamma}, \mathbf{y}; \mathfrak{m})$ and   ${\mathcal M}^{I=3, n_*=m}(\bs{\gamma}, \mathbf{y}'; \mathfrak{m})$ are equal.
\end{thm}
Let $\mathbf{y}_\infty \in \overline{\mathcal S}_{\mathbf{a}, \hh(\mathbf{a})}$ be the $2g$-tuple of intersection points obtained by replacing $x_i$ or $x_i'$ in $\mathbf{y}$ or $\mathbf{y}'$ with $z_\infty$. We fix $\bs{\gamma'} \in \widehat{\mathcal O}_{2g-1}$ and, for $r \gg 0$, define
$$\mathfrak{P}= [r, + \infty)^2 \times \widehat{\mathcal M}_{\R \times [0,1] \times \overline{S}}^{I=1, n_*=1}(\mathbf{y}_\infty, \mathbf{y}) \times {\mathcal M}_{\overline{W}_-}^{I=0, n_*=0}(\delta_0 \bs{\gamma'}, \mathbf{y}_\infty) \times \widehat{\mathcal M}_{\R \times \overline{N}}^{I=2, n_*=m-1}(\bs{\gamma}, \delta_0 \bs{\gamma'}).$$
The moduli spaces in the definition of $\mathfrak{P}$ have dimension $0$, $0$ and $1$ respectively, and therefore $\mathfrak{P}$ has dimension $3$.

For the rest of the section we will identify the quotient of a moduli space by the $\R$ action with a slices for that action (i.e\ the moduli spaces $\widehat{\mathcal M}$ with a slice in the corresponding moduli spaces ${\mathcal M}$).\footnote{Slices are denoted by $\widetilde{\mathcal M}$ in \cite{I}.} Gluing gives us an embedding
$$G \colon \mathfrak{P} \to {\mathcal M}_{\overline{W}_-}^{I=3, n_*=m}(\bs{\gamma}, \mathbf{y}).$$
Strictly speaking, we should allow the glued curve to have boundary on the Lagrangian submanifolds associated to the extension of the arcs $\overline{a}_i$ to the other side of $z_\infty$; these are the {\em extended moduli spaces} of \cite[Definition 5.7.24]{I}. However, \cite[Claim 7.12.12]{I} shows that this is a technical complication which can be ignored at a first reading.

To prove Theorem \ref{bad degenerations come in pairs} we need to understand when the image of $G$ belongs to ${\mathcal M}_{\overline{W}_-}^{I=3, n_*=m}(\bs{\gamma}, \mathbf{y}; \mathfrak{m})$ and show that, essentially, it will only depend on how holomorphic curves in $\widehat{\mathcal M}_{\R \times \overline{N}}^{I=2, n_*=m-1}(\bs{\gamma}, \delta_0 \bs{\gamma'})$ approach $\delta_0$.

To a sequence $(T_{-,n}, T_{+,n}, u^{-1}, u^{0}, u^{1}) \in \mathfrak{P}$ with $T_{\pm,n} \to + \infty$ we associate a holomorphic function $w \colon B_- \to \C$ by the truncate-and-rescale technique used in the proof of Theorem \ref{thm: difficile} applied to the sequence $G(T_{-,n}, T_{+,n}, u^{-1}, u^{0}, u^{1})$. We can see this function as a sort of normal vector to $(u^{-1}, u^{0}, u^{1})$ in the compactification of ${\mathcal M}_{\overline{W}_-}^{I=3, n_*=m}(\bs{\gamma}, \mathbf{y})$ along which the sequence $G(T_{-,n}, T_{+,n}, u^{-1}, u^{0}, u^{1})$ approaches $(u^{-1}, u^{0}, u^{1})$.

The function $w$ satisfies the following properties:
\begin{itemize}
\item $w(s,t)\in \R^+ \cdot e^{i\eta(t)}$ for all $(s,t) \in \partial B_-$, where $\eta \colon [0,1] \to \R$ is a decreasing function with $\eta(0)= \frac{\pi}{m}$ and $\eta(1)=0$;
\item $\lim \limits_{s \to -\infty} \left| w(s,t)-c_1 e^{-\frac{\pi}{m}(s + i t - i)} \right| < \infty$ for some $c_1\in \R^+$; and
\item $\lim \limits_{s \to +\infty} \left| w(s,t)-c_2 e^{\pi(s+it)} \right | < \infty$ for some $c_2\in \C^\times$.
\end{itemize}
Functions of this sort form a non-empty open cone ${\mathcal N}$ inside a three-dimensional real vector space, are determined by the constants $c_1$ and $c_2$ and have a unique zero.
 See \cite[Subsection 7.12.1]{I} for the proof of all these properties (and more).
Moreover, the zero of a holomorphic function $w$ associated to a sequence $(T_{-,n}, T_{+,n}, u^{-1}, u^{0}, u^{1})$ is the limit of the projections to $B_-$ of the intersection points between the image of $G(T_{-,n}, T_{+,n}, u^{-1}, u^{0}, u^{1})$ and the section at infinity $\sigma_\infty^-$.

The careful reader has surely observed that the truncate-and-rescale technique in the proof of Theorem \ref{thm: difficile} produced functions which were bounded for $s \to -\infty$, while here they have a pole of order $\frac{\pi}{m}$. The reason is that there we took a limit for $m \to +\infty$ which made the pole disappear; on the other hand, here we work with a large, but fixed, value of $m$.

We fix $r_0>r$ and denote $\partial \mathfrak{P}_{(r_0)}= \{T_+= r_0 \} \subset \mathfrak{P}$; thus $\partial \mathfrak{P}_{(r_0)}$ is two-dimensional. 
If $r_0$ is generic and sufficiently large, $G(\partial \mathfrak{P}_{(r_0)})$ intersects transversely every branch of the moduli space ${\mathcal M}^{I=3, n_*=m}_{\overline{W}_-}(\bs{\gamma}, \mathbf{y}; \mathfrak{m})$ converging to a building $(u^{-1}, u^0, u^{+1})$ with
$u^{-1} \in \widehat{\mathcal M}_{\R \times [0,1] \times \overline{S}}^{I=1, n_*=1}(\mathbf{y}_\infty, \mathbf{y})$,  $u_0 \in {\mathcal M}_{\overline{W}_-}^{I=0, n_*=0}(\delta_0 \bs{\gamma'}, \mathbf{y}_\infty)$ and  $u_1 \in \widehat{\mathcal M}_{\R \times \overline{N}}^{I=2, n_*=m-1}(\bs{\gamma}, \delta_0 \bs{\gamma'})$.

We define $\Upsilon' \colon \partial \mathfrak{P}_{(r_0)} \to B_-$ such that $\Upsilon' (T_-, r_0, u^{-1}, u^0, u^{+1})$ is the projection to $B_-$ of the intersection point between the image of $G(T_-, r_0, u^{-1}, u^0, u^{+1})$ and $\sigma_{\infty}^-$. We define a second map $\Upsilon'' \colon \partial \mathfrak{P}_{(r_0)} \to B_-$ such that $\Upsilon'' (T_-, r_0, u^{-1}, u^0, u^{+1})$ is the (unique) zero of a holomorphic function $w$ whose constants $c_1$ and $c_2$ are determined by the asymptotic eigenfunctions of $u^{\pm 1}$. Both functions are proper, and therefore their degrees are well defined.

For $r_0$ large enough, the holomorphic curves in $G(\partial \mathfrak{P}_{(r_0)})$ which intersect $\sigma_\infty^-$ near $\mathfrak{m}$ are close to breaking into the building $(u^{-1}, u^0, u^{+1})$ by \cite[Lemma 7.12.15]{I}, so $\Upsilon'$ and $\Upsilon''$ are $C^0$ close by the truncate-and-rescale argument defining $w$ and therefore they have the same degree. Since $(0, \frac 32)$ is a regular value for $\Upsilon'$ and $\Upsilon''$ depends only on the asymptotic eigenfunctions of the ends of $u^{\pm 1}$, this implies that the cardinality of the set
$$G(\partial \mathfrak{P}_{(r_0)}) \cap {\mathcal M}^{I=3, n_*=m}_{\overline{W}_-}(\bs{\gamma}, \mathbf{y}; \mathfrak{m})$$
depends only on the asymptotic eigenfunctions of the ends of $u^{\pm 1}$. Now recall that $u^{-1}$ is a thin strip. The asymptotic eigenfunction of its positive end, which can be easily computed explicitly --- see \cite[Subsection 7.7.1]{I},  depends only on the local behavior of the arcs $\overline{a}_i$ and $\overline{\hh}(\overline{a}_i)$, which is the same for both ends of the arc. Thus the same construction with $\mathbf{y}'$ in the place of $\mathbf{y}$
gives the same result. Thus we have proved the following theorem.
\begin{thm}
$\Psi$ is a chain map.
\end{thm}

\section{Homotopies}\label{sec: homotopies}
In this section we prove the following.
\begin{thm}\label{Phi is an isomorphism}
$\Phi_* \colon \widehat{HF}(S, \mathbf{a}, \hh(\mathbf{a})) \to PFH_{2g}(N, \alpha_0, \omega)$ is an isomorphism.
\end{thm}

In order to prove Theorem \ref{Phi is an isomorphism} we define a chain homotopy
$$H \colon \widehat{CF}(S, \mathbf{a}, \hh(\mathbf{a})) \to \widehat{CF}(S, \mathbf{a}, \hh(\mathbf{a}))$$
between $\Psi \circ \Phi$ and a quasi-isomorphism $\Theta$, and a chain homotopy
$$K \colon PFC_{2g}(N, \alpha_0, \omega) \to  PFC_{2g}(N, \alpha_0, \omega)$$
between $\Phi \circ \Psi$ and the identity by counting isolated holomorphic curves in one-parameter families of cobordisms. The definition of these maps and the proof that they have the required properties  occupies the largest part of \cite{II}.

We define the fibration $\pi_{\overline{W}_+} \colon \overline{W}_+ \to B_+$ with fibre $\overline{S}$ and monodromy $\overline{\hh}$. In $\overline{W}_+$ we consider the singular Lagrangian submanifold $L_{\mathbf{a}} \subset \partial \overline{W}_+$ defined as the trace of the parallel transport of a copy of $\overline{\mathbf{a}}$ on $\pi_{\overline{W}_+}^{-1}(3,1)$ along $\partial B_+$. Given $\mathbf{y} \in {\mathcal S}_{\mathbf{a}, \hh(\mathbf{a})}$ and $\bs{\gamma} \in \widehat{\mathcal O}_{2g}$, we denote by
${\mathcal M}_{\overline{W}_+}(\mathbf{y}, \bs{\gamma})$ the moduli space of holomorphic curves in $\overline{W}_+$ with boundary on $L_{\overline{\mathbf{a}}}$ which are positively asymptotic to the chords over $\mathbf{y}$ and negatively asymptotic to $\bs{\gamma}$.
We define the multisection $(\sigma_\infty^+)^\dagger$ in analogy with $(\sigma_\infty^-)^\dagger$ in $\overline{W}_-$ and denote the induced intersection number by $n_*$. If we choose the almost complex structure on $\overline{W}_+$ so that  $(\sigma_\infty^+)^\dagger$ is holomorphic, by \cite[Lemma 5.4.9]{I} we have
\begin{equation}\label{eqn: disagio}
{\mathcal M}_{\overline{W}_+}^{n_*=0}(\mathbf{y}, \bs{\gamma}) = {\mathcal M}_{W_+}(\mathbf{y}, \bs{\gamma}).
\end{equation}
\subsection{Homotopy for $\Psi \circ \Phi$} \label{homotopy I}
\begin{figure}
\begin{center}
\includegraphics[width=10cm]{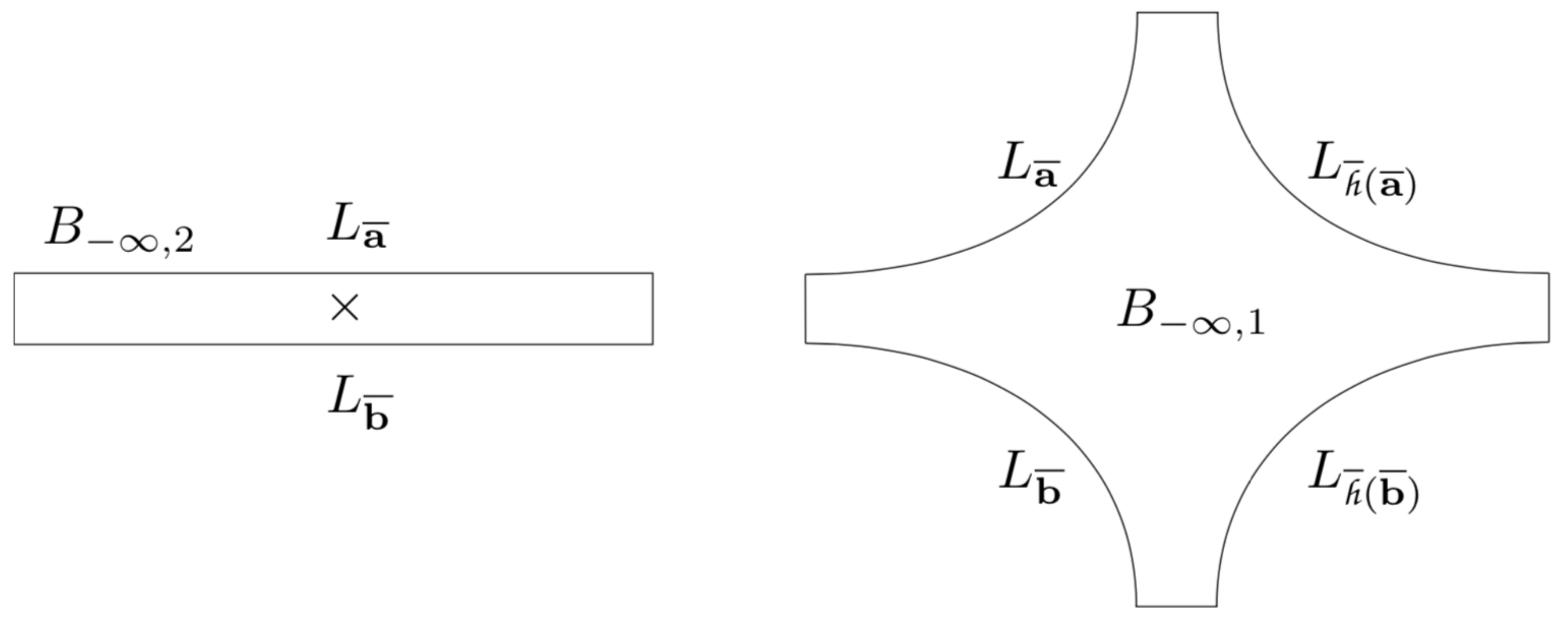}
\end{center}
\caption{The surface $B_{- \infty,2}$ to the left and $B_{-\infty,1}$ to the right. The $\times$ represents the point $\mathfrak{b}_{- \infty}$ and the labels on the boundary are the Lagrangian submanifolds.}
\label{figure: winfty}
\end{figure}
For $\tau \in \R$ we define $B_\tau$ as a smoothing of
 $$\R \times \R / 2 \Z \setminus ((e^\tau, + \infty) \times (1,2) \cup (- \infty, -e^\tau) \times (1,2)),$$
which can also be seen, up to a biholomorphism, as an annulus with a puncture on each boundary component. We choose a smooth family of points\footnote{These points are denoted by $\overline{\mathfrak{m}}^b(\tau)$ in \cite{II}.}
$$\mathfrak{b}_\tau \in (-e^\tau, e^\tau) \times \left \{ \frac 32 \right \} \subset B_\tau$$
such that $(B_\tau, \mathfrak{b}_\tau)$ converges (in the SFT sense; see \cite[Section 3.4]{BEHWZ}) to $(B_{\pm \infty}, \mathfrak{b}_{\pm \infty})$,  as $\tau \to \pm \infty$, where:
\begin{itemize}
\item $B_{+ \infty}= B_+ \sqcup B_-$ (with $B_+$ on top of $B_-$) and $\mathfrak{b}_{+ \infty}= (0, \frac 32) \in B_-$, and
\item $B_{- \infty}= B_{- \infty, 1} \sqcup B_{-\infty, 2}$ where $B_{- \infty, 1}$ is obtained from $\{-2 \le s \le 2 \} \cup \{ 0 \le t \le 1 \} \subset \R^2$ by smoothing the
corners (and therefore is biholomorphic to $D^2 \setminus \{ \pm 1, \pm i \}$), $B_{- \infty, 2} = [-2, 2] \times \R$ and $\mathfrak{b}_{-\infty} \in (0,0) \in [-2,2] \times \R$. See Figure \ref{figure: winfty}.
\end{itemize}
Every $(B_\tau, \mathfrak{b}_\tau)$ admits an anti-holomorphic involution which contains $\mathfrak{b}_\tau$ in its fixed point set. This property is used in a symmetry argument similar to that of Subsection \ref{Psi si a chain map}.

For each $\tau \in \R$, we define a fibration
$$\pi_{\overline{W}_\tau} \colon \overline{W}_\tau \to B_\tau$$
with fibre $\overline{S}$ and monodromy $\overline{\hh}$ (depending on an integer $m \gg0$ which will be suppressed from the notation; the reader should keep in mind that all statements hold for $m$ large enough). For $\tau = \pm \infty$ they extend to fibrations over the limit surfaces. The total space $\overline{W}_\tau$ (for $\tau \in \R \cup \{ \pm \infty \}$) admits a symplectic form $\Omega_{\overline{W}_\tau}$ such that all fibres are symplectic.

\begin{figure}
\begin{center}
\includegraphics[width=2cm]{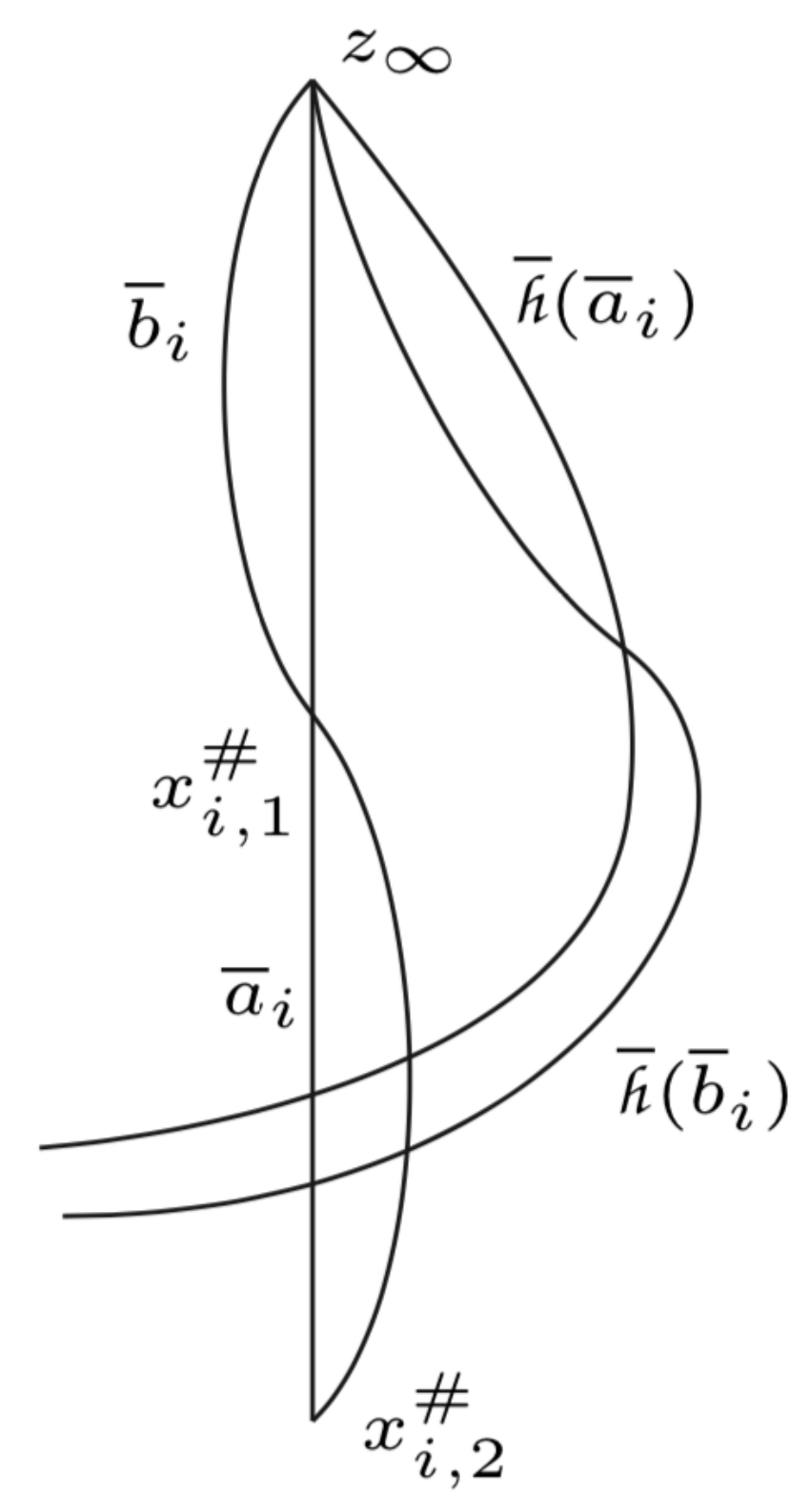}
\end{center}
\caption{The arcs $\overline{a}_i$, $\overline{b}_i$, $\overline{\hh}(\overline{a}_i)$, and $\overline{\hh}(\overline{b}_i)$ near $z_\infty$.}
\label{figure: aandb}
\end{figure}
Let $\overline{\mathbf{b}}= (\overline{b}_1, \ldots, \overline{b}_{2g})$ be small Hamiltonian perturbations of $\overline{\mathbf{a}}$ (depending on $m$) such that
\begin{itemize}
\item $\overline{a}_i \cap \overline{b}_i= \{ z_\infty, x^\sharp_{i,1}, x^\sharp_{i,2}, x^\sharp_{i,3} \}$, where $x^\sharp_{i,2} \in \op{int}(S)$ and $x^\sharp_{i,1}, x^\sharp_{i,3} \in \overline{S} \setminus S$,
\item $\overline{b}_i$, near $z_\infty$, is obtained from $\overline{a}_i$ by a small clockwise rotation, and
\item $\overline{b}_i$ approaches $\overline{a}_i$ sufficiently fast as $m \to \infty$.
\end{itemize}
See Figure \ref{figure: aandb}.
We denote $\partial_+ B_\tau = \partial B_\tau \cap [0, + \infty) \times \R /2 \Z$ and $\partial_- B_\tau = \partial B_\tau \cap (- \infty, 0] \times \R /2 \Z$. On $\partial_+ B_\tau$ we consider the singular Lagrangian submanifold $L_{\overline{\mathbf{a}}}$ obtained by parallel transporting $\overline{\mathbf{a}}$ and on $\partial_- B_\tau$ we consider the singular Lagrangian submanifold $L_{\overline{\mathbf{b}}}$ obtained by parallel transporting $\overline{\mathbf{b}}$, so that
\begin{align*}
L_{\mathbf{a}} \cap [2 e^\tau, + \infty) \times \{ 1 \} \times \overline{S} & =  [2 e^\tau, + \infty) \times \{ 1 \} \times \overline{\mathbf{a}}, \\
L_{\mathbf{a}} \cap [2 e^\tau, + \infty) \times \{ 0 \} \times \overline{S} & =  [2 e^\tau, + \infty) \times \{ 0 \} \times \overline{\hh}(\overline{\mathbf{a}}), \\
L_{\overline{\mathbf{b}}} \cap (- \infty, - 2e^\tau] \times \{ 1 \} \times \overline{S} & =  (-\infty, -2e^\tau] \times \{ 1 \} \times \overline{\mathbf{b}}, \\
L_{\overline{\mathbf{b}}} \cap (-\infty, -2e^\tau] \times \{ 0 \} \times \overline{S} & =  (- \infty, -2e^\tau] \times \{ 0 \} \times \overline{\hh}(\overline{\mathbf{b}}).
\end{align*}
The Lagrangian submanifolds $L_{\overline{\mathbf{a}}}$ and $L_{\overline{\mathbf{b}}}$ induce Lagrangian submanifolds in $\partial \overline{W}_{\pm \infty}$ in an obvious way. See Figure \ref{figure: winfty} for the Lagrangian submanifolds in $\partial \overline{W}_{-\infty}$.

Let $W_\tau \subset \overline{W}_\tau$ be the total space of a fibration with fibre $S$ and monodromy $\hh$. We choose a generic family of almost complex structures $J_\tau$ on $\overline{W}_\tau$  which restrict to the pull back of a split almost complex structure on $\overline{W}_\tau \setminus W_\tau \cong B_\tau \times D$. We also fix a family of marked points $\mathfrak{m}_\tau = (\mathfrak{b}_\tau, z_\infty) \in \overline{W}_\tau$.

Given $\mathbf{y}_+ \in {\mathcal S}_{\mathbf{a}, \hh(\mathbf{a})}$ and $\mathbf{y}_- \in {\mathcal S}_{\mathbf{b}, \hh(\mathbf{b})}$, we denote by ${\mathcal M}_{\overline{W}_*}(\mathbf{y}_+, \mathbf{y}_-; \mathfrak{m}_*)$ the moduli space of  pairs $(\tau, u)$, where $\tau \in \R$ and $u$ is a $J_\tau$-holomorphic curves in $\overline{W}_\tau$ with boundary on $L_{\mathbf{a}} \cup L_{\mathbf{b}}$ which is positively asymptotic to the chords over $\mathbf{y}_+$, negatively asymptotic to the chords $\mathbf{y}_-$ and passes through $\mathfrak{m}_\tau$. To any such $u$ we associate the ECH-type index $I(u)$ satisfying the index inequality --- see \cite[Subsection 3.2.6]{II} --- and the intersection number $n_*(u)$, which is the algebraic intersection between $u$ and the $J_\tau$-holomorphic multisection $(\sigma_\infty^\tau)^\dagger$ defined by a point $z_\infty^\dagger \in \overline{S}$ close to $z_\infty$ and not on the arcs $\overline{\mathbf{a}}$ and $\overline{\mathbf{b}}$; see \cite[Notation 3.2.9]{II}.

\begin{dfn}
We define\footnote{For sake of exposition we define here a simplified homotopy map. As explained later, the actual homotopy map must take into account the chord over $z_\infty$.} the map $\overline{H} \colon \overline{CF}(S, \mathbf{a}, \hh(\mathbf{a})) \to \overline{CF}(S, \mathbf{b}, \hh(\mathbf{b}))$ by
\begin{equation}\label{eqn: homotopy H}
\overline{H}(\mathbf{y}_+) = \sum \limits_{\mathbf{y}_- \in {\mathcal S}_{\mathbf{b}, \hh(\mathbf{b})}} \# {\mathcal M}_{\overline{W}_*}^{I=1, n_*(u)=m}(\mathbf{y}_+, \mathbf{y}_-; \mathfrak{m}_*) \mathbf{y}_-.
\end{equation}
\end{dfn}
Let ${\mathcal S}^\sharp$ be the set of unordered $2g$-tuple of intersection points ${x^\sharp_{1, j(1)}, \ldots, x^\sharp_{2g, j(2g)}}$ between $\overline{\mathbf{a}}$ and $\overline{\mathbf{b}}$. In particular no intersection point is equal to $z_\infty$.
Given $\mathbf{y}_+ \in {\mathcal S}_{\mathbf{a}, \hh(\mathbf{a})}$, $\mathbf{y}_- \in {\mathcal S}_{\mathbf{b}, \hh(\mathbf{b})}$ and $\mathbf{x}_\pm \in {\mathcal S}^\sharp$,
we will denote the moduli spaces of holomorphic curves in $\overline{W}_{-\infty, 1}$ with boundary conditions $L_{\overline{\mathbf{a}}} \cup L_{\overline{\mathbf{b}}} \cup L_{\overline{\hh}(\overline{\mathbf{b}})} \cup L_{\overline{\hh}(\overline{\mathbf{a}})}$ and asymptotic to the chords over $\mathbf{y}_+$, $\mathbf{x}_+$, $\mathbf{y}_-$ and $\overline{\hh}(\mathbf{x}_-)$ at the four ends of $\overline{W}_{- \infty, 1}$ (in counterclockwise order starting from the top) by ${\mathcal M}_{\overline{W}_{- \infty, 1}}(\mathbf{y}_+, \mathbf{x}_+, \mathbf{y}_-, \overline{\hh}(\mathbf{x}_-))$. Similarly, we denote the moduli space of holomorphic curves in $W_{- \infty, 2}$ with boundary conditions $L_{\overline{\mathbf{a}}} \cup L_{\overline{\mathbf{b}}}$ which are asymptotic to the chords over $\mathbf{x}^\sharp_\pm$ for $t \to \pm \infty$ and pass through $\mathfrak{m}_{-\infty}=(0,0)$ by ${\mathcal M}_{\overline{W}_{- \infty, 2}}(\mathbf{x}^\sharp_+, \mathbf{x}^\sharp_-; \mathfrak{m}_{- \infty})$. In the following definition we abbreviate 
\begin{align*}
{\mathcal M}_1(\mathbf{y}_+, \mathbf{x}^\sharp_+, \mathbf{y}_-, \overline{\hh}(\mathbf{x}^\sharp_-)) & = {\mathcal M}_{\overline{W}_{- \infty, 1}}^{I=0, n_*=0}(\mathbf{y}_+, \mathbf{x}_+, \mathbf{y}_-, \overline{\hh}(\mathbf{x}_-)) \quad \text{and}  \\ {\mathcal M}_2(\mathbf{x}^\sharp_+, \mathbf{x}^\sharp_-) & = {\mathcal M}_{\overline{W}_{- \infty, 2}}^{I=2, n_*=m}(\mathbf{x}^\sharp_+, \mathbf{x}^\sharp_-; \mathfrak{m}_{- \infty})
\end{align*}
\begin{dfn}
We define the map $\overline{\Theta} \colon \overline{CF}(S, \mathbf{a}, \hh(\mathbf{a})) \to \overline{CF}(S, \mathbf{b}, \hh(\mathbf{b}))$ as
\begin{equation}\label{eqn: isomorphism 1}
\overline{\Theta}(\mathbf{y}_+) = \sum \limits_{\mathbf{y}_- \in {\mathcal S}_{\mathbf{b}, \hh(\mathbf{b})}} \sum \limits_{\mathbf{x}_\pm^\sharp \in {\mathcal S}^\sharp} \# \left ( {\mathcal M}_1(\mathbf{y}_+, \mathbf{x}_+^\sharp, \mathbf{y}_-, \overline{\hh}(\mathbf{x}_-^\sharp)) \times {\mathcal M}_2(\mathbf{x}^\sharp_+, \mathbf{x}^\sharp_-) \right ) \mathbf{y}_-.
\end{equation}
\end{dfn}

\begin{thm}\label{thm: H is a homotopy}
The maps $\overline{H}$ and $\overline{\Theta}$ satisfy the relation
\begin{equation}\label{eqn: H is a homotopy}
\overline{\partial} \circ \overline{H}+ \overline{H} \circ \overline{\partial}=  \overline{\Psi} \circ \overline{\Phi} + \overline{\Theta} + \overline{V}
\end{equation}
for a map $\overline{V}$ whose image is generated by elements of the form $ \{ x_i \} \cup \mathbf{y}' + \{ x_i' \} \cup \mathbf{y}'$ where $x_i$ and $x_i'$ here denote the intersection points of $b_i$ ad $\hh(b_i)$ on $\partial S$ and $\mathbf{y}'$ is a $(2g-1)$-tuple of intersection points between the remaining arcs.
\end{thm}
Theorem \ref{thm: H is a homotopy} is proved by analysing the degenerations of the moduli spaces ${\mathcal M}_{\overline{W}_*}^{I=2, n_*=m}(\mathbf{y}_+, \mathbf{y}_-; \mathfrak{m}_*)$. Degenerations which occur at some $\tau \in \R$ contribute to the left-hand side, degenerations which occur at $\tau \to + \infty$ contribute to $\overline{\Psi} \circ \overline{\Phi} + \overline{V}$ and degenerations which occur at $\tau \to - \infty$ contribute to $\overline{\Theta}$. For $\tau \to + \infty$, the cobordisms $\overline{W}_\tau$ degenerate into the juxtaposition of $\overline{W}_+$ on top of $\overline{W}_-$ with the base-point $\mathfrak{m}_{+ \infty} \in \overline{W}_-$. Sequences of holomorphic curves in ${\mathcal M}_{\overline{W}_*}^{I=2, n_*=m}(\mathbf{y}_+, \mathbf{y}_-; \mathfrak{m}_*)$ degenerate, for $\tau \to \infty$, either to two-level buildings in ${\mathcal M}_{\overline{W}_+}^{I=0, n_*=0}(\mathbf{y}_+, \bs{\gamma}) \times
{\mathcal M}_{\overline{W}_-}^{I=2, n_*=m}(\bs{\gamma}, \mathbf{y}; \mathfrak{m}_{+ \infty})$ for some $\bs{\gamma} \in \widehat{\mathcal O}_{2g}$ or to three-level buildings
involving a holomorphic curve in $\overline{W}_+$ with a negative and at $\delta_0$, the section at infinity $\sigma_{\infty}^-$ in $\overline{W}_-$, and a thin strip in $\R \times [0,1] \times \overline{S}$. The two-level degenerations contribute to $\overline{\Psi} \circ \overline{\Phi}$ by Equation \eqref{eqn: disagio}, while the three-level degenerations contribute to $\overline{V}$.

The main tools we use in the analysis of bad degenerations are, as for proving that $\Psi$ is a chain map, the intersection numbers $n_*$, the ECH-type index $I$ and truncate-and-rescale arguments to analyse how a degenerating sequence approaches its limit. However, the cases to treat are many more than in the previous section and their analysis is more difficult.

\begin{thm} \label{understanding Theta}
The map $\overline{\Theta}$ induces a quasi-isomorphism
$$\Theta \colon \widehat{CF}(S, \mathbf{a}, \hh(\mathbf{a})) \to \widehat{CF}(S, \mathbf{b}, \hh(\mathbf{b})).$$
\end{thm}
Before proving Theorem \ref{understanding Theta} we explain how, together with Equation \eqref{eqn: H is a homotopy}, it proves that $\Psi_* \circ \Phi_*$ is an isomorphism. The missing step is to prove that $\overline{H}$ induces a map $H \colon \widehat{CF}(S, \mathbf{a}, \hh(\mathbf{a})) \to \widehat{CF}(S, \mathbf{b}, \hh(\mathbf{b}))$. Unfortunately we were not able to prove directly that $\overline{H}(\{ x_i \} \cup \mathbf{y}')= \overline{H}(\{ x_i' \} \cup \mathbf{y}')$ if $\mathbf{y}'$ is a $(2g-1)$-tuple of intersection points between the remaining arcs, and thus we extend all maps in Equation \eqref{eqn: H is a homotopy} to a chain complex generated by $2g$-tuples of intersection points which can contain $z_\infty$ and whose homology is $\widehat{HF}(S, \mathbf{a}, \hh({\mathbf{a}}))$. This, of course, brings in more bad degenerations to control.

Now we sketch the main steps of the proof of Theorem \ref{understanding Theta}. We consider the two subsets ${\mathcal S}^\sharp_{odd}$ and ${\mathcal S}^\sharp_{even}$ of ${\mathcal S}^\sharp$: the first consists of $2g$-tuple $\{ x^\sharp_{1, j(1)}, \ldots, x^\sharp_{2g, j(2g)} \}$ where $j(i) \in \{1,3 \}$, and the second of the unique element $\{ x^\sharp_{1, 2}, \ldots, x^\sharp_{2g, 2} \}$. By \cite[Lemma 3.2.15]{II}, the moduli spaces ${\mathcal M}_{W_{- \infty, 2}}^{I=2, n_*=m}(x_+^\sharp, x_-^\sharp; \mathfrak{m}_{-\infty})$ is nonempty only when $x_-^\sharp \in {\mathcal S}^\sharp_{odd}$ and $x_+^\sharp \in {\mathcal S}^\sharp_{even}$.

For any $x_-^\sharp \in {\mathcal S}^\sharp_{odd}$ and $x_+^\sharp \in {\mathcal S}^\sharp_{even}$, the number of elements in the moduli space ${\mathcal M}_{W_{- \infty, 2}}^{I=2, n_*=m}(x_+^\sharp, x_-^\sharp; \mathfrak{m}_{-\infty})$ is equal to the relative Gromov-Taubes invariant $G_3$ defined as the number of embedded holomorphic curves in $D^2 \times \overline{S}$ with boundary on $\partial D^2 \times \mathbf{a}$, representing the relative homology class $[\overline{S}]+2g[D^2]$ and passing through $(0, z_\infty), (x_1, 1), \ldots, (x_{2g}, 1), (x_{2g+1}, -1), \ldots, (x_{4g}, -1)$ with $x_i, x_{i+2g} \in \overline{a}_i$. See \cite[Subsection 2.4.1]{II}.

By \cite[Theorem 2.4.2]{II} $G_3=1$; the proof is similar to the computation of the relative Gromov-Taubes invariant in Subsection \ref{subsec: geometric U}. Thus we can rewrite
$$ \overline{\Theta}(\mathbf{y}_+) = \sum \limits_{\mathbf{y}_- \in {\mathcal S}_{\mathbf{b}, \hh(\mathbf{b})}} \sum \limits_{\mathbf{x}_+^\sharp \in {\mathcal S}^\sharp_{odd}} \sum \limits_{x^\sharp_- \in {\mathcal S}^\sharp_{even}} \# {\mathcal M}_{\overline{W}_{- \infty, 1}}^{I=0, n_*=0}(\mathbf{y}_+, \mathbf{x}_+^\sharp, \mathbf{y}_-, \overline{\hh}(\mathbf{x}_-^\sharp))  \mathbf{y}_-.$$
This sum now has a fairly standard interpretation in Heegaard Floer theory as composition of two triangle maps which are know to give an isomorphism in homology: see \cite[Lemma 3.3.13]{II}. We can also argue as follows: any $2g$-tuple of intersection points $\mathbf{y} \in {\mathcal S}_{\mathbf{a}, \hh(\mathbf{a})}$ has a closest $2g$-tuple of intersection points $\mathbf{y}_ \in  {\mathcal S}_{\mathbf{a}, \hh(\mathbf{a})}$. There is a unique holomorphic curve in $W_{-\infty, 1}$ of small energy with $I=0$, $n_*=0$ and ends at the chords over $\mathbf{y}_+$, $\mathbf{x}^\sharp_+ \in {\mathcal S}^\sharp_{even}$,  $\mathbf{y}_-$ and $\mathbf{x}_- \in {\mathcal S}^\sharp_{odd}$: it is a union of $2g$ sections which project to the fibre on the fishtail-shaped quadrilateral in Figure \ref{figure: aandb}.

\begin{rem}
It might be surprising, at a first uncareful look, that those curves have index zero. The reason is that the concave angle gives a holomorphic map from a disc with four punctures on the boundary to $\overline{S}$ covering the quadrilateral for any conformal structure of the punctured disc, which depends on the cross-ratio of the four punctures. However, we are not looking just for maps into $\overline{S}$ but for multisections of $\overline{W}_{- \infty, 1} \to B_{- \infty, 1}$, so we also have a holomorphic map from the four-punctured disc to $B_{- \infty, 1}$.  The conformal structure of $B_{-\infty, 1}$ is fixed, and that fixes the conformal structure of the four-punctured disc.
\end{rem}

Thus the ``low energy part'' of $\overline{\Theta}$ is an isomorphism, and therefore $\overline{\Theta}$ is a quasi-isomorphism by standard homological algebra.
If a component of $\mathbf{y}_+$ is either $x_i$ or $x_i'$, a direct inspection shows that every element in the moduli space ${\mathcal M}_{\overline{W}_{- \infty, 1}}^{I=0, n_*=0}(\mathbf{y}_+, \mathbf{x}_+^\sharp, \mathbf{y}_-, \overline{\hh}(\mathbf{x}_-^\sharp))$ must have the small energy quadrilateral starting at the chord over $x_i$ or $x_i'$ as an irreducible component. This implies that $\overline{\Theta}(\{ x_i \} \cup \mathbf{y}')= \overline{\Theta}(\{ x_i' \} \cup \mathbf{y}')$ for all $\mathbf{y}'$, and therefore $\Theta$ induces a well defined map
$\Theta \colon \widehat{CF}(S, \mathbf{a}, \hh(\mathbf{a})) \to  \widehat{CF}(S, \mathbf{a}, \hh(\mathbf{a}))$ which is still a quasi-isomorphism.

\subsection{Homotopy for $\Phi \circ \Psi$}
In this section $B_\tau$ will denote a family of cylinders with a disc removed which, for $\tau \to \pm \infty$, converges to $B_{\pm \infty}$ such that:
\begin{itemize}
\item $B_{+ \infty}= B_- \sqcup B_+$ (with $B_-$ on top of $B_+$), and
\item $B_{- \infty}= B_{- \infty, 1} \vee B_{- \infty, 2}$ where $B_{- \infty, 1}= \R \times \R /2 \Z$, $B_{- \infty, 2}= D^2$, and the two components are attached by identifying $(0, \frac 32) \in \R \times \R/ 2 \Z$ with $0 \in D^2$.
\end{itemize}
In $B_\tau$ we choose a marked point $\mathfrak{b}_\tau$ (smoothly in $\tau$) such that
$\lim \limits_{\tau \to \pm \infty} \mathfrak{b}_\tau = \mathfrak{b}_{\pm \infty}$ where $\mathfrak{b}_{+ \infty}= (0, \frac 32) \in B_-$ and $\mathfrak{b}_{- \infty}= \frac i2 \in D^2$.
We will also assume that every $(B_\tau, \mathfrak{b}_\tau)$ admits an anti-holomorphic involution and that $\mathfrak{b}_\tau$ is contained in its fixed point set. This property is used in a symmetry argument similar to that of Subsection \ref{Psi si a chain map}.

\begin{dfn}\label{dfn: double dagger}
The Reeb vector field of $N$ satisfies {\em Property $(\dagger \dagger)_i$} if every  simple elliptic orbit $\gamma$ intersecting a fibre at most $i$ times has first return map which is conjugated to a rotation by a sufficiently small negative angle so that the incoming partition of $(\gamma, j)$, for $j=1, \ldots, i$, is $(\underbrace{1, \ldots, 1}_{\text{$j$ times}})$ and the outgoing partition is $(j)$.
\end{dfn}
By \cite[Section 2.5]{I} it is possible to assume Property $(\dagger \dagger)_i$, for any $i$, after a $C^1$-small modification of the stable Hamiltonian structure supported on arbitrarily small neighbourhoods of the elliptic orbits intersecting a fibre at most $i$ times. Note that  we can see Property $(\dagger\dagger)_i$ as a property of the monodromy $\hh$. In this subsection we will always assume Property $(\dagger \dagger)_{2g}$.

For each $\tau \in \R$, we define a fibration
$$\pi_{\overline{W}_\tau} \colon \overline{W}_\tau \to B_\tau$$
with fibre $\overline{S}$ and monodromy $\overline{\hh}$ depending on an integer $m \gg0$ which will be suppressed from the notation. As in the previous subsection, the reader should keep in mind that all statements hold for $m$ large enough. For $\tau = \pm \infty$ those fibrations extend to fibrations over the limit surfaces. The total space $\overline{W}_\tau$ (for $\tau \in \R \cup \{ \pm \infty \}$) admits a symplectic form $\Omega_{\overline{W}_\tau}$ such that all fibres are symplectic. On $\partial \overline{W}_\tau$, for $\tau \in \R \cup \{ \pm \infty \}$, we consider the Lagrangian submanifold $L_{\overline{\mathbf{a}}}$ defined as the trace of the symplectic parallel transport along $\partial B_\tau$ applied to a copy of $\overline{\mathbf{a}}$ on the fibre over a point of $\partial B_\tau$. We also fixed the family of marked points $\mathfrak{m}_\tau =(\mathfrak{b}_\tau, z_\infty) \in \overline{W}_\tau$.

Let $W_\tau \subset \overline{W}_\tau$ be the total space of the fibration over $B_\tau$ with fibre $S$ and monodromy $\hh$. We choose a generic family of almost complex structures $J_\tau$ on $\overline{W}_\tau$ among those which restrict to the pull back of a split almost complex structure on $\overline{W}_\tau \setminus W_\tau \cong B_\tau \times D$. Since $\overline{W}_{-\infty, 1}= \R \times \overline{N}$, we also require that $J_{- \infty, 1}$ be compatible with the stable Hamiltonian structure on $\overline{N}$.

Given orbit sets $\bs{\gamma}_\pm \in \widehat{\mathcal O}_{2g}$, we denote by 
${\mathcal M}_{\overline{W}_*}(\bs{\gamma}_+, \bs{\gamma}_-; \mathfrak{m}_*)$
the moduli space of pairs $(u, \tau)$ where $\tau \in \R$ and $u$ is a $J_\tau$-holomorphic curves in $\overline{W}_\tau$ with boundary on $L_{\overline{\mathbf{a}}}$ which is positively asymptotic to $\bs{\gamma}_+$, negatively asymptotic to $\bs{\gamma}_-$ and passes through $\mathfrak{m}_\tau$.  To each $u$ in this moduli space we associate the ECH-type index $I(u)$ satisfying the index inequality --- see \cite[Subsection 4.2.6]{II} --- and the intersection number $n_*(u)$, which is the algebraic intersection between $u$ and the $J_\tau$-holomorphic multisection $(\sigma_\infty^\tau)^\dagger$ defined by a point $z_\infty^\dagger \in \overline{S}$ close to $z_\infty$ and not on the arcs $\overline{\mathbf{a}}$.

\begin{dfn}
We define the map $K \colon PFC_{2g}(N, \alpha_0, \omega) \to PFC_{2g}(N, \alpha_0, \omega)$ by
\begin{equation} \label{eqn: definition of K}
K(\bs{\gamma}_+)= \sum \limits_{\bs{\gamma}_- \in \widehat{\mathcal O}_{2g}} \# {\mathcal M}_{\overline{W}_*}^{I=1, n_*=m}(\bs{\gamma}_+, \bs{\gamma}_-; \mathfrak{m}_*) \bs{\gamma}_-.
\end{equation}
\end{dfn}

We say that a holomorphic $u$ curve in $\overline{W}_{- \infty, 1}$ passes through $\zeta \in \pi^{-1}_{\overline{W}_{- \infty, 1}}(0, \frac 32)$ with multiplicity $r$ if the multiplicity of $\zeta$ as intersection point between $u$ and $\pi^{-1}_{\overline{W}_{- \infty, 1}}(0, \frac 32)$ is $r$. To make sense of the multiplicity of the intersection we must assume that  $\pi^{-1}_{\overline{W}_{- \infty, 1}}(0, \frac 32)$ is holomorphic. In fact, in \cite{I} and \cite{II} we assume that most almost complex structures are compatible with the various fibrations.

We trivialise the fibration 
$$\pi_{\overline{W}_{- \infty, 2}} \colon \overline{W}_{- \infty, 2} \to B_{- \infty, 2}=D^2$$ 
and obtain a commutative diagram
$$\xymatrix{
\overline{W}_{- \infty, 2} \ar[rr]^{\cong} \ar[dr]_{\pi_{\overline{W}_{- \infty, 2}}} & & D^2 \times \overline{S} \ar[dl] \\
& D^2.}$$
We pull back the identification $\pi_{\overline{W}_{- \infty, 2}}^{-1}(0) \cong \overline{S}$ given by the diagram above to an identification $\pi^{-1}_{\overline{W}_{- \infty, 1}}(0, \frac 32) \cong \overline{S}$ via the gluing $\pi^{-1}_{\overline{W}_{- \infty, 1}}(0, \frac 32) \cong \pi_{\overline{W}_{- \infty, 2}}^{-1}(0)$.
 Thus a set of ``points with multiplicities'' $\mathfrak{z}=\{(\zeta_1, r_1), \ldots, (\zeta_l, r_l)\}$ with $\zeta_i \in \pi^{-1}_{\overline{W}_{- \infty, 1}}(0, \frac 32)$ with $r_i>0$ and $r_1 + \ldots + r_l=2g$ can be seen as an element $\mathfrak{z} \in \op{Sym}^{2g}(\overline{S})$. 

Given $\mathfrak{z}=\{(\zeta_1, r_1), \ldots, (\zeta_l, r_l)\} \in \op{Sym}^{2g}(\overline{S})$,  we denote  by ${\mathcal M}_{\overline{W}_{- \infty, 1}}(\bs{\gamma}_+, \bs{\gamma}_-; \mathfrak{z})$ the moduli space of holomorphic curves in $\overline{W}_{- \infty, 1}=\R \times \overline{N}$ which are positively asymptotic to $\bs{\gamma}_+$ and negatively asymptotic to $\bs{\gamma}_-$ for $\bs{\gamma}_\pm \in \widehat{\mathcal O}_{2g}$ and pass through $\zeta_i$ with multiplicity $r_i$ for each $i=1, \ldots, l$. 

Similarly, we 
denote by ${\mathcal M}_{\overline{W}_{- \infty, 2}}(\mathfrak{m}, \mathfrak{z})$ the moduli space of holomorphic curves in $\overline{W}_{- \infty, 2}= D^2 \times \overline{S}$ with boundary on $L_{\overline{\mathbf{a}}}= \partial D^2 \times \overline{\mathbf{a}}$, representing the relative homology class $[\overline{S}] + 2g[D^2]$ and passing through $\mathfrak{m}_{-\infty}=(\frac i2, z_\infty)$ and $\mathfrak{z}$. 
\begin{dfn}
We define the map $\Xi \colon PFC_{2g}(N) \to PFC_{2g}(N)$ as
$$\Xi(\bs{\gamma}_-)= \sum \limits_{\bs{\gamma}_- \in \widehat{\mathcal O}_{2g}} \sum_{\mathfrak{z} \in \op{Sym}^{2g}(\overline{S})} \# \left ({\mathcal M}_{\overline{W}_{- \infty, 1}}^{I=0, n^*=0}(\bs{\gamma}_+, \bs{\gamma}_-; \mathfrak{z}) \times {\mathcal M}_{\overline{W}_{- \infty, 2}}(\mathfrak{m}, \mathfrak{z}) \right ).$$
\end{dfn}
The sum is finite because the moduli spaces ${\mathcal M}_{\overline{W}_{- \infty, 1}}^{I=0, n^*=0}(\bs{\gamma}_+, \bs{\gamma}_-; \mathfrak{z})$ are nonempty only for finitely many $\mathfrak{z} \in \op{Sym}^{2g}(\overline{S})$.

The following theorem is proved by analysing the degenerations of the moduli spaces ${\mathcal M}_{\overline{W}_\tau}^{I=2, n_*=m}(\bs{\gamma}_+, \bs{\gamma}_-; \mathfrak{m})$. To control the bad degenerations and show that either they cannot occur or they comes in pairs we use a combination of intersection theory, index computations, truncate-and-rescale and symmetry arguments like in Subsection \ref{Psi si a chain map}.
\begin{thm} \label{thm: K is a homotopy}
The maps $K$ and $\Xi$ satisfy the relation
\begin{equation}\label{eqn: K is a homotopy}
\partial \circ K + K \circ \partial = \Phi \circ \Psi + \Xi.
\end{equation}
\end{thm}
Degenerations of ${\mathcal M}_{\overline{W}_\tau}^{I=2, n_*=m}(\bs{\gamma}_+, \bs{\gamma}_-; \mathfrak{m})$ which occur at some $\tau \in \R$ contribute to the left-hand side, degenerations which occur at $\tau \to + \infty$ contribute to $\Phi \circ \Psi$ because $\overline{W}_{+ \infty}$ is the juxtaposition of $\overline{W}_-$ on top of $\overline{W}_+$
and the basepoint $\mathfrak{m}_{+ \infty}$ is in $\overline{W}_-$, as in the previous subsection, and degenerations which occur at $\tau \to + \infty$ contribute to $\Xi$.

Property $(\dagger \dagger)_{2g}$ is used in the proof Theorem \ref{thm: K is a homotopy} to simplify the gluing of the broken holomorphic curves appearing in the definition of $\Xi$ to holomorphic curves in  ${\mathcal M}_{\overline{W}_\tau}^{I=2, n_*=m}(\bs{\gamma}, \bs{\gamma}; \mathfrak{m})$ when there is an elliptic orbit with multiplicity larger that one in $\bs{\gamma}$; see the proof of \cite[Lemma 4.3.4]{II}. We take $\bs{\gamma}_+ = \bs{\gamma}_- = \bs{\gamma}$ because holomorphic curves in ${\mathcal M}_{\overline{W}_{- \infty. 1}}^{I=0, n_*=0}(\bs{\gamma}_+, \bs{\gamma}_-)$ are necessarily branched covers of trivial cylinders because we have chosen $J_{-\infty, i}$ compatible with the stable Hamiltonian structure.

To exemplify the use of Property $(\dagger \dagger)_{2g}$ suppose for simplicity that we want to glue $(v_1, v_2)$ where $v_1 \in {\mathcal M}_{\overline{W}_{- \infty, 1}}^{I=0, n^*=0}(\gamma^{2g}, \gamma^{2g}; \mathfrak{z})$  is a degree $2g$ branched cover of the trivial cylinder over some elliptic orbit $\gamma$. Then 
$\mathfrak{z}=\{(\zeta, 2g) \}$ with $\zeta = \gamma(\frac 32)$. Recall that, in the definition of the moduli spaces for embedded contact homology, branched covers of trivial cylinders with the same degree are identified, but when we glue we must choose a representative. The ends of the curve after gluing must satisfy the incoming and outgoing partitions, which means that we must choose a representative of $v_1$ satisfying the same conditions. The form of the partitions of $\gamma$ implies that we can choose $v_1$ to be a branched cover of $\gamma$ with a unique branch point at $\zeta$ of order $2g-1$.

To prove that $\Phi_* \circ \Psi_*$ is an isomorphism, it remains to compute $\Xi$.
\begin{thm}\label{thm: computation of Xi} $\Xi$ is the identity.
\end{thm}
To prove the theorem, we first compute the relative Gromov-Taubes invariant $G_2$ defined as the number of elements in the moduli spaces ${\mathcal M}_{\overline{W}_{- \infty, 2}}(\mathfrak{m}_{- \infty}, \mathfrak{z})$; see \cite[Section 2.3]{II}. To do so, we degenerate $B_{-\infty, 2} \cong D^2$ into a copy of $D^2$ attached to a copy of $S^2$ along a point so that $\mathfrak{m}_{-\infty}$ remains in $S^2$. Holomorphic curves in ${\mathcal M}_{\overline{W}_{- \infty, 2}}(\mathfrak{m}_{- \infty}, \mathfrak{z})$ degenerate to constant sections of $D^2 \times \overline{S}$ passing through the points in $\mathfrak{z}$ and holomorphic curves in $S^2 \times \overline{S}$ in the homology class $[\overline{S}] \times 2g[S^2]$ passing through $\mathfrak{m}_{- \infty}$ and $\mathfrak{z}$. We denote the moduli space of the latter holomorphic curves by ${\mathcal M}_{S^2 \times \overline{S}} (\mathfrak{m}_{- \infty}, \mathfrak{z})$. Thus $\# {\mathcal M}_{\overline{W}_{- \infty, 2}}(\mathfrak{m}_{- \infty}, \mathfrak{z}) = \# {\mathcal M}_{S^2 \times \overline{S}} (\mathfrak{m}_{- \infty}, \mathfrak{z})$. By \cite[Theorem 2.3.3]{II}
$\# {\mathcal M}_{S^2 \times \overline{S}} (\mathfrak{m}_{- \infty}, \mathfrak{z}) = 1$ for any $\mathfrak{z} \in \op{Sym}^{2g}(\overline{S})$. The proof is similar to the computation of the relative Gromov-Taubes invariant in Subsection \ref{subsec: geometric U}. Then
$$\Xi(\bs{\gamma}_+) = \sum \limits_{\bs{\gamma}_- \in \widehat{\mathcal O}_{2g}} \# {\mathcal M}_{\overline{W}_{- \infty, 1}}^{I=0, n^*=0}(\bs{\gamma}_+, \bs{\gamma}_+) \bs{\gamma}_-.$$
Since, for a suitable choice of almost complex structure, the moduli space ${\mathcal M}_{\overline{W}_{- \infty, 1}}^{I=0, n^*=0}(\bs{\gamma}_+, \bs{\gamma}_+)$ consists of trivial cylinders, $\Xi$ is the identity map. This proves Theorem \ref{thm: computation of Xi} and therefore concludes the proof of Theorem \ref{Phi is an isomorphism}.

\section{Stabilisation}\label{sec: stabilisation}
In Section \ref{sec: homotopies} we proved that $\Phi_* \colon \widehat{HF}(S, \mathbf{a}, \hh(\mathbf{a})) \to PFH_{2g}(N, \alpha_0, \omega)$ is an isomorphism if $\hh$ satisfies Property $(\dagger \dagger)_{2g}$. In view of Lemma \ref{lemma: PFH=ECH}, to prove that $\widehat{\Phi}_*$ is an isomorphism it remains to prove that $PFH_{2g}(N, \alpha_0, \omega)$ is isomorphic to $\widehat{PFH}(N, \partial N, \alpha_0, \omega)$; this will be done by a stabilisation argument.

In this section it would be desirable to assume Property $(\dagger \dagger)_i$ for all $i$, but this is not possible. Nevertheless, in practice, we can pretend it is for the following reason.
\begin{lemma}[See {\cite[Lemma 5.5.2]{II}}] \label{lemma: pretend i=infty}
There is a sequence of stable Hamiltonian structures $(\alpha_0^i, \omega^i)$ with Reeb vector field $R_i$ satisfying Property $(\dagger \dagger)_i$ such that
\begin{itemize}
\item $R_i$ and $R_{i+1}$ coincide outside of a small neighbourhood of the orbits of $R_i$  intersecting a fibre at most $i+1$ times, and
\item  the continuation maps
$$\mathfrak{K}_i \colon PFC_i(N, \alpha_0^i, \omega^i) \to PFC_i(N, \alpha_0^{i+1}, \omega^{i+1})$$
satisfy $\mathfrak{K}_i(\bs{\gamma})= \bs{\gamma}$ for every generator of $PFC_i(N, \alpha_0^i, \omega^i)$.
\end{itemize}
\end{lemma}
This implies that
$$\varinjlim PFH_i(N, \alpha_0^i, \omega^i)= \widehat{PFH}(N, \partial N, \alpha_0^0, \omega^0).$$

\begin{prop}\label{prop: stabilisation}
If $\hh$ satisfies  Property $(\dagger \dagger)_{2g+2}$, then the map
$$\mathfrak{j} \colon PFC_{2g}(N) \to PFC_{2g+2}(N), \quad \mathfrak{j}(\bs{\gamma})= e^2 \bs{\gamma}$$
induces an isomorphism in homology.
\end{prop}
We recall that $\mathfrak{j}$ is a chain map because no holomorphic curve in $\R \times N$ can have a nontrivial positive end at a cover of $e$.

Let $(T, \hh_T)$ be the open book decomposition of $S^3$ such that $T$ is a torus with one boundary component and $\hh_T$ is isotopic, relative to the boundary, to the product of right-handed Dehn twists along two curves intersecting exactly at one point. In particular, the binding is the right-handed trefoil knot. Let $N_T$ be the mapping torus of $(T, \hh_T)$. We also assume that $\hh_T^* \beta_T- \beta_T$ is exact for some Liouville form $\beta_T$ on $T$ and that the Reeb vector field  of the stable Hamiltonian structure on $N_T$ induced by the fibration is nondegenerate and satisfies Property $(\dagger \dagger)_{2g+2}$ on $\op{int}(N_T)$, while $\partial N_T$ is a negative Morse-Bott torus.

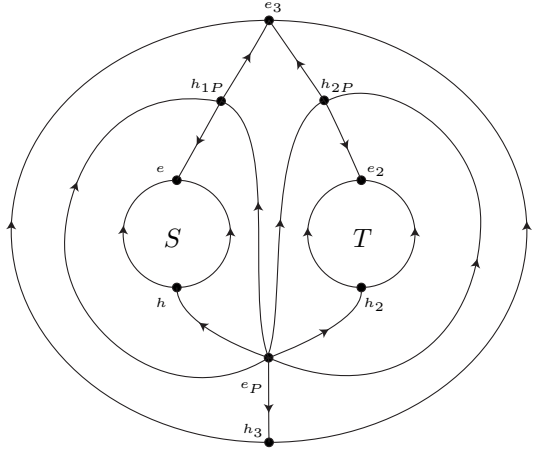
\begin{figure}
\centering
\begin{overpic}[width=7cm]{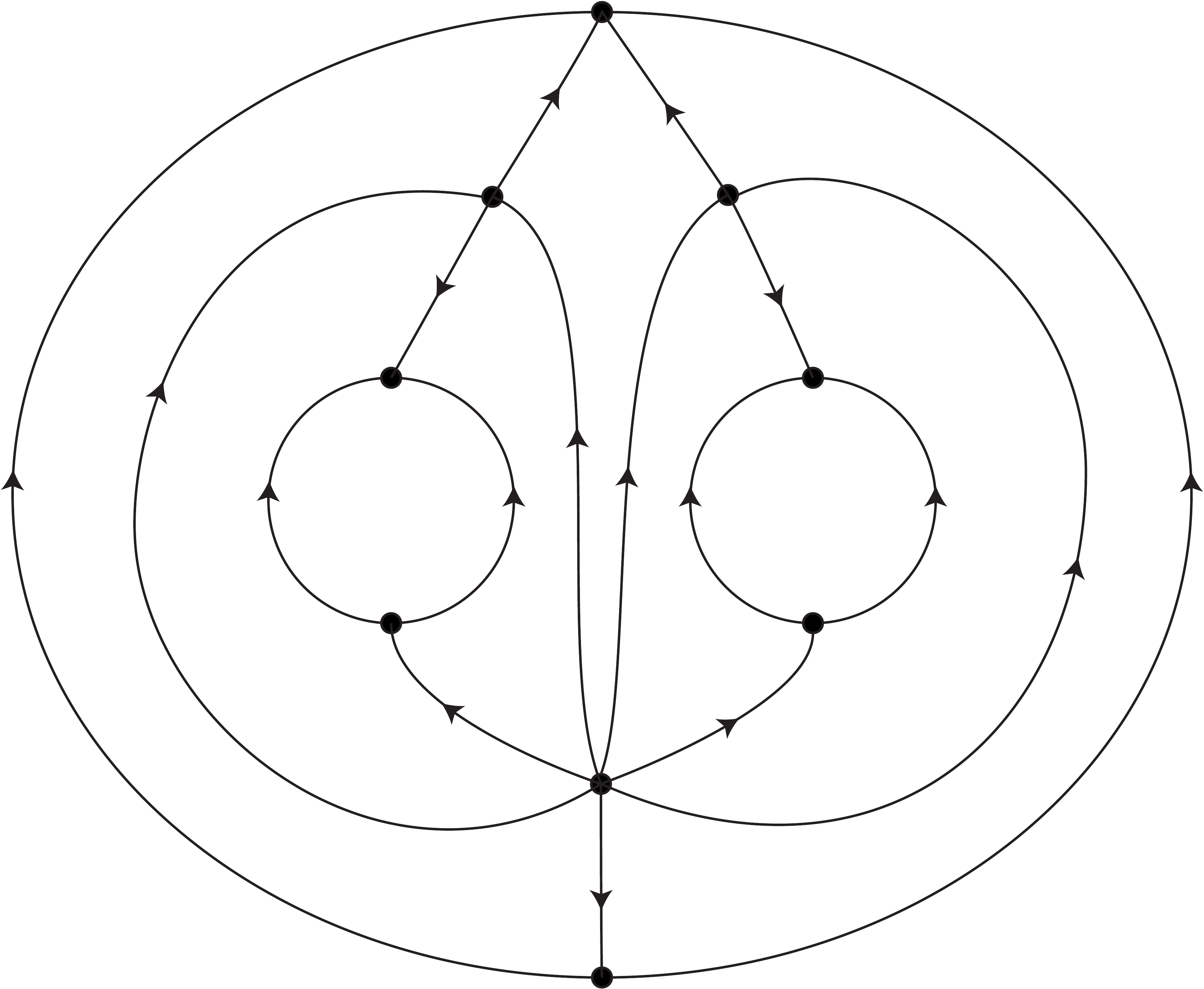}
\put(28.5,52.3){\tiny $e$} \put(68.5,52.3){\tiny $e_2$}
\put(49,83.5){\tiny $e_3$} \put(44.5,11.2){\tiny $e_P$}
\put(35,68.3){\tiny $h_{1P}$} \put(59.7,68.3){\tiny $h_{2P}$}
\put(45,2.8){\tiny $h_3$} \put(28.5,26.7){\tiny $h$}
\put(68,26.7){\tiny $h_2$} \put(30,38){$S$} \put(66,38){$T$}
\end{overpic}
\caption{The pair of pants $P$. The negative gradient trajectories of the Morse function $f$ are shown. The critical points $e$ and $h$ belong to the boundary component $\partial_1P$. The critical points $e_i$ and $h_i$, for $i=2,3$, belong to the boundary component $\partial_iP$.}
\label{fig: sprime}
\end{figure}

Next, we take a pair of pants $P$ with $\partial P = \partial_1 P \sqcup \partial_2 P \sqcup \partial_3 P$, we fix an area form on $P$ and choose a function $h \colon P \to \R$ which is sufficiently close to $1$ in the $C^\infty$ topology and such that its negative gradient flow is as shown in Figure \ref{fig: sprime}. We denote the time $1$ flow of $f$ by $\hh_P$ and the mapping torus of $(P, \hh_P)$ by $N_P$. Finally, we define $\widetilde{S} = P \cup S \cup T$, where $\partial_1P$ is glued to $\partial S$ and $\partial_2 P$ to $\partial T$. We define also $\widetilde{\hh} \colon \widetilde{S} \to \widetilde{S}$ such that $\widetilde{\hh}|_S= \hh$, $\widetilde{\hh}_T= \hh_T$ and $\widetilde{\hh}|_P= \hh_P$. Then $(\widetilde{S}, \widetilde{\hh})$ is an abstract open book decomposition of $M$ with page of genus $2g+2$ and we denote the mapping torus of $(\widetilde{S}, \widetilde{\hh})$ by $\widetilde{N}$.

Let $\{a_{2g+1}, a_{2g+2} \}$ be a basis of arcs for $T$ and,  for $i=2g+1$ or $2g+2$, let $\{ x_i, x_i'\}$ be the intersection points between $a_i$ and $\hh_T(a_i)$ on $\partial N_T$. On $\widetilde{S}$ we consider a basis of arcs $\widetilde{\mathbf{a}}= \{ \widetilde{a}_1, \ldots, \widetilde{a}_{2g+2} \}$ where $\widetilde{a}_i$, for $i=1, \ldots, 2g+2$, is obtained by extending $a_i$ up to $\partial \widetilde{S}= \partial_3 S$. The intersection points between $\widetilde{a}_i$ and $\widetilde{\hh}(\widetilde{a}_i)$ in $\widetilde{S}$ will be denoted by $\widetilde{x}_i$ and $\widetilde{x}_i'$. Observe that the boundary behavior of the monodromy $\hh_P$ forces one intersection point in $\op{int}(P)$ between $\widetilde{a}_i$ and $\widetilde{\hh}(\widetilde{a}_i)$ for each segment of $\widetilde{a}_i \setminus a_i$.

We define the stabilisation map
$$\mathfrak{S} \colon \widehat{CF}(S, \mathbf{a}, \hh(\mathbf{a})) \to  \widehat{CF}(\widetilde{S}, \widetilde{\mathbf{a}}, \widetilde{\hh}(\widetilde{\mathbf{a}})), \quad \mathfrak{S}(\mathbf{y})= \mathbf{y} \cup \{ \widetilde{x}_{2g+1}, \widetilde{x}_{2g+2} \}.$$
Similarly, we define the stabilisation map
$$\mathfrak{T} \colon PFC_{2g}(N) \to PFC_{2g+2}(\widetilde{N}), \quad \mathfrak{T}(\bs{\gamma})=e_3^2 \bs{\gamma}.$$
It is easy to see that $\mathfrak{S}$ and $\mathfrak{T}$ are chain maps.
\begin{lemma}\label{lemma: S isomorphism}
$\mathfrak{S}$ induces an isomorphism in homology.
\end{lemma}
\begin{proof}
One can check that
$$\widehat{HF}(\widetilde{S}, \widetilde{\mathbf{a}}, \widetilde{\hh}(\widetilde{\mathbf{a}})) \cong \widehat{HF}(S, \mathbf{a}, \hh(\mathbf{a})) \otimes \widehat{HF}(T, \mathbf{a}_T, \hh_T(\mathbf{a}_T))$$
and $\widehat{HF}(T, \mathbf{a}_T, \hh_T(\mathbf{a}_T))$ is generated by the class of $\{\widetilde{x}_{2g+1}, \widetilde{x}_{2g+2} \}$. Since $\mathbf{y} \cup \{ \widetilde{x}_{2g+1}, \widetilde{x}_{2g+2} \}$ is homologous to $\mathbf{y} \cup \{ x_{2g+1}, x_{2g+2} \}$ because of the intersection points between $\widetilde{a}_i$ and $\widetilde{\hh}(\widetilde{a}_i)$ contained in $\op{int}(P)$, the lemma follows.
\end{proof}
\begin{lemma}\label{lemma: T isomorphism}
$\mathfrak{T}$ induces an isomorphism in homology.
\end{lemma}
\begin{proof}
By an argument similar to the proof of Lemma \cite[Lemma 5.2.1]{II},
$$\Phi_{(\widetilde{S}, \widetilde{\hh})}(\mathbf{y} \cup \{ \widetilde{x}_{2g+1},  \widetilde{x}_{2g+1} \})= e_3^2 \Phi_{(S, \hh)}(\mathbf{y}),$$
and therefore there is a commutative diagram
$$\xymatrix{
\widehat{CF}(S, \mathbf{a}, \hh(\mathbf{a})) \ar[r]^{\mathfrak{S}} \ar[d]_{\Phi_{(S, \hh)}} & \widehat{CF}(\widetilde{S}, \widetilde{\mathbf{a}}, \widetilde{\hh}(\widetilde{\mathbf{a}})) \ar[d]^{\Phi_{(\widetilde{S}, \widetilde{\hh})}} \\
PFC_{2g}(N) \ar[r]^{\mathfrak{T}} & PFC_{2g+2}(\widetilde{N}).
}$$
Then Theorem \ref{Phi is an isomorphism} and Lemma \ref{lemma: S isomorphism} imply that $\mathfrak{S}$ induces an isomorphism in homology.
\end{proof}
Let $\iota \colon PFC_{2g+2}(N) \to PFC_{2g+2}(\widetilde{N})$ be the inclusion of complexes induced by the inclusion $N \subset \widetilde{N}$.
\begin{lemma}\label{lemma: iota is injective}
The inclusion $\iota$ is a chain map and induces an injective map $\iota_* \colon PFH_{2g+2}(N) \to PFH_{2g+2}(\widetilde{N})$.
\end{lemma}
\begin{proof}
Both claims are a consequence of \cite[Lemma 5.2.1]{II}, stating that each irreducible component of a holomorphic curve in $\R \times \widetilde{N}$ is contained in one of $\R \times N$, $\R \times N_T$ or $\R \times N_P$. This is a consequence of the blocking lemma \cite[Lemma 5.2.3]{binding}, the trapping lemma \cite[Lemma 5.3.2]{binding} and Property $(\dagger \dagger)_{2g+2}$, which help us control the behavior of the curves at the ends, and therefore their relative homology classes.
\end{proof}
\begin{proof}[Proof of Proposition \ref{prop: stabilisation}]
 We claim that the diagram
$$\xymatrix{
PFH_{2g}(N) \ar[r]^-{\mathfrak{T}_*} \ar[d]_{\mathfrak{j}_*} & PFH_{2g+2}(\widetilde{N}) \\ PFH_{2g+2}(N) \ar[ur]^{\iota_*}
}$$
commutes. In fact, by Morse-Bott theory, there are holomorphic cylinders from $h_{1P}$ to $e$ and $e_3$ coming from the unstable manifolds of $h_{1P}$ for the negative gradient flow of $f$ (see Figure \ref{fig: sprime} and \cite[Lemma 5.2.2]{II}) and therefore, if $\bs{\gamma}$ is a cycle in $PFC_{2g}(N)$, we have $\partial(h_{1P}(e+e_3) \bs{\gamma})=
e^2 \bs{\gamma} + e_3^2 \bs{\gamma}$.

The claim implies that $\iota_* \circ \mathfrak{j}_*$ is an isomorphism, and therefore $\iota_*$ is surjective. Thus it is invertible by Lemma \ref{lemma: iota is injective}, and this shows that $\mathfrak{j}_* = \iota_*^{-1} \circ \mathfrak{T}_*$ is invertible.
\end{proof}

In view of Lemma \ref{lemma: pretend i=infty} we can repeat the stabilisation argument arbitrarily many times and conclude that the natural map $PFH(N)_{2g} \to \widehat{PFH}(N, \partial N)$ is an isomorphism. This implies that $\widehat{\Phi}_*$ is an isomorphism and therefore finishes the proof of Theorem \ref{main}.

\end{document}